\documentclass[12pt,journal]{IEEEtran}
\onecolumn
\usepackage{mathrsfs}
\usepackage{algorithmic}
\usepackage{algorithm}
\usepackage{amsmath}
\usepackage{amsthm}
\usepackage[hidelinks]{hyperref}
\usepackage{graphicx}
\usepackage{amssymb}
\usepackage{epstopdf}
\usepackage{enumerate}
\usepackage{longtable,tabularx,float}
\usepackage{cite}
\usepackage{mathcomp}
\usepackage{multirow}
\usepackage{supertabular}
\usepackage{stmaryrd}
\usepackage{color}
\usepackage{url}
\usepackage{makecell}
\usepackage[OT2,OT1]{fontenc}
\usepackage{bm}
\usepackage{bbm}
\usepackage{caption}
\newtheorem{corollary}{Corollary}[section]
\newtheorem{definition}{Definition}[section]
\newtheorem{lemma}{Lemma}[section]

\newtheorem{theorem}{Theorem}[section]

\newtheorem{problem}{Problem}[section]

\newtheorem{construction}{Construction}[section]
\newtheorem{claim}{Claim}[section]
\newtheorem{fact}{Fact}[section]

\newtheorem{conjecture}{Conjecture}[section]

\usepackage{xcolor}
\definecolor{cream}{RGB}{203, 237, 204}

\begin{document}

\title{The Hilton-Milner type results of $(k, \ell)$-sum-free sets in $\mathbb F_p^n$}

\author{Xin~Wei, Xiande~Zhang, and Gennian~Ge
\thanks{This project was supported by the National Key Research and Development Program of China under Grant 2025YFC3409900 and Grant 2023YFA1010200, the National Natural Science Foundation of China under Grant 12171452 and Grant 12231014,  Beijing Scholars Program, and Quantum Science and Technology-National Science and Technology Major Project 2021ZD0302902. This work was also funded by the Excellent Doctoral Students Study Abroad Support Program of the University of Science and Technology of China.}
 \thanks{X. Wei ({\tt weixinma@mail.ustc.edu.cn}) is with the School of Mathematical Sciences, University of Science and Technology of China, Hefei, 230026, Anhui, China.}
 \thanks{X. Zhang ({\tt drzhangx@ustc.edu.cn}) is with the School of Mathematical Sciences, University of Science and Technology of China, Hefei, 230026, Anhui, China, and with Hefei National Laboratory, University   of   Science   and   Technology   of   China,   Hefei, 230088, China.}
 \thanks{G. Ge ({\tt gnge@zju.edu.cn}) is with the School of Mathematical Sciences, Capital Normal University, Beijing, 100048, China.}
 \thanks{This work was completed while the first author was at the National University of Singapore, hosted by Professor Ka Hin Leung.}

}
\maketitle

\begin{abstract}
For a prime $p\equiv 2\pmod 3$, it is well known that the largest sum-free subsets of $\mathbb F_p^n$ have size $\frac{p+1}3p^{n-1}$, and the extremal sets must be a cuboid of the form $\{\frac{p+1}3,\frac{p+1}3+1,\ldots, \frac{2p-1}3\}\times\mathbb F_p^{n-1}$ up to isomorphism. Recently, Reiner and Zotova proved a Hilton-Milner type stability result showing that for large $p$,
 any sum-free set not contained in the  extremal cuboid  has size at most $\frac{p-2}3p^{n-1}$, and all possible structures attaining this bound were classified.


  In this paper, we develop a general Hilton-Milner theory for $(k, \ell)$-sum-free sets in $\mathbb F_p^n$ for $k>\ell\geq 1$. We determine the maximum size of such sets for all $p\equiv \mu\pmod{k+\ell}$ with $2\leq \mu\leq k+\ell-1$, 
  and show that the  extremal configurations are precisely $\lceil\frac{\mu-1}{2}\rceil$ non-isomorphic cuboids. Beyond the extremal regime, we prove sharp Hilton–Milner type stability results showing that, for all sufficiently large $p$, a $(k, \ell)$-sum-free set not contained in any of these extremal cuboids is uniformly bounded away from the maximum by a gap $p^{n-1}$, and we determine the full structure of all sets achieving this second-best bound in several broad parameter ranges. In particular, when $2\leq \mu\leq k+\ell-3$ 
   (which is tight),  only two  structural types occur for all $k+\ell\geq 5$; and when $\mu=2$ or $3$, we obtain a complete classification for all $k>\ell\geq 1$.

  Our arguments combine additive combinatorics and Fourier-analytic methods, and make use of recent progress toward the long-standing $3k-4$
 conjecture, highlighting new connections between inverse additive number theory and extremal problems over finite vector spaces.

%


\end{abstract}

\begin{IEEEkeywords}
\boldmath $(k, \ell)$-sum-free sets, Hilton-Milner type theorem, arithmetic progressions, additive combinatorics.
\end{IEEEkeywords}

\section{Introduction}

Let $G$ be a finite abelian group. A subset $A$ of $G$ is called  \emph{sum-free} if the equation $x+y=z$ has no solutions with $x,y,z\in A$. Equivalently, if we take the standard notation $A+B:=\{x+y:x\in A,y\in B\}$ for subsets $A$ and $B$ in $G$, then $A$ is sum-free if and only if $A\cap (A+A)=\emptyset$. How large can a sum-free subset be in a finite abelian group? Although the first work dates back to 1969 by Yap~\cite{Yap1969},  the  question was only completely resolved by Green and Ruzsa in 2005 for all finite abelian groups~\cite{Green2005}. For sum-free sets over integers, research goes back to the pioneering work of Schur~\cite{Schur1916} when attacking the Fermat's last theorem, and it remains at the forefront of current research (see~\cite{Bourgain1997,Eberhard2014,Tran2018, Lev2024}).



When $G$ is a vector space over finite fields, sum-free sets have connections to
projective geometry and coding theory. Let $G=\mathbb F_p^n$ with $n\ge 1$ and prime $p\equiv 2\pmod 3$,  Yap~\cite{Yap1969} described the unique structure of a maximum sum-free subset up to additive group isomorphisms as below.  Here, two subsets of $\mathbb F_p^n$ are isomorphic to each other if there is an (additive) isomorphism of $\mathbb F_p^n$ sending one to the other, i.e.  considering $\mathbb F_p^n$ as an additive group.


\begin{theorem}[~\cite{Yap1969}]~\label{thm_sum_free_cuboid}
Let $p=3m+2$ be a prime number with $m\geq 1$. If $A\subset\mathbb F_p^n$ is a sum-free set, then $|A|\le (m+1) p^{n-1}$. Moreover, if $|A|= (m+1) p^{n-1}$, $A$ is 
isomorphic to the following unique form,
$$A_{2, 1}\triangleq\left[m+1, 2m+1\right]\times \mathbb F_p^{n-1}.$$
Here, the notation $[a,b]$ with $a,b\in \mathbb F_p$ means the set $\{a,a+1,\ldots,b\}$ in $\mathbb F_p$.
\end{theorem}
The subset $[a,b]\subset \mathbb F_p$ is called an \emph{interval} in $\mathbb F_p$. A subset of $\mathbb F_p^n$  formed by an interval multiplied by a subspace of dimension $n-1$ is called a \emph{cuboid}.  Theorem~\ref{thm_sum_free_cuboid} shows that, the structure of a sum-free set of maximum size must be a cuboid with the unique parameters.
Once the extremal structure is known, it is natural to ask the corresponding
stability problem: what is the structure of sum-free sets in $\mathbb F_p^{n}$ of
size close to the maximum size?


Recently Reiher and Zotova~\cite{Reiher2024} proved a Hilton-Milner type result of Theorem~\ref{thm_sum_free_cuboid}, which is a strong stability result. They showed that if a sum-free set is \emph{nontrivial}, i.e., not contained in the extremal cuboid, then there is a large gap between its size and the maximum size, which means a sum-free set of size close to the maximum size must be contained in the extremal cuboid.  They also characterized the possible structures for the nontrivial optimal sum-free sets.

\begin{theorem}[\cite{Reiher2024}]\label{thm_sum_free_0}
Let $p=3m+2$ be a prime number with $m\geq 1$.
If a sum-free set $A\subset \mathbb F_p^n$ is not isomorphic to a subset of $A_{2, 1}$, then $|A|\le mp^{n-1}$.

Moreover, if $|A|= mp^{n-1}$, the structures of $A$ are determined up to isomorphism. More specifically, there exists an integer $s\in [n-1]$, and a subset $P \subset \mathbb F_p^{s}$ satisfying $0\notin P+P$ ($P$ can be empty), such that $A=A_0\times \mathbb F_p^{n-1-s}$, where $A_0\subset \mathbb F_p\times \mathbb F_p^{s}$ is defined by
$$A_0=\{(m, 0)\}\sqcup \left(\{m+1\}\times(\mathbb F_p^{s}\backslash P)\right)\sqcup \left([m+2, 2m-1]\times\mathbb F_p^{s} \right) \sqcup \left(\{2m\}\times(\mathbb F_p^{s}\backslash\{0\})\right) \sqcup \left(\{2m+1\}\times P \right).$$
Here the symbol ``$\sqcup$'' means disjoint unions. 

\end{theorem}

The classical Hilton-Milner theorem~\cite{Hilton1967} is a strong stability result for the famous Erd\H{o}s-Ko-Rado theorem~\cite{Erdos1961}. There has been a surge in Hilton-Milner type stability results for various extremal configurations,
including cross-intersecting families~\cite{Borg2013}, intersecting chains~\cite{Erdoes2000}, and generalized Tur\'{a}n problems~\cite{Zhao2025}.

In this paper, we generalize Theorem~\ref{thm_sum_free_cuboid} and Theorem~\ref{thm_sum_free_0} from sum-free sets to $(k, \ell)$-sum-free sets. For any positive integer $h\geq 2$,
 by denoting $2A:=A+A$, the set of all \emph{$h$-term sums} is defined by $hA:=(h-1)A+A$. This notation should be distinguished from $h\cdot A$, which means $\{ha: a\in A\}$ in our paper. Given positive integers $k>\ell\ge 1$, we say a nonempty set $A$ is \emph{$(k, \ell)$-sum-free} if $kA\cap\ell A=\emptyset$. Equivalently, this is to say, the equation
$$x_1+x_2+\cdots+ x_k= x_1'+\cdots +x_\ell'$$
has no solution with all unknowns $x_1, \ldots, x_k, x_1', \ldots,x_\ell'$ in $A$.
The maximum size of $(k, \ell)$-sum-free sets over abelian groups or positive integers has also been widely studied, see for example \cite{Bajnok2009, Eberhard2015, Jing2021, Jing2024}.
Further, many cornerstone results in additive combinatorics study structures or relationships of $h$-term sums for varying $h$. For example, Freiman-Ruzsa theorem and the
recently solved
polynomial Freiman-Ruzsa conjecture~\cite{Schoen2011, Green2007, Sanders2013, Gowers2025} describe the structure of sets with small $h$-term sum sets; Roth's theorem determines the maximum density of subsets of integers avoiding nontrivial $3$-term arithmetic progressions~\cite{Roth1953, Bourgain1999, Kelley2023}; and the sum-product phenomenon concerns the non-coexistence of small sum sets and small product sets within certain rings~\cite{Erdos1983,Tao2008, Bourgain2009, Tang2023,o2025sum}.

\subsection{Main results}
In the following, we always assume that $k>\ell\ge 1$ and $n, m\ge 1$ unless otherwise stated.
Analogous to Theorem~\ref{thm_sum_free_cuboid}, we first establish the following result.
Note that the upper bound in Theorem~\ref{thm_k_l_max_structure} can also be deduced from \cite[Theorem 5]{Bajnok2009}. 

\begin{theorem}\label{thm_k_l_max_structure}
Let  $p=(k+\ell)m+2+\lambda$ be a prime number with $\lambda\in [0, k+\ell-3]$.  If $A\subset \mathbb F_p^n$ is a $(k, \ell)$-sum-free subset, then $|A|\le (m+1)p^{n-1}$.

Moreover, there are in total $\lceil\frac{\lambda+1}2\rceil$ mutually non-isomorphic structures of $A$ with $|A|= (m+1)p^{n-1}$. More specifically, for any $j\in [0, \lceil\frac{\lambda+1}2\rceil-1]\subset \mathbb F_p$, denote $a_j:=-(km+1+j)(k-\ell)^{-1}$. Then the $j$-th structure of $A$ is a cuboid
$$A_{k, \ell, p, j}\triangleq[a_j, a_j+m]\times \mathbb F_p^{n-1}.$$
For convenience, we name them \emph{extremal cuboids}.
Especially, if $\lambda=0, 1$, the extremal cuboid is unique up to isomorphism.
\end{theorem}

If $A$ is isomorphic to a subset of an extremal cuboid,  $A$ is said to be \emph{trivial}; otherwise, $A$ is \emph{nontrivial}.
Next, we state our main results Theorems~\ref{thm_with_fourier}-\ref{thm_3_1}, which are the corresponding Hilton-Milner type results of Theorem~\ref{thm_k_l_max_structure}.
All of them state that when $p$ is large enough (depending on $k$ and $\ell$), if $A$ is a nontrivial $(k, \ell)$-sum-free subset, then $|A|\le mp^{n-1}$, and if the equality holds, the structures of $A$ can be fully characterized. To lower bound $p$, we need  a function $T$ over the rational field whose definition is more involved. For a smoother reading experience, at this moment,  we only mention that $T$ is related to the Freiman's $(3k-4)$ conjecture, for which the formal definition will be introduced in Subsection~\ref{subsec:cp}.


To characterize the structures of $A$, we define five types of subsets of size $mp^{n-1}$. The following is a collection of two types of structures that are commonly used in this paper.

\begin{definition}\label{cons_12}
Let $p=(k+\ell)m+2+\lambda$ be a prime number. Let $A\subset\mathbb F_p^n$ with $|A|=mp^{n-1}$. 
\begin{itemize}
  \item[(1)]$A$ is said to be of \emph{type 1} if $A$ is isomorphic to $[a, a+m-1]\times\mathbb F_p^{n-1}$ for some $a$ in $\mathbb F_p$ satisfying that $\ell a-k(a+m-1)$ belongs to $[1,\ell]$ or $[\lambda+\ell+2, k]$. The latter case is valid only when $\lambda+\ell+2\le k$.
  \item[(2)] \emph{(When $\ell=1$).} $A$ is said to be of \emph{type 2} if $\ell=1$, $n\geq 2$
and $A$ is isomorphic to
$$\left(\{a\}\times\{\mathbb F_p^{n-1}\backslash V\}\right)\sqcup \left([a+1, a+m-1]\times \mathbb F_p^{n-1}\right)\sqcup \left(\{a+m\}\times V\right),$$
where $a=mk(\ell-k)^{-1}\in \mathbb F_p$ is fixed, and $V\subset \mathbb F_p^{n-1}$ is any proper subspace.
\end{itemize}

\end{definition}

If $A$ is of type 1 or 2,  $A$ is said to be \emph{normal}.
It can be proved that a normal $A$ is a nontrivial $(k, \ell)$-sum-free subset (see Section~\ref{sec:uniq}).  Our first result shows the other way, that is,  when $k+\ell$ and $p$ is large, any nontrivial $(k, \ell)$-sum-free subsets of size at least $mp^{n-1}$ must be normal.

\begin{theorem}[When $\lambda\in \lbrack 0, k+\ell-5 \rbrack$]~\label{thm_with_fourier}
Suppose $k+\ell\geq 5$. Let $p=(k+\ell)m+2+\lambda$ be a prime number with  $\lambda\in [0, k+\ell-5]$ and $m\ge \max\{ 5T(\frac1{k+\ell}), (k+\ell+\lambda+4)(1+\frac1{k+\ell})(k+\ell)^{1\slash(k+\ell-2)}\}$. If $A\subset \mathbb F_p^n$ is a nontrivial $(k, \ell)$-sum-free set of size at least $mp^{n-1}$, then $A$ is normal up to isomorphism.


\end{theorem}

 If $A$ is neither trivial nor normal,  $A$ is said to be \emph{abnormal}.
We remark that our requirement $\lambda\leq k+\ell-5$ 
 in Theorem~\ref{thm_with_fourier} is optimal. In fact, if
 $\lambda= k+\ell-4$, 
 there is always an abnormal example of $A$ as in the following definition.

\begin{definition}[When $\lambda= k+\ell-4$]~\label{cons_3}
Suppose $k+\ell\geq 5$. Let $p=(k+\ell)m+2+\lambda$ be a prime number with $\lambda= k+\ell-4$. Let $A\subset\mathbb F_p^n$ with $|A|=mp^{n-1}$. Choose $a=(\ell m+k-1)(k-\ell)^{-1}\in \mathbb F_p$.
Then $A$ is said to be of \emph{type 3} if $A$ is isomorphic to
$$([a-1, a+m]\backslash\{a, a+m-1\})\times\mathbb F_p^{n-1}.$$
\end{definition}

It can be proved that a set  of type $3$ is an abnormal $(k, \ell)$-sum-free subset (see Section~\ref{sec:uniq}).

Remember that when $\lambda=0$ or $1$, the extremal cuboid is unique up to isomorphism by Theorem~\ref{thm_k_l_max_structure}. As our second result, we give a Hilton-Milner type result for all pairs $(k, \ell)$ when $\lambda=0, 1$ and determine the extremal structures of nontrivial $(k, \ell)$-sum-free subsets for large $p$. Since the cases for $(k,\ell)=(2,1)$ and $k+\ell\ge \lambda+5$ have been settled in Theorem~\ref{thm_sum_free_0} and Theorem~\ref{thm_with_fourier}, respectively, we only need to consider the cases when $(k+\ell,\lambda)=(5,1)$ or $(4,1)$.
 For both cases,  our results conclude that $|A|\le mp^{n-1}$ as before but when $|A|=mp^{n-1}$ they are not always normal. We split them into two cases by defining two more types of abnormal sets.


\begin{definition}[When $(k+\ell,\lambda)=(5,1)$]~\label{cons_4}
Let $p=5m+3$ be a prime number. Let $A\subset\mathbb F_p^n$ with $|A|=mp^{n-1}$ and $n\ge 2$. Then $A$ is said to be of \emph{type 4} if there exists a proper
subspace $V\subset \mathbb F_p^{n-1}$ such that $A$ is isomorphic to
$$\left(\{2m+1, 3m+2\}\times V \right)\sqcup \left(\{2m+2, 3m+1\}\times(\mathbb F_p^{n-1}\backslash V)\right)\sqcup \left([2m+3, 3m]\times \mathbb F_p^{n-1}\right).$$

\end{definition}

Note that if  $V$ is allowed to be $ \mathbb F_p^{n-1}$ in Definition~\ref{cons_4}, then  type 4 goes to type 3. To avoid this overlap, we let $V$ in Definition~\ref{cons_4} be proper. The arguments of reaching type 3 and type 4 are different in our proofs (See Section~\ref{sbsc:51}). We summarize the results for $(k+\ell,\lambda)=(5,1)$ as follows.

\begin{theorem}[When $(k+\ell,\lambda)=(5,1)$]~\label{thm_(5, 1)_unnormal}
Let $p=5m+3$ be a prime number with $m\ge \max\{  5 T(1\slash 5), 21\}$. Let $(k, \ell)=(4, 1)$ or $(3, 2)$.
If $A\subset\mathbb F_p^n$ is a nontrivial $(k, \ell)$-sum-free set with $|A|\ge mp^{n-1}$, then $A$ is either normal or of type 3 or 4.
%
\end{theorem}

When $(k+\ell,\lambda)=(4,1)$, we define another type of abnormal subsets.

\begin{definition}[When $(k+\ell,\lambda)=(4,1)$]~\label{cons_5}
Let $p=4m+3$ be a prime number. Let $A\subset\mathbb F_p^n$ with $|A|=mp^{n-1}$. Then $A$ is said to be of \emph{type 5} if there exists an integer $s\in [0,n-1]$, and a nonempty subset $P \subset \mathbb F_p^{s}$ satisfying $0\notin 3P$, such that $A=A_0\times \mathbb F_p^{n-1-s}$, where $A_0\subset \mathbb F_p\times \mathbb F_p^{s}$ is defined by
\begin{align*}
  A_0=&\{(a-1, 0)\}\sqcup \left(\{a\}\times\{\mathbb F_p^{s}\backslash P\} \right)\sqcup \left([a+1, a+m-2]\times\mathbb F_p^{s}\right) \\
  &\sqcup \left(\{a+m-1\}\times(\mathbb F_p^{s}\backslash\{0\})\right) \sqcup \left(\{a+m\}\times P\right),
\end{align*}
where $a=(m+2)\slash2\in \mathbb F_p$.

%
\end{definition}

Note that if we allow $P=\emptyset$ in Definition~\ref{cons_5}, then type 5 goes to type 2 when $s\ge 1$, and goes to type $1$ when $s=0$. So we set $P\neq\emptyset$  in Definition~\ref{cons_5} to avoid the overlap. This structure is similar to that in Theorem~\ref{thm_sum_free_0}. We have the following corresponding theorem.

\begin{theorem}[When $(k+\ell,\lambda)=(4,1)$]~\label{thm_3_1}
Let $p=4m+3$ be a prime number with $m\ge \max\{  5 T(1\slash 4), 23\}$. 
If $A\subset\mathbb F_p^n$ is a nontrivial $(3, 1)$-sum-free set with $|A|\ge mp^{n-1}$, then $A$ is either normal or of type 5.

%
\end{theorem}

\subsection{Function $T$ and $3k-4$ conjecture}\label{subsec:cp}
Now we define the function $T$ in the lower bounds of $m$ in Theorems~\ref{thm_with_fourier}-\ref{thm_3_1}.
Before that,
 we need to introduce some results about the $3k-4$ conjecture, which studies under what condition an additive set can be covered by a short arithmetic progression. If an additive set $A$ can be covered by an arithmetic progression of length $|2A|-|A|+1$, we say $A$ has the {\it covering property}. In the integer ring Freiman proved the following $3k-4$ theorem, see  \cite{Freiman1964} and \cite[Theorem 7.1]{Grynkiewicz2013}.
\begin{theorem}[Freiman $3k-4$ Theorem]~\label{thm_3k-4_origin}
Let $A$ be a finite set of integers satisfying $|2A| \le 3|A| - 4$. Then $A$ has the covering property.
\end{theorem}
However, the analog of Theorem~\ref{thm_3k-4_origin} over $\mathbb F_p$ seems much harder, which leads to the following famous $3k-4$ conjecture, see for example~\cite{Roedseth2006}. For behaviors of arithmetic progressions in $\mathbb F_p$, see Section~\ref{sbsc:cd} for more explanations.
\begin{conjecture}[$3k-4$ conjecture]~\label{conj_3k-4_origin}
If $p$ is a large prime number and $A\subset \mathbb F_p$ satisfies $|A+A|\le\min\{ 3|A|-4, p-1\}$, then $A$ has the covering property.
\end{conjecture}

Although many efforts have been taken during the decades~\cite{Freiman1961,Serra2000, Roedseth2006, Serra2009, Candela2019}, the  $3k-4$ conjecture is still open. Significant progress does emerge, though, when considering a ``sharper'' $A$. There are mainly two ways to explain the sharpness. On the one hand, one can require a smaller extension ratio of $A$, which means $|A+A|\slash|A|$ is smaller; on the other hand, one can require $A$ to  be a smaller portion in $\mathbb F_p$, which means the density $|A|\slash p$ is limited. The original $3k-4$ conjecture says that the covering property holds for any density if the extension ratio is roughly $3$.
Known results reveal a trade-off between the extension ratio and the density. See the followings.


\begin{theorem}[\cite{Grynkiewicz2024}, Theorem 1.4]\label{thm_new_vosper}
Let $A, B\subset \mathbb F_p$ be nonempty subsets satisfying $A+B\ne \mathbb F_p$. Set $C=-\mathbb F_p\backslash(A+B)$. If we have $|A+B|=|A|+|B|+r$ for some $r$ satisfying \[|A+B|\le\min\{|A|+1.0527|B|-3, p-9(r+3)\},\]
there are arithmetic progressions $P_A, P_B, P_C\subset \mathbb F_p$ of the same common difference such that for all $X\in\{A, B, C\}$, $X\subset P_X$ and $|P_X|\le |X|+r+1$.
\end{theorem}

Setting $B=A$ in Theorem~\ref{thm_new_vosper}, we know that the covering property  holds for any density if the extension ratio $|A+A|\slash |A|$ is roughly at most 2.0527.

\begin{theorem}[\cite{Roedseth2006}]\label{thm_new_freiman_2.4}
Let $A \subset \mathbb F_p$ with $|A+A|\le 2.4|A|-3$. If further $|A|\le p\slash10.7$, then $A$ has the covering property.
\end{theorem}
By Theorem~\ref{thm_new_freiman_2.4}, if we reduce the density $|A|\slash p$ to $1\slash 10.7$, we only need  the extension ratio $|A+A|\slash |A|$ to be at most 2.4.
Based on this observation, we introduce the function $T$ as follows.

\begin{definition}\label{deftau}
For a constant $c\in (0,1)$, we say that $\tau\in (0, 1]$ is \emph{feasible for $c$} if any
$A \subseteq \mathbb F_p$ with $|A|\le cp$ has the covering property provided that $|A+A|\le (2+\tau)|A|-3$.
  Define the supremum of feasible $\tau$ for $c$ as $\tau(c)$, and denote $T(c)=\frac 1{\tau(c)}$.
\end{definition}

By Definition~\ref{deftau}, for any $0\le c\le c'\le 1$, we must have $\tau(c)\ge \tau(c')$. From Theorem~\ref{thm_new_vosper} we know that for any $0\le c<1$, $\tau(c)\ge 0.0527$ and hence $T(c)< 20$. From Theorem~\ref{thm_new_freiman_2.4} we know that $T(\frac1{10.7})\le 2.5$. When $c$ goes down, $T(c)$ does not increase. So in Theorems~\ref{thm_with_fourier}-\ref{thm_3_1}, we can replace all the values of $T$ by $20$ if we are not committed to finding the best possible lower bound of $m$.

\subsection{Main idea and organization}
Our idea of proving Theorems~\ref{thm_with_fourier}-\ref{thm_3_1} is motivated from \cite{Reiher2024} which uses a projection method. Given a pair $(v, K)$ such that $K\subset\mathbb F_p^n$ is a subspace of dimension $n-1$ and $v\in \mathbb F_p^n\backslash K$, we can naturally write any element in $\mathbb F_p^n$ as $x=iv+x'$ with $i\in \mathbb F_p$ and $x'\in K$. For any set $A\subset \mathbb F_p^n$, writing each element in this form, we can partition $A$ into $p$ parts based on different $i\in \mathbb F_p$, and project each part into a subset of $K$. 
 First, we show that if $A$ is $(k, \ell)$-sum-free and a little ``unbalanced'' based on $(v, K)$, that is, most parts of $A$ in this partition are empty except for a few parts, then by some additive combinatorial analysis,  $A$ is forced to be extremely unbalanced, which means almost all parts of $A$ are either very close to $K$ or empty. Then considering the set of those few indices $i$ with nonempty parts (i.e. \emph{support} of $A$), the problem is reduced to consider structures of   $(k, \ell)$-sum-free subsets of the $1$-dimensional space $\mathbb F_p$.
By Fourier analysis we show that, if $A$ is $(k, \ell)$-sum-free, the spectral radius of $A$ is big, leading to that for at least one choice of $(v, K)$, the corresponding distribution of $A$ is unbalanced.


The main contribution of this paper is mainly two-folded. First, we give a combining method of additive combinatorial analysis and Fourier analysis for the structures of large $(k, \ell)$-sum-free sets in $\mathbb F_p^n$, which yield Hilton-Milner type results for a wide range of parameters. We believe that this method can be extended to incorporate more parameters that have not yet been considered. Second, we find a new application of the $3k-4$ conjecture in bounding the field size in our main results. Further improvements on this conjecture can give better lower bound of $p$ applicable to our results.


The article is organized as follows. In Section~\ref{sec_prelim} we give some notations and preliminary results frequently used in this paper. The proofs of Theorem~\ref{thm_k_l_max_structure} and the non-overlapping among the five types are also included.
Section~\ref{sec_add_struc} is devoted to analyze the structures of the support of a large $(k, \ell)$-sum-free subset by using additive combinatorial tools. Theorem~\ref{thm_with_fourier} is proved in Section~\ref{sec_proof_of_main_theorem}, while the cases when $A$ is of small weight and large weight are treated differently, and the latter one applies Fourier analysis.
By following the same proving process but with extra criterion and techniques, Theorems~\ref{thm_(5, 1)_unnormal} and \ref{thm_3_1} are proved in Section~\ref{sec_unique_extrem}. Section~\ref{sec_conc} concludes our paper and presents several open problems.



\section{Preliminaries}\label{sec_prelim}

We begin with some basic tools widely used in additive combinatorics. Given a finite nonempty subset $A$ of an abelian group $G$, we write $Sym(A)\triangleq\{g\in G: \{g\}+A=A\}$  for the \emph{symmetric group} of $A$. Then $A$ is a union of some cosets of $Sym(A)$ and thus $A=A+Sym(A)$. We call two sets of $G$ {\it isomorphic} if there is an automorphism of $G$ sending one to the other.

\begin{theorem}[Kneser~\cite{Kneser1953, Kneser1954}]\label{thm_Kneser}
If $A_1, \ldots, A_k$ are finite nonempty subsets of an abelian group $G$, and $H=Sym(A_1+\cdots + A_k)$, then
\[|A_1+\cdots+ A_k|\ge |A_1+ H|+\cdots +|A_k+ H|-(k-1)|H|.\]
Consequently, $|A_1|+\cdots +|A_k|\leq |A_1+\cdots+ A_k|+(k-1)|H|$.
\end{theorem}

When  $G=\mathbb F_p$ and $k=2$, instead we have

\begin{theorem}[Cauchy-Davenport \cite{Cauchy1813,Davenport1935}]\label{thm_Cauchy_daven}
If $p$ is a prime number and $A, B\subset \mathbb F_p$ are nonempty, then
\[|A+B|\ge \min\{p, |A|+|B|-1\}.\]
\end{theorem}

The extremal case of Theorem~\ref{thm_Cauchy_daven} was completely determined by \cite{vosper1956critical}. The following characterization is the main case when $p$ is big.
\begin{theorem}[Vosper~\cite{vosper1956critical}]\label{thm_Vosper}
Let $p$ be a prime and $A, B\subset \mathbb F_p$ satisfying $|A|, |B|\ge 2$ and $|A+B|\le p-2$. Then $|A+B|=|A|+|B|-1$ if and only if $A$ and $B$ are \emph{representable} as arithmetic progressions (APs) in $\mathbb F_p$ with the same common difference.
\end{theorem}


\subsection{Arithmetic progressions in $\mathbb F_p$}\label{sbsc:cd}
In Theorem~\ref{thm_Vosper}, we use ``representable'' to mean that we can arrange all elements of the set such that any two consecutive elements have a common difference.
 For APs over $\mathbb Z$, the common difference is {\it unique up to sign}, meaning that if $A$
 is an AP with common difference $d$, then the only other possible common  difference is $-d$ by reversing the order of elements in $A$. However,  APs in $\mathbb F_p$ can wrap around. If $A\subset \mathbb F_p$ is representable by an AP, we may wonder whether the common difference is unique. It seems that there is no conclusion about this question in the literature. A counterexample is that  $A=\mathbb F_p\backslash \{x\}$ with any fixed point $x$ can be representable by an AP of length $p-1$ with any common difference in $\mathbb F_p\backslash\{0\}$. We will show that besides this example, the common difference of an AP in $\mathbb F_p$ is unique up to sign, see Lemma~\ref{lem_interval_distinct}. Before proving Lemma~\ref{lem_interval_distinct}, we introduce the following useful notations.

For integers $a\leq b$, we use the same notation $[a,b]$ to denote the set $\{a,a+1,\ldots,b\}$ and write $[b]=[1,b]$ for short.
Let $A$ be a subset of $\mathbb F_p$ with $ |A|\in [2, p-2]$. For any $d\in [(p-1)\slash 2]$, define $E_d(A)=\{(x, y)\in A\times\mathbb (\mathbb F_p\setminus A): |x-y|=d\}$ and $e_d(A):=|E_d(A)|$, that is, the number of pairs of difference $d$ with exactly one element in $A$. Given $A$, let $\mathcal E(A)$ denote the multiset by collecting all $e_d(A)$s for $d\in [(p-1)\slash 2]$. We have the following two facts:
\begin{itemize}
  \item $A$ is an AP of common difference $d$ if and only if $e_d(A)=2$;
  \item  If $A, A'\subset \mathbb F_p$ are isomorphic, then $\mathcal E(A)= \mathcal E(A')$.
\end{itemize}

\begin{lemma}\label{lem_interval_distinct}
Let $p$ be a prime and let $A\subset \mathbb F_p$ with $ |A|\in [2, p-2]$ be an AP in $\mathbb F_p$. Then the common difference of $A$ is unique up to sign.
\end{lemma}
\begin{proof}
We prove by induction on the size of $A$. When $|A|=2$, it is trivially true. Suppose  the statement is true for all $|A|\le k$ with some $k\ge 2$. Now consider the case when  $|A|=k+1 \le p-2$.

Suppose   $A$ is representable by an AP  with common difference $d_1$. 
By considering $d_1^{-1}\cdot A$, we can assume that $d_1=1$ without loss of generality and hence $A$ is an interval, say $[x, x+k]$. Our goal is to show that, if $A$ can be written as $\{y, y+d_2, \ldots, y+kd_2\}$ for some $y\in \mathbb F_{p}$, we must have $d_2=\pm 1$.

If $p\le 2k+1$, we know that $\{y+(k+1)d_2, y+(k+2)d_2, \ldots, y+(p-1)d_2\}=\mathbb F_p\backslash A= [x+k+1, x-1]$, which is an AP of length between $2$ and $k$. From induction hypothesis, $d_2=\pm 1$.

If  $p\ge 2k+3$, $|A|=k+1< p/2$. It suffices to show that $e_d(A)>2$ for any  $d\in [2,(p-1)\slash 2]$.
As $A=[x, x+k]$, 
we have the four pairs $(x+d-2, x-2), (x+d-1, x-1), (x+k-d+1, x+k+1), (x+k-d+2, x+k+2)\in E_d(A)$ when  $d\in [2, k+1]$, and the four pairs $(x, x+d), (x+1, x+d+1), (x+k, x+k-d), (x+k-1, x+k-d-1)\in  E_d(A)$ when  $d\in [k+2, (p-1)\slash 2]$. Hence  $e_d(A)\ge 4$ for any  $d\in [2,(p-1)\slash 2]$. This completes our proof.
\end{proof}

Next we show that deleting one point from an interval generally does not result in an AP.


\begin{lemma}\label{lem_one_hole_not_AP}
Let $I=[a,b]$ be an interval in $\mathbb F_p$ with $4\leq |I|\le p-3$. Then for any $x\in I\setminus \{a,b\}$, $I\setminus \{x\}$ is not an AP.
\end{lemma}
\begin{proof} Let $A=I\setminus \{x\}$.
We only need to prove that $e_d(A)> 2$ for any $d\in [(p-1)\slash2]$. It is clear that {from $|I|\le p-3$,} $e_1(I)=2$, $e_2(I)=4$, and $e_{d}(I)\ge 6$ for any $d\ge 3$. The difference between $E_d(I)$ and $E_d(A)$ are the pairs of length $d$ containing $x$. For $d=1$, $E_1(A)=E_1(I)\cup\{(x, x-1), (x, x+1)\}$ and hence $e_1(A)=4$. For $d=2$, as $|I|\ge 4$, at most one of the pairs $(x, x-2)$ and $(x, x+2)$ is in $E_2(I)\setminus E_2(A)$,   hence $e_2(A)\ge e_2(I)-1\ge 3$. For $d\ge 3$,  at most two pairs $(x, x\pm d)$ are in $E_d(I)\setminus E_d(A)$, and hence $e_d(A)\ge e_d(I)-2\ge 4$.
\end{proof}

\subsection{Non-overlapping among different types}\label{sec:uniq}

This subsection is devoted to prove that the sets of the five types defined in Definitions~\ref{cons_12}-\ref{cons_5} are all nontrivial and mutually distinct from each other under isomorphisms.

We first introduce the concept of a decomposition of  a linear space $\mathbb F_p^n$ with $n\ge 2$. For any subspace $K\subset\mathbb F_p^n$ with codimension $1$ and a vector $v\notin K$, we simply call the pair $(v, K)$ a {\it decomposition} of $\mathbb F_p^n$, meaning that for every $x\in \mathbb F_p^n$, there exists a unique pair $(i, x')\in \mathbb F_p \times K$ such that $x=iv+x'$. Under this decomposition, a subset $A\subset \mathbb F_p^n$ can be partitioned into $p$ parts (not necessarily non-empty): $A\cap (iv+K),~i\in \mathbb F_p$. Denote $A_i$ the $i$th part but considered as a subset in $K$, that is,
\begin{equation}\label{eqai}
   A_i:=(A-iv)\cap K,~i\in \mathbb F_p \text{ and hence } A\cap (iv+K)=iv+A_i.
\end{equation}
Denote $Supp(A)$ as the set of $i\in \mathbb F_p$ such that $|A_i|\ne 0$, and call $\omega(A)=|Supp(A)|$ the \emph{weight} of $A$.
We don't impose the decomposition $(v,K)$ into those notations when they are clear.

\begin{lemma}\label{lem_contain_cret}
Let $A$ and $A'$ be two subsets of $\mathbb F_p^n$ for $n\ge 2$. Consider $A$ under a decomposition $(v, K)$ and $A'$ under a decomposition $(v', K')$. Suppose $\max_i |A_i|=p^{n-1}$ and $\omega(A')<p$.
Then $f(A)\subseteq A'$ for some automorphism $f$ of $\mathbb F_p^n$ implies that $s\cdot Supp(A)\subseteq Supp(A')$ for some $s\in \mathbb F_p$.
\end{lemma}
\begin{proof} 
Suppose $f(A)\subseteq A'$ for some automorphism $f$ of $\mathbb F_p^n$.
We first claim that $f(K)=K'$. Otherwise, as $\dim(K)=\dim(K')=n-1$, there exists $\epsilon \in K\backslash\{0\}$ such that $f(\epsilon)\in v'+ K'$. Suppose that $|A_i|=p^{n-1}$ for some $i\in \mathbb F_p$.  That is, $A_i=K$ and $(iv+K)\subseteq A$. Then the following $p$ mutually different elements $iv, iv+\epsilon, \ldots, iv+(p-1)\epsilon$ are in $A$. Suppose $f(iv)\in A'\cap (sv'+K')=sv'+A_s'$, then
for each $j\in \mathbb F_p$,

\[f(iv+j\epsilon)=f(iv)+jf(\epsilon)\in f(iv)+jv'+ K'=(s+j)v'+K'.\]
Since $f(A)\subseteq A'$, $f(iv+j\epsilon)\in A'\cap\left((s+j)v'+K'\right)=(s+j)v'+A'_{s+j}$, $j\in \mathbb F_p$.  Thus $Supp(A')=\mathbb F_p$, contradicting to the assumption that $\omega(A')<p$.


Since $f(K)=K'$, there exists some $s\in \mathbb F_p\backslash\{0\}$ such that $f(v)\in sv'+ K'$. For any $i\in Supp(A)$, there exists an element $x\in iv +K$ by definition. Then $f(x)\in A'\cap (si v'+K')$, which means $si\in Supp(A')$. So we have $s\cdot Supp(A)\subseteq Supp(A')$.
\end{proof}



By using Lemma~\ref{lem_one_hole_not_AP} and Lemma~\ref{lem_contain_cret}, we can prove that the sets with the five types are all nontrivial and  distinct from each other.


\begin{lemma}\label{lem_not_contained}
Let $p=(k+\ell)m+2+\lambda$ be a prime number with $k+\ell\geq 4$, $m\geq 5$ and $\lambda\in[0,k+\ell-3]$. Any subset $A\subset \mathbb F_p^n$ of some type $1$-$5$ is nontrivial.
Moreover, any two $A$'s of different types are not equivalent up to isomorphism.

\end{lemma}
\begin{proof} For each extremal cuboid $A_{k, \ell, p, j}$ defined in Theorem~\ref{thm_k_l_max_structure},  $Supp(A_{k, \ell, p, j})=[a_j,a_j+m]$ and $\omega(A_{k, \ell, p, j})=m+1<p$. Each $A$ of some type has the property $\max_i |A_i|=p^{n-1}$ under the natural decomposition, satisfying the requirements of Lemma~\ref{lem_contain_cret}. So if $A$ is contained in some $A_{k, \ell, p, j}$ under an isomorphism, there exists some $s\in \mathbb F_p\backslash\{0\}$ such that $s\cdot Supp(A)\subseteq Supp(A_{k, \ell, p, j})$. We will deduce a contradiction for each type of $A$.

When $A$ is of type $4$ or $5$,  $|Supp(A)|=\omega(A)=m+2>\omega(A_{k, \ell, p, j})$. So there is no way that $s\cdot Supp(A)\subseteq Supp(A_{k, \ell, p, j})$  for some $s\in \mathbb F_p\backslash\{0\}$. A contradiction occurs.


When $A$ is of type 2, $\omega(A)=m+1$ and $Supp(A)$ is an interval $[a, a+m]$ with $a=mk(\ell-k)^{-1}$. As $s\cdot Supp(A)= Supp(A_{k, \ell, p, j})$ is also an interval, $s$ can only be $\pm 1$ by Lemma~\ref{lem_interval_distinct}, which leads to $a= a_j$ or $-a= a_j+m$. Both cases lead to obvious contradictions by knowing  $a_j=-(mk+j+1)(k-\ell)^{-1}$ with $j\in [0, \lceil\frac{\lambda+1}2\rceil-1]$ and $\lambda\in[0,k+\ell-3]$.

When $A$ is of type 1, $\omega(A)=m$ and $Supp(A)$ is an interval $[a, a+m-1]$ with $a$ satisfying certain requirements. As $s\cdot Supp(A)\subset Supp(A_{k, \ell, p, j})$, that is $s\cdot [a, a+m-1]\subset [a_j, a_j+m]$, and $s\cdot Supp(A)$ is an AP, by Lemma~\ref{lem_one_hole_not_AP}, $s\cdot [a, a+m-1]=[a_j, a_j+m-1]$ or $s\cdot [a, a+m-1]=[a_j+1, a_j+m]$. By the same analysis as in the type 2 case, $s$ must be  $\pm 1$. All four possibilities lead to obvious contradictions.

When $A$ is of type 3, $\omega(A)=m$, and $Supp(A)=\{a-1, a+m\}\cup[a+1, a+m-2]\triangleq A^\star$
with  $a=(\ell m+k-1)(k-\ell)^{-1}$ and  $k+\ell=\lambda+4$.
\begin{itemize}
 \item We first compute $e_d(A^\star)$ for each $d\in [\frac{p-1}{2}]$. It is easy to check that $E_1(A^\star)=\{(a-1, a-2), (a-1, a), (a+1, a), (a+m-2, a+m-1), (a+m, a+m-1), (a+m, a+m+1)\}$ and $E_2(A^\star)=\{(a-3, a-1), (a+2, a), (a+m-3, a+m-1), (a+m, a+m+2)\}$. Hence $e_1(A^\star)=6$ and $e_2(A^\star)=4$. As $m\ge 5$, $a+3, a+m-4\in A^\star$. Hence $\{(a-1, a-4), (a+1, a-2), (a+3, a), (a+m-4, a+m-1), (a+m-2, a+m+1), (a+m, a+m+3)\}\subseteq E_3(A^\star)$, which gives $e_3(A^\star)\ge 6$. For all $d\in [4,\frac{p-1}2]$, as $k+\ell\ge 5$ in Definition~\ref{cons_3}, $\{(a-1, a-1-d), (a+1, a+1-d), (a+2, a+2-d), (a+m-3, a+m-3+d), (a+m-2, a+m-2+d), (a+m, a+m+d)\}\subseteq E_d(A^\star)$ and hence $e_d(A^\star)\ge 6$. To sum up, when $m\ge 5$, $e_2(A^\star)=4$ and $e_d(A^\star)\ge 6$ for any $d\neq 2$.
 \item Since $e_d(A^\star)>2$ for all $d$, $A^\star$ is not an AP, and so does $s\cdot A^\star$. As $s\cdot A^\star\subset Supp(A_{k, \ell, p, j})$,
$s\cdot A^\star$ must be of the form $[a_j, a_j+m]\backslash \{a_j+x\}$ for some $x\in [2, m-1]$ by Lemma~\ref{lem_one_hole_not_AP}. Notice that $e_1([a_j, a_j+m]\backslash \{a_j+x\})=|\{(a_j, a_j-1), (a_j+x-1, a_j+x), (a_j+x+1, a_j+x), (a_j+m, a_j+m+1)\}|=4$. As $\mathcal E(A^\star)=\mathcal E(s\cdot A^\star)$ only contain one copy of $4$, i.e.,  $e_2(A^\star)=e_1(s\cdot A^\star)=4$, we conclude that $2s=\pm 1$ (since by the definition of $e_d(A)$, $e_{\pm jd}(j\cdot A)=e_d(A)$), and hence $s=(p\pm 1)\slash 2$. Further, as $a+1, a+2$ are both in $A^\star$, there exist two elements with distance $(p-1)\slash 2$ in $s\cdot A^\star=[a_j, a_j+m]\backslash \{a_j+x\}$. This is an obvious contradiction as $k+\ell\ge 4$.

\end{itemize} 

Finally, we show that  any two $A$'s of different types are not equivalent up to isomorphism. Suppose $A$ and $A'$ are of types $i_1$ and $i_2$, respectively with $i_1\neq i_2\in [5]$, and they are equivalent up to isomorphism. By Lemma~\ref{lem_contain_cret}, $\omega(A)=\omega(A')$, which forces $(i_1, i_2)=(1, 3)$ or $(4, 5)$. The latter case does not happen because type 5 requires $k+\ell=4$ but type 4 requires $k+\ell=5$. So the only choice is $(i_1, i_2)=(1, 3)$. Again by Lemma~\ref{lem_contain_cret} this means there exists some $s\in\mathbb F_p\backslash\{0\}$ such that $s\cdot A^\star$ with $a=(\ell m+k-1)(k-\ell)^{-1}$ is an interval. This contradicts to that $e_d(A^\star)>2$ for all $d$.\end{proof}


\subsection{The Largest $(k, \ell)$-sum-free sets}

Note that the upper bound in Theorem~\ref{thm_k_l_max_structure} on the size of a  $(k, \ell)$-sum-free set in $\mathbb F_p$ can be deduced from \cite[Theorem 5]{Bajnok2009}. However, the characterization of the extremal structures can not be found in the literature. For completeness, we prove the bounds and extremal cases of  Theorem~\ref{thm_k_l_max_structure} together in this section.

We begin with the case when the dimension equals  1. For convenience, from now on, we always assume that $k>\ell\ge 1$ are integers, and write
\begin{equation}\label{eqp}
  p:=m(k+\ell)+\lambda+2 \text{ be a prime with $m\geq 1$  and $\lambda\in [0, k+\ell-3]$}.
\end{equation}
That is, $m:= \lfloor\frac{p-2}{k+\ell}\rfloor$.

\begin{lemma}\label{lem_m+1_interval}
Let $A\subset \mathbb F_p$ be a $(k, \ell)$-sum-free subset. Then $|A|\le m+1$. If $\lambda\in [0, k+\ell-3]$ and $|A|= m+1$, then $A$ must be an interval $[a_j, a_j+m]$  up to isomorphism, where $a_j=-(km+1+j)(k-\ell)^{-1}$ for some $j\in [0, \lceil\frac{\lambda+1}2\rceil-1]$ as defined in Theorem~\ref{thm_k_l_max_structure}.
That is, $A$ has $\lceil\frac{\lambda+1}2\rceil$ choices up to isomorphism, and $A$ is unique when  $\lambda=0$ or $1$.
\end{lemma}
\begin{proof}
From the $(k, \ell)$-sum-free property of $A$, $kA\cap\ell A=\emptyset$ and hence
\[p\ge |kA|+|\ell A|\ge (k+\ell)|A|-(k+\ell)+2,\]
by Theorem~\ref{thm_Kneser}. This means $p-2\ge (k+\ell)(|A|-1)$ and hence $|A|\le m+1$.

If $|A|= m+1$ and $\lambda\in [0, k+\ell-3]$, we first show that up to isomorphism $A$ must be an interval. If not,
by Theorem~\ref{thm_Vosper}, as $|A|\ge 2$ and $|kA|\le p-2$, we have $|A+A|\ge 2|A|$ and hence $|iA|\geq i|A|$ for any $2\leq i\leq k$. In particular, we have
\[|kA|\ge k|A|=km+k.\] However,
\[|kA|\le p-|\ell A|\le p-\ell| A|=km+\lambda+2-\ell,\] which is a contradiction since $km+\lambda+2-\ell<km+k$.

Now suppose $A=[a, a+m]$ for some $a\in \mathbb F_p$.
Then $kA=[ka, ka+km]$ and $\ell A=[\ell a, \ell a+\ell m]$ which are disjoint. So $|\mathbb F_p \setminus(kA\cup \ell A)|= p- |kA|-|\ell A|=\lambda$. We consider the length of segments between $kA$ and $\ell A$ on the cycle formed by elements of $\mathbb F_p$. For each $j\in [0, \lambda]$,  suppose $\ell a -(ka+km)-1=j$. Then
\[a=-(km+1+j)(k-\ell)^{-1}.\]
By Lemma~\ref{lem_interval_distinct}, two values $j$ and $j'$ lead to the same interval up to isomorphism if and only if $j+j'=\lambda$. Hence  $j\in [0, \lceil\frac{\lambda+1}2\rceil-1]$ and $A$ has $\lceil\frac{\lambda+1}2\rceil$ choices up to isomorphism.
\end{proof}

From $n=1$ to general $n$, we need the following lemma, which is an extension of \cite[Lemma 2.5]{Reiher2024}.

\begin{lemma}\label{lem_common_k}
Let $E_1, E_2, \ldots, E_k, F_1, \ldots, F_\ell$ be $k+\ell$ nonempty subsets of $\mathbb F_p^n$ such that $(\sum_{i=1}^k E_i)\cap(\sum_{j=1}^\ell F_j)=\emptyset.$ Then $\sum_{i=1}^k|E_i|+\sum_{j=1}^\ell |F_j|\le p^n+(k+\ell-2)p^{n-1}$.

 Moreover, if $\sum_{i=1}^k|E_i|+\sum_{j=1}^\ell |F_j|> p^n+ (k+\ell-2)p^{n-2}$, there exists a decomposition $(v,K)$ of $\mathbb F_p^n$ satisfying the  following properties.  For each $ X\in\{E_i\}_{i\in [k]}\cup\{F_j\}_{j\in [\ell]}$, denoted $X^\star=Supp(X)$. Then
    \begin{itemize}
 \item[(i)] $(\sum_{i=1}^k E_i^\star)\cap(\sum_{j=1}^\ell F_j^\star)=\emptyset.$
   \item[(ii)] $\sum_{i=1}^k|E_i^\star|+\sum_{j=1}^\ell |F_j^\star|\in [p+1, p+k+\ell-2]$.
\end{itemize}

\end{lemma}
\begin{proof} Denote $E=\sum_{i=1}^k E_i$ and $F=\sum_{j=1}^\ell F_j$.
Let $K_1=Sym(E)$ and let $K_2=Sym(F)$. As both $E$ and $F$ are nonempty and disjoint from each other, $|K_1|, |K_2|\le p^{n-1}$. By Theorem~\ref{thm_Kneser},
\begin{equation*}
\begin{aligned}
\sum_{i=1}^k |E_i|+ \sum_{j=1}^\ell |F_j|& \le |E|+ (k-1)|K_1|+ |F|+(\ell-1)|K_2|\\
&\le p^n+(k+\ell-2)p^{n-1},
\end{aligned}
\end{equation*}
%
which is the first assertion.

Moreover, if $\sum_{i=1}^k|E_i|+\sum_{j=1}^\ell |F_j|> p^n+ (k+\ell-2)p^{n-2}$, at least one of $K_1$ and $K_2$ has size $p^{n-1}$. Without loss of generality, let $|K_1|=p^{n-1}$ and denote it by $K$. Choose a vector $v\notin K$, thus $(v,K)$ forms a decomposition of $\mathbb F_p^n$. 

(i) As $K=Sym(E)$,  $E$ is a union of cosets of $K$, that is $E=E+K$. Since $E\cap F=\emptyset$, the smallest union of cosets of $K$ containing $F$, that is $F+K$,  is disjoint from $E$. Noting that $E+K=\sum_{i=1}^k (E_i+K)$ and $ F+K=\sum_{i=1}^\ell (F_i+K)$,
we have  $\sum_{i=1}^k (E_i+K)\cap\sum_{j=1}^\ell (F_i+K)=\emptyset.$ Hence $(\sum_{i=1}^k E_i^\star)\cap(\sum_{j=1}^\ell F_j^\star)=\emptyset$.




(ii) As $\sum_{i=1}^k (E_i+K)\cap\sum_{j=1}^\ell (F_j+ K)=\emptyset$, by the upper bound in the first statement,  \[p^n+(k+\ell-2)p^{n-1}\ge \sum_{i=1}^k|E_i+K|+\sum_{j=1}^\ell|F_j+K|=|K|(\sum_{i=1}^k|E_i^\star|+\sum_{j=1}^\ell|F_j^\star|).\] Hence $\sum_{i=1}^k|E_i^\star|+\sum_{j=1}^\ell|F_j^\star|\le p+k+\ell-2$. On the other hand, as $|K|(\sum_{i=1}^k|E_i^\star|+\sum_{j=1}^\ell|F_j^\star|)\ge  \sum_{i=1}^k|E_i|+\sum_{j=1}^\ell |F_j|> p^n+ (k+\ell-2)p^{n-2}$, we have $\sum_{i=1}^k|E_i^\star|+\sum_{j=1}^\ell|F_j^\star|\ge p+1$.
\end{proof}

Now we are ready to prove Theorem~\ref{thm_k_l_max_structure}.

\begin{proof}[Proof of Theorem~\ref{thm_k_l_max_structure}]
The case when $n=1$ has been proven in Lemma~\ref{lem_m+1_interval}.

For general $n\ge 2$, let $A\subset \mathbb F_p^n$ be a $(k,\ell)$-sum-free set with maximum size.
Let $K_1$ be the symmetric group of $kA$, then $kA$ is a union of cosets of $K_1$, and thus $kA =kA+K_1=k(A+K_1)$. Since further $kA\cap \ell A =\emptyset$ and $\ell A+K_1=\ell(A+K_1)$, we have $k(A+K_1)\cap \ell(A+K_1)=\emptyset$, that is
 $A+K_1$ is also $(k, \ell)$-sum-free. By the maximality of $|A|$ and $A\subseteq A+K_1$, we have $A=A+K_1$, and hence $|K_1|$ divides $|A|$. Similarly, if we denote $K_2$ the symmetric
group of $\ell A$, then $A=A+K_2$ and $|K_2|$ divides $|A|$. If $\max\{|K_1|, |K_2|\}= p^{n-1}$, then $p^{n-1}\mid|A|$. By Lemma~\ref{lem_common_k}, $(k+\ell)|A|\leq p^n+(k+\ell-2)p^{n-1}$, thus $|A|\leq (m+1)p^{n-1}$.
If $\max\{|K_1|, |K_2|\}\le p^{n-2}$, then by Theorem~\ref{thm_Kneser}, $(k+\ell)|A|\le p^n+(k+\ell-2)p^{n-2}<(m+1)p^{n-1}(k+\ell)$, where the last inequality requires $\lambda \leq k+\ell-3$. Hence $|A|<(m+1)p^{n-1}$.

When $|A|= (m+1)p^{n-1}$, that is,
\[(k+\ell)|A|=(m+1)p^{n-1}(k+\ell)=(p+k+\ell-2-\lambda)p^{n-1}> p^n+(k+\ell-2)p^{n-2},\]
where the last inequality holds because $k+\ell-2-\lambda\ge 1$ and $ p> k+\ell-2$.
Then by Lemma~\ref{lem_common_k} (i), there exists a  decomposition $(v,K)$ of $\mathbb F_p^n$,  such that under this decomposition, the support $A^\star\subset \mathbb F_p$ of $A$  is $(k, \ell)$-sum-free. 
By Lemma~\ref{lem_m+1_interval}, $|A^\star|\leq m+1$. But by $|A|= (m+1)p^{n-1}$, $|A^\star|\geq m+1$. So $|A^\star|= m+1$ and $A^\star$ is an interval up to isomorphism. Further, $A=A^\star\cdot v +K$. Thus $A$ is isomorphic to one of the $A_{k, \ell, p, j}$'s. The mutually difference among $A_{k, \ell, p, j}$ of different $j$ are trivial from Lemma~\ref{lem_interval_distinct} and Lemma~\ref{lem_contain_cret}.

%
%
\end{proof}

\section{Additive Structure Analysis}\label{sec_add_struc}


Let $(v,K)$ be a decomposition of $\mathbb F_p^n$. Recall that for any $A\subset \mathbb F_p^n$, $A_i=(A-iv)\cap K$, $i\in \mathbb F_p$.  To characterize the structure of $A$, we need to first characterize the support of $A$ under the decomposition $(v,K)$. 
Let $b_1, b_2, \ldots, b_p$ be a reordering  of elements in $\mathbb F_p$ such that
$|A_{b_1}|\ge |A_{b_2}|\ge\cdots\ge|A_{b_p}|.$
For each $i\in [p]$, denote
\begin{equation}\label{eqbcbeta}
 C_i:=\{b_1, \ldots, b_i\}, B_i:=A_{b_i},  \text{ and }  \beta_i(A):=|B_i|\slash p^{n-2}.
  \end{equation}
  Then $C_i\subset \mathbb F_p$, $B_i\subset K$ and
 $0\le \beta_i\le p$ for each $i\in [p]$. Denote $\omega:=\omega(A)$, then $C_\omega=Supp(A)$.


\begin{fact}\label{fact1}
  Let  $A\subset \mathbb F_p^n$ be a $(k, \ell)$-sum-free set. If $\sum_{i\in [k]}b_{{r_i}}=\sum_{j\in[\ell]}b_{{s_j}}$ for some ${r_1},\ldots, {r_k}, {s_1},\ldots,{s_\ell}\in [\omega]$ (not necessarily distinct),  then
  \begin{itemize}
    \item[(1)] $\left(\sum_{i=1}^k B_{{r_i}}\right)\cap \left(\sum_{j=1}^\ell B_{{s_j}}\right)=\left(\sum_{i=1}^k A_{b_{{r_i}}}\right)\cap \left(\sum_{j=1}^\ell A_{b_{{s_j}}}\right)=\emptyset,$ and
    \item[(2)] $\sum_{i=1}^k|B_{{r_i}}|+\sum_{j=1}^\ell |B_{{s_j}}|=\sum_{i=1}^k|A_{b_{{r_i}}}|+\sum_{j=1}^\ell |A_{b_{{s_j}}}|\le p^{n-1}+(k+\ell-2)p^{n-2}$.
  \end{itemize}

\end{fact}
\begin{proof} 
  (1) Suppose on the contrary that  there exist $a'_i\in B_{{r_i}}$, $i\in [k]$, and $c'_j\in B_{{s_j}}$, $j\in [\ell]$, such that $\sum_{i=1}^{k}a'_i=\sum_{j=1}^{\ell}c'_j$. Writing $a'_i=a_i-b_{{r_i}}v$ and $c'_j=c_j-b_{{s_j}}v$ for some $a_i,c_j\in A$, we deduce that $\sum_{i=1}^{k}a_i=\sum_{j=1}^{\ell}c_j$ by $\sum_{i\in [k]}b_{{r_i}}=\sum_{j\in[\ell]}b_{{s_j}}$. This contradicts to the fact that $A$ is $(k, \ell)$-sum-free.

  (2) As each $ B_{{r_i}}= A_{b_{{r_i}}} \subset K \cong \mathbb F_p^{n-1}$, by (1) and Lemma~\ref{lem_common_k}, we get (2).
\end{proof}



When $|A|\geq mp^{n-1}$, $\omega=\omega(A)\geq m$. In the next subsection, we characterize the structure of $C_m$ which is very close to $Supp(A)$.

\subsection{Structure of $C_m$}\label{sbsec:a}

For convenience, we say an interval $[a,b]$ is a {\it hole} of $C\subset \mathbb F_p$ if $[a,b]\cap C=\emptyset$ and $\{a-1,b+1\}\subset C$. A hole of length $j$ is called a $j$-hole. Let $I$ be the shortest interval covering $C$. We say $C$ is {\it $j$-holed from $I$} if each hole of $C$ contained in $I$ has length at most $j$, and call $C$ a \emph{$j$-holed interval} of length $|I|$.
Our goal of this subsection is to prove the following lemma, which determines the possible structures of $C_m$ for a $(k, \ell)$-sum-free set $A\subset \mathbb F_p^n$ with $|A|\ge mp^{n-1}$ under certain conditions.

\begin{lemma}\label{lem_stru_of_Cm}
Let $A\subset \mathbb F_p^n$ be a $(k, \ell)$-sum-free set of size at least $mp^{n-1}$. If under some decomposition $(v, K)$ of $\mathbb F_p^n$, $k\beta_m\ge p+k-1$, then $kC_m\cap \ell C_{\omega}= \emptyset$ and consequently $C_m$ is $(k, \ell)$-sum-free. Moreover, if further $\lambda\in [0,k+\ell-5]$ 
and $m\ge \max\{7, 5T(\frac 1{k+\ell})\}$, then up to isomorphism, $C_m$ is either an interval of length $m$, or covered by a $(k, \ell)$-sum-free interval of length $m+1$.
\end{lemma}

Before proving this lemma, we need some basic results about the structure of a $(k, \ell)$-sum-free set $C\subset \mathbb F_p$ of size $m$ which has an interval cover of length very close to $m$. 

\begin{claim}\label{clm3}
Let $C\subset \mathbb F_p$ be a $(k, \ell)$-sum-free set of size $m\ge 3$, and let $I$ be the shortest interval covering $C$.
If $|I|=m+1$, then $I$ is $(k, \ell)$-sum-free.
\end{claim}
\begin{proof}
Denote $I=[a, a+m]$ for some $a\in\mathbb F_p$.
Denote $\{b\}=I\backslash C$. As $I$ is the shortest cover, $b\in [a+1, a+m-1]$.

If $b\notin\{a+1, a+m-1\}$, from $m\ge 3$ we have $kI=kC$ and $\ell I=\ell C$. Since $C$ is $(k, \ell)$-sum-free, $I$ is also $(k, \ell)$-sum-free.

 Otherwise, $b\in\{a+1, a+m-1\}$. We only prove when $b=a+1$, while $b=a+m-1$ is similar. As $m\ge 3$, for any $2\le i\le k$, $iC=\{ia\}\cup[ia+2, ia+im]=iI\backslash\{ia+1\}$.
Particularly, $kC=kI \backslash\{ka+1\}$ and $\ell C=\ell I\backslash\{\ell a+1\}$. From $kC\cap \ell C=\emptyset$, we also have $kI\cap \ell I=\emptyset$.
\end{proof}

\begin{claim}\label{clmcm}
Let $C\subset \mathbb F_p$ be a $(k, \ell)$-sum-free set of size $m\ge 3$, and let $I=[a, a+m+t]$ be the shortest interval covering $C$ with $t\in\{1, 2\}$. Let $I_1$ and $I_1'$ be the longest sub-intervals of $C$ containing $a$ and $a+m+t$, respectively.
If $C$ is $j$-holed from $I$ for some $j\ge 1$, then $|I_1|\leq j$ and $|I_1'|\leq j$ .
\end{claim}
\begin{proof}
By symmetry we only need to prove $|I_1|\leq j$. Assume that on the contrary $|I_1|\geq j+1$.
Then for any integer $i\ge 1$, as long as $iI\neq \mathbb F_p$, it is easy to see that $iC$ is $j$-holed from $iI$ with both ends contained in $iC$. Specially, $kC$ and $\ell C$ are $j$-holed from intervals $kI= [ka, ka+km+kt]$ and $\ell I=[\ell a, \ell a+\ell m+\ell t]$, respectively.

As $kI_1\subset kC$,  $kI_1\cap \ell C=\emptyset$. Actually we further have $k I_1\cap \ell I= \emptyset$. Otherwise, since the two ends of $\ell I$ are in $\ell C$, we have $k I_1\subseteq \ell I$. 
However, we have the assumption that $|kI_1|\geq j+1$. This contradicts $kI_1\cap \ell C=\emptyset$ as $\ell C$ is $j$-holed from $\ell I$.

By the same analysis, we also get $\ell I_1\cap kI=\emptyset$. As $kI_1$ and $\ell I_1$ are leading sub-intervals of $kI$ and $\ell I$ in the anticlockwise direction, respectively, we have $kI\cap \ell I=\emptyset$. Hence
\[p\ge |kI|+|\ell I|=km+kt+1+\ell m+\ell t+1=(m+t)(k+\ell)+2= p-\lambda+ (k+\ell)t.\]
As $(k+\ell )t\ge k+\ell > \lambda$, this leads to an obvious contradiction.
\end{proof}

\begin{claim}\label{clmi1}
Let $C\subset \mathbb F_p$ be a $(k, \ell)$-sum-free set of size $m\ge 3$, and let $I=[a, a+m+t]$ be the shortest interval covering $C$ with $t\in\{1, 2\}$.
Suppose  $C$ is $j$-holed from $I$ for some $j\ge 1$.
If $a+1\notin C$, $a+2\in C$, and the longest sub-interval of $C$ containing $a+2$ is of length at least $j+1$,
then $|kI\cap\ell I|\leq 4$.
\end{claim}
\begin{proof}
Let $I'$ be the shortest interval covering $a$ and the longest sub-interval of $C$ starting at $a+2$.
Hence $kI'\backslash\{ka+1\}\subseteq kC$, and $\ell I'\backslash\{\ell a+1\}\subseteq \ell C$.
This means $kI'\backslash\{ka, ka+1\}\subseteq kC$ and $\ell I'\backslash\{\ell a, \ell a+1\}\subseteq \ell C$ are leading intervals of $kI\backslash\{ka, ka+1\}$ and $\ell I\backslash\{\ell a, \ell a+1\}$, respectively, of lengths at least $j+1$.

By  the same arguments as in Claim~\ref{clmcm}, we have $(kI\backslash\{ka, ka+1\})\cap(\ell I\backslash\{\ell a, \ell a+1\})=\emptyset$. So $|kI\cap\ell I|\leq 4$.
\end{proof}

\begin{proof}[Proof of Lemma~\ref{lem_stru_of_Cm}] We first prove that for any $r_1, \ldots, r_k\in[m]$ (not necessarily mutually different), $\sum_{i=1}^{k}B_{r_i}= K$. Otherwise, $H=Sym(\sum_{i=1}^{k}B_{r_i})$ is a proper subset of $K$ and hence $|H|\le p^{n-2}$. However, by Theorem~\ref{thm_Kneser},
\[|\sum_{i=1}^{k}B_{r_i}|\ge \sum_{i=1}^{k}|B_{r_i}|-(k-1)|H|\ge k|B_{m}|-(k-1)|H|\ge (p+k-1)p^{n-2}-(k-1)p^{n-2}= p^{n-1}=|K|,\]
which still means $\sum_{i=1}^{k}B_{r_i}=K$. The last inequality is by $k\beta_m\ge p+k-1$ and $|H|\le p^{n-2}$.

Next, we show that 
$kC_m\cap \ell C_{\omega}= \emptyset$.
If otherwise $kC_m\cap \ell C_{\omega}\neq \emptyset$,  there exist $b_{{r_i}}\in C_{m}$, $i\in [k]$, and $b_{{s_j}}\in C_{\omega}$, $j\in [\ell]$, such that $\sum_{i\in [k]}b_{{r_i}}=\sum_{j\in[\ell]}b_{{s_j}}$. By Fact~\ref{fact1},
$(\sum_{i=1}^k B_{{r_i}})\cap (\sum_{j=1}^\ell B_{{s_j}})=\emptyset.$ However, this is impossible since $\sum_{j=1}^\ell B_{{s_j}}\neq\emptyset$ and  $\sum_{i=1}^{k}B_{{r_i}}=K$.

As $\omega\geq m$, from $kC_m\cap \ell C_{\omega}= \emptyset$ we have $C_m$ is $(k, \ell)$-sum-free. To characterize the structure of $C_m$, we first show that up to isomorphism $C_m$ can be covered by an interval of length $m+3$ in $\mathbb F_p$  by using the definition of the function $T$. Let $c=\frac1{k+\ell}$, then $|C_m|=m\leq cp$. Since $m\geq 5T(c)$, $\tau(c)m\geq 5$. If $|2C_m|\le 2m+2$, then $|2C_m|\le (2+\tau(c))m-3$, which implies that
$C_m$ has the covering property by Definition \ref{deftau}. That is, if $|2C_m|\le 2m+2$, $C_m$ is covered by an AP of length $|2C_m|-|C_m|+1$ which is at most $m+3$, or by an interval of length $m+3$ up to isomorphism. Next, we show $|2C_m|\le 2m+2$.
 To estimate the size of $|2C_m|$,  we view $kC_m$ and $\ell C_m$ as  sums of many $2C_m$ and possibly one  more $C_m$. When $k$ and $\ell$ are both odd, by the $(k, \ell)$-sum-free property of $C_m$ and Theorem~\ref{thm_Cauchy_daven}, we have
\[p\ge |kC_m|+|\ell C_m|\ge \frac{k-1}2|2C_m|+\frac{\ell-1}2|2C_m|+2|C_m|-\frac{k+\ell+2}2+2.\]
So \begin{equation}\label{eq2cm}
     |2C_m|\le 2m+1+\frac{2(\lambda+2)}{k+\ell-2}< 2m+3,
   \end{equation}
 where the last inequality is by $\lambda \leq k+\ell-5$. As $|2C_m|$ is an integer, $|2C_m|\le 2m+2$. The three other cases when at least one of $k$ and $\ell$ is even yield the same bound $|2C_m|\le 2m+2$ by the same analysis.

Now denote $I$ the shortest interval covering $C_m$, thus $m\leq |I|\leq m+3$.
Write $I=[a, a+m+t]$ for some $a\in\mathbb F_p$ and $t\in [-1, 2]$. So $\{a,a+m+t\}\subset C_m$.
Then $C_m\subset I$ can be seen as a union of at most four disjunct intervals $C_m= \sqcup_{i\in [h]} I_i$ for some $h\in [4]$, which are numbered sequentially.  If $|I|=m$, then $C_m=I$ is already an interval of length $m$.
If $|I|=m+1$, $C_m$ is already contained in an interval of length $m+1$, which is $(k, \ell)$-sum-free from Claim~\ref{clm3}.

For all the remaining cases, that is, $t\in \{1,2\}$, we will deduce contradictions. Since $|I|\geq m+2$, we have $h\in \{2, 3, 4\}$ and $C_m$ is  $(t-h+3)$-holed from $I$.

When $h=2$, $C_m$ is $3$-holed from $I$. By Claim~\ref{clmcm}, $|I_1|, |I_2|\leq 3$. But this contradicts $C_m= \sqcup_{i\in [2]} I_i$ and $|C_m|=m \ge 7$. 


When $h=4$, $C_m$ is $1$-holed from interval $I$ and $t=2$. By Claim~\ref{clmcm},  $|I_1|=|I_4|=1$. By symmetry,  $|I_2|\ge \frac{m-2}{2}\ge 2$. By Claim~\ref{clmi1}, $|kI\cap\ell I|\leq 4$. Hence $p\ge |kI|+|\ell I|-4= (m+2)(k+\ell)-2> m(k+\ell)+2+\lambda=p$, which is  a contradiction. Here the last inequality is from $\lambda \leq k+\ell-3$.



When $h=3$, $C_m$ is  $2$-holed from interval $I$. By Claim~\ref{clmcm},  $|I_1|,|I_3|\leq 2$ and $|I_2|\ge m-4\ge 3$. By symmetry, we can assume that the hole between $I_1$ and $I_2$ is a $1$-hole. If $|I_1|=1$, by Claim~\ref{clmi1}, $|kI\cap\ell I|\leq 4$. Then
\begin{equation}\label{eqp3}
  p\ge |kI|+|\ell I|-4\geq  (m+1)(k+\ell)-2> m(k+\ell)+2+\lambda=p,
\end{equation}
 leading to a contradiction. Here, the last inequality in (\ref{eqp3}) requires $\lambda \leq k+\ell-5$.
If $|I_1|=2$, $k(I_1\cup I_2)$ and $\ell (I_1\cup I_2)\backslash[\ell a, \ell a+2]$ form leading intervals of length at least three of $kI$ and $\ell I\backslash[\ell a, \ell a+2]$, respectively, which lead to $|kI\cap \ell I|\le 3$ by the similar arguments as in Claim~\ref{clmi1}. Here three points are excluded from $\ell (I_1\cup I_2)$ because of the possibility of $\ell=1$. Then
\begin{equation}\label{eqp4}
  p\ge |kI|+|\ell I|-3\geq  (m+1)(k+\ell)-1> m(k+\ell)+2+\lambda=p,
\end{equation}
 which is again a contradiction. Here, the last inequality in (\ref{eqp4}) requires $\lambda \leq k+\ell-4$.
%
%
%
%
\end{proof}

We remark that the requirement $\lambda \leq k+\ell-5$ in Lemma~\ref{lem_stru_of_Cm} is best possible.
A counter example when $\lambda =k+\ell-4$ has been stated in Definition~\ref{cons_3}, where the shortest interval covering $C_m$ is of length $m+2$.

\subsection{Upper bounds on the size of $k+\ell$ parts}
By the definition of $B_i$, the size of $B_i$ is small when $i$ is big. In this section, we give upper bounds on the size of a union of $k+\ell$ parts $B_i$'s when the subscript sum is big.
These bounds will play important roles in estimating the weight of $A$ in the next subsection.
The following  bound is an easy consequence of Lemma~\ref{lem_common_k} and Fact~\ref{fact1}.

\begin{lemma}\label{coro_beta_sum}
   Let  $A\subset \mathbb F_p^n$ be $(k, \ell)$-sum-free. Then under a decomposition $(v, K)$ of $\mathbb F_p^n$, we have \[\sum_{i=1}^k \beta_{r_i}+\sum_{j=1}^\ell \beta_{s_j}\le p+k+\ell-2\] for  some $r_1, \ldots, r_k, s_1, \ldots, s_\ell\in[\omega]$ (not necessary mutually different) if one of the following holds:
   \begin{itemize}
     \item[(1)]$\left(\sum_{i=1}^k C_{r_i}\right)\cap \left(\sum_{j=1}^\ell C_{s_j}\right)\neq\emptyset$; or
     \item[(2)] $\sum_{i=1}^k r_i+\sum_{j=1}^\ell s_j\ge p+k+\ell-1$.
   \end{itemize}
\end{lemma}
\begin{proof}
(1) As $(\sum_{i=1}^k C_{r_i})\cap (\sum_{j=1}^\ell C_{s_j})\neq\emptyset$, there exists $b_{\tilde{r_i}}\in C_{r_i}$, $i\in [k]$, and $b_{\tilde{s_j}}\in C_{s_j}$, $j\in [\ell]$, such that $\sum_{i\in [k]}b_{\tilde{r_i}}=\sum_{j\in[\ell]}b_{\tilde{s_j}}.$
By Fact~\ref{fact1}, we have
$\sum_{i=1}^k |B_{\tilde{r_i}}|+\sum_{j=1}^\ell |B_{\tilde{s_j}}|\leq p^{n-1}+(k+l-2)p^{n-2}$. Then
\[\sum_{i=1}^k \beta_{r_i}+\sum_{j=1}^\ell \beta_{s_j}\le \sum_{i=1}^k \beta_{\tilde{r_i}}+\sum_{j=1}^\ell \beta_{\tilde{s_j}}= \left(\sum_{i=1}^k |B_{\tilde{r_i}}|+\sum_{j=1}^\ell |B_{\tilde{s_j}}|\right)\slash p^{n-2}\le p+k+\ell-2.\] Here, the first inequality is because that for any $x\in\{r_i: i\in [k]\}\cup\{s_j: j\in [\ell]\}$, $\tilde{x}\leq x$ and thus $\beta_{\tilde{x}}\ge \beta_{x}$ by the definition of $C_i$.

(2) This one can be reduced to (1) since if $\sum_{i=1}^{k}|C_{r_i}|+\sum_{j=1}^{\ell}|C_{s_j}|= \sum_{i=1}^k r_i+\sum_{j=1}^\ell s_j\ge p+k+\ell-1,$ we have
$(\sum_{i=1}^k C_{r_i})\cap (\sum_{j=1}^\ell C_{s_j})\neq\emptyset$
by Lemma~\ref{lem_common_k}.
\end{proof}

%


The next two bounds, Lemmas~\ref{lem_improving_coro} and \ref{lem_improving_improving_coro}, are modifications of Lemma~\ref{coro_beta_sum} (2) for special $r_1, \ldots, r_k, s_1, \ldots, s_\ell$ by using the solved case of $3k-4$ conjecture when $|2A|=2|A|$. We list a stronger form as follows, whose proof can be
seen in \cite{Grynkiewicz2024} for a short path from \cite{Grynkiewicz2009} to it.
\begin{theorem}[\cite{Grynkiewicz2024,Grynkiewicz2009}]\label{thm_3k-4_r=0}
Let prime $p\ge 2$, and let $A, B\subset \mathbb Z_p$ nonempty satisfying $|A|\ge |B|\ge 4$ and
\[|A+B|=|A|+|B|\le\min\{|A|+2|B|-4, p-4\}.\]
Then by letting $C=-\mathbb Z_p\backslash(A+B)$, we can find APs $P_X, X\in\{A, B, C\}$ in $\mathbb Z_p$ with the same common difference, such that $ X\subseteq P_X \text{ and } |P_X|\le |X|+1$ for any $X\in\{A, B, C\}$.
\end{theorem}

\begin{lemma}\label{lem_improving_coro} Let  $A\subset \mathbb F_p^n$ be $(k, \ell)$-sum-free and $m\ge 4$. Suppose under a decomposition of $\mathbb{F}_p^n$,
 $\omega(A)\ge m+2$, and $r_1, \ldots, r_k, s_1, \ldots, s_\ell$ are chosen from $[m-1, \omega]$, such that both $m-1$ and $m$ appear at least once. If $\sum_{i=1}^k r_i+\sum_{j=1}^\ell s_j\geq p+k+\ell-2$, we have $\sum_{i=1}^k \beta_{r_i}+\sum_{j=1}^\ell \beta_{s_j}\le p+k+\ell-2$.
\end{lemma}
\begin{proof} By Lemma~\ref{coro_beta_sum}, we only need to consider the case when $\sum_{i=1}^k r_i+\sum_{j=1}^\ell s_j= p+k+\ell-2$.

Observe that any ordering of the $k+\ell$ subscripts $r_1, \ldots, r_k, s_1, \ldots, s_\ell$ does not change the value $\sum_{i=1}^k \beta_{r_i}+\sum_{j=1}^\ell \beta_{s_j}$.
If there exists an ordering of $r_1, \ldots, r_k, s_1, \ldots, s_\ell$ such that $(\sum_{i=1}^k C_{r_i})\cap (\sum_{j=1}^\ell C_{s_j}) \ne \emptyset$, the upper bound follows by Lemma \ref{coro_beta_sum} (1).
So it suffices to prove that such an ordering must exist.

We begin with an arbitrary ordering $L$ such that $r_1=m-1$ and $s_1=m$.
Based on $L$, we denote $A=\sum_{i=1}^k C_{r_i}$ and $B= \sum_{j=1}^\ell C_{s_j}$.
If $A\cap B\neq \emptyset$ then we are done. 
Otherwise, $A\cap B=\emptyset$, which implies
\begin{equation}\label{eq_5}
p\ge |A|+|B|.
\end{equation}
By recursively applying Theorem~\ref{thm_Cauchy_daven}, we have
\begin{equation}\label{eq_6}
|A|\ge \sum_{i=1}^k r_i- k+1;
\end{equation} and
\begin{equation}\label{eq_7}
|B|\ge \sum_{j=1}^\ell s_j- \ell+1.
\end{equation}
Since $\sum_{i=1}^k r_i+\sum_{j=1}^\ell s_j= p+k+\ell-2$, all equalities in Eq. (\ref{eq_5}), (\ref{eq_6}) and (\ref{eq_7}) hold. By the equality in Eq. (\ref{eq_6}) and Theorem~\ref{thm_Vosper}, all $C_{r_i}$, $i\in [k]$, and hence together with $A$ are all APs sharing the same common difference. Similarly, by the equality in Eq. (\ref{eq_7}) and Theorem~\ref{thm_Vosper}, $C_{s_j}$, $j\in [\ell]$, and $B$ are all APs sharing the same common difference.
By the equality in Eq. (\ref{eq_5}), $A=\mathbb{F}_p\setminus B$. Since $r_i,s_j\geq m-1\geq 3$, we have $3\leq |A|,|B|\leq p-3$. By Lemma~\ref{lem_interval_distinct},  $A$ and $B$  share the same common difference, and so do all the mentioned APs. Without loss of generality, we set the common difference to be $1$, and hence those APs are all intervals.


Specially, $A$, $B$, $C_m$ and $C_{m-1}$ are all intervals. Denote $C_{m-1}=[a, b]$ and $A=[x, y]$. Then $B=[y+1, x-1]$ and without loss of generality, we can set $C_m=[a, b+1]$ by symmetry.
Consider a new ordering $L'$ from $L$ by merely swapping the values of $r_1$ and $s_1$. That is, in $L'$, $r_1=m$, $s_1=m-1$, while other indices are coherent with $L$.
Denote the new $A$ and $B$ on $L'$ as $A'$ and $B'$, respectively. From that $A=C_{m-1}+{\sum_{i=2}^kC_{r_i}}=[x, y]$, we know that $A'=C_{m}+{\sum_{i=2}^kC_{r_i}}=[x, y+1]$. Similarly from $B$, we  get $B'=C_{m-1}+{\sum_{j=2}^{\ell}C_{s_j}}=[y+1, x-2]$. This means for $L'$, $A'\cap B'\neq\emptyset$, and $L'$ is indeed the required ordering.
\end{proof}

We further relax the restriction on $\sum_{i=1}^k r_i+\sum_{j=1}^\ell s_j$ in the following lemma.

\begin{lemma}\label{lem_improving_improving_coro} Let  $A\subset \mathbb F_p^n$ be $(k, \ell)$-sum-free and $m\ge 5$. Suppose  under a decomposition of $\mathbb F_p^n$,
 $\omega(A)\ge m+2$. Let $t$ be any integer in $[2, \omega-m]$. If $r_1, \ldots, r_k, s_1, \ldots, s_\ell$ are chosen from $[m-1, \omega]$, such that each element in  $[m-1, m+t]$ is chosen at least once, and $\sum_{i=1}^k r_i+\sum_{j=1}^\ell s_j\geq p+k+\ell-3$, then  $\sum_{i=1}^k \beta_{r_i}+\sum_{j=1}^\ell \beta_{s_j}\le p+k+\ell-2$ unless the following happens up to isomorphism:
\begin{itemize}
\item  $\ell=1$, and all $C_i$, $i\in [m-1, m+t]$ are intervals. Moreover, for any $i\in [m+1, m+t]$, if $C_i$ can be expressed as $[a, b]$ for some $a$ and $b$, then $C_{i-2}=[a+1, b-1]$.
\end{itemize}
\end{lemma}
\begin{proof}By Lemma~\ref{lem_improving_coro}, we only need to consider the case when $\sum_{i=1}^k r_i+\sum_{j=1}^\ell s_j= p+k+\ell-3$.
Then it suffices to show that, except the case stated in the lemma, there is an ordering of $r_1, \ldots, r_k, s_1, \ldots, s_\ell$ such that $(\sum_{i=1}^k C_{r_i})\cap (\sum_{j=1}^\ell C_{s_j}) \neq \emptyset$.


For any fixed ordering $L$, we follow the same notations $A$ and $B$ in the proof of Lemma~\ref{lem_improving_coro}, and have the same three inequalities Eq. (\ref{eq_5})-(\ref{eq_7}) when $A\cap B=\emptyset$. When $\sum_{i=1}^k r_i+\sum_{j=1}^\ell s_j= p+k+\ell-3$, exactly two of the three equalities hold, while the third one has a gap $1$. Depending on which one has a gap $1$, we analysis them one by one.

Suppose Eq. (\ref{eq_6}) has gap $1$, that is $|A|= \sum_{i=1}^k r_i- k+2$. By Theorem~\ref{thm_Vosper}, there must exist two $C_i$'s, say the first two, satisfying $|C_{r_1}+C_{r_2}|=r_1+ r_2$. As $k\ge 3$, we know that $|C_{r_1}+C_{r_2}+C_{r_i}|=|C_{r_1}+C_{r_2}|+|C_{r_i}|-1$ must hold for any $3\le i\le k$.
Hence $C_{r_1}+C_{r_2}$, $C_{r_3}$, \ldots, $C_{r_k}$ are all APs of the same unique common difference. Without loss of generality, we assume that they are all intervals. Then $A$ and $B$ are complementary intervals of $\mathbb{F}_p$, and  each $C_{s_j}, j\in [\ell]$ is an interval for any $\ell \geq 1$.  So $|C_{r_1}+C_{r_2}|\le |A|=p-|B|\le p-C_{m-1}= p-(m-1)\le p-4$ as $m\geq 5$. By Theorem~\ref{thm_3k-4_r=0}, both $C_{r_1}$ and $C_{r_2}$ are contained in intervals of lengths $r_1+1$ and $r_2+1$, respectively. Since $C_{r_1}+C_{r_2}$ is an interval of length $r_1+r_2$, the only choice is that one (say $C_{r_2}$) is an interval, while the other one $C_{r_1}$ is a $1$-holed interval of length $r_1+1$.

Suppose  Eq. (\ref{eq_7}) has gap $1$, then $\ell\ge 2$. By the same analysis as above, there is exactly one $C_{s_j}$, say $C_{s_1}$, to be a $1$-holed interval of length $s_1+1$, and all other $C_{s_j}$ are intervals.

Suppose Eq. (\ref{eq_5}) has gap $1$, that is $p=|A|+|B|+1$. Then the two equalities in  Eqs. (\ref{eq_6}) and (\ref{eq_7}) hold.  So $C_{r_i}$, $i\in [k]$, and  $A$ are all APs with the same common difference.
Suppose $C_{r_i}$ and $A$ are intervals. By $|B|=p-|A|-1$, $B$ is either an interval  or has exactly one $1$-hole. If $\ell\geq 2$, then $C_{s_j}$, $j\in [\ell]$, and $B$ are all APs with the same common difference. By Lemma~\ref{lem_one_hole_not_AP}, $B$ must be an interval. If $\ell=1$,  $B=C_{s_1}$ is either a $1$-holed interval of length $s_1+1$ or an interval of length $s_1$.

To sum up, no matter which inequality has  gap $1$, the sets $C_{r_i}$, $i\in [k]$ and $C_{s_j}, j\in[\ell]$ must satisfy one of the following two cases:
\begin{itemize}
  \item[(1)] All sets $C_{r_i}$, $i\in [k]$ and $C_{s_j}, j\in [\ell]$ are intervals except one set, which has exactly one $1$-hole.
  \item[(2)] All sets $C_{r_i}$, $i\in [k]$ and $C_{s_j}, j\in [\ell]$ are intervals. This happens only when $p=|A|+|B|+1$.
\end{itemize}
Next, we show that except the case stated in the lemma, there is a new ordering $L'$ with $A'\cap B'\ne \emptyset$.

Suppose Case (1) happens. Since the values of $r_i, s_j$ cover $ [m-1,m+t]$ and exactly one out of those $C_{r_i}$, $C_{s_j}$ sets is not an interval, there exists one $i\in [m-1, m+t-1]$ such that both $C_i$ and $C_{i+1}$ are intervals.  Let $C_j$ be that $1$-holed interval. As $k\ge 2$, we can assume that in $L$, $(r_1, r_k, s_1)=(i, j, i+1)$.
%
Consider a new ordering $L'$ from $L$ by merely swapping the values of $r_1$ and $s_1$.
Under both $L$ and $L'$, it is easy to check that Eq. (\ref{eq_6}) has gap $1$. Then the same argument in the proof of Lemma~\ref{lem_improving_coro} also works here, which can derive that if $A\cap B=\emptyset$ over $L$, then $A\cap B\neq\emptyset$ over $L'$. 

Suppose Case (2) happens. Then all $C_{r_i}$ and $C_{s_j}$ are all intervals, while $p=|A|+|B|+1$ under any ordering of $r_1,\ldots,r_k, s_1,\ldots,s_\ell$.
Specially, all $C_i$ with $i\in [m-1, m+t]$ are all intervals.
Suppose $C_i=[a, b]$ for some $i\in [m+1, m+t]$. Then $C_{i-2}$ as an interval has two choices by symmetry: $[a+2, b]$ and $[a+1, b-1]$.

\begin{itemize}
  \item Assume $C_{i-2}=[a+2, b]$. Suppose $(r_1, s_1)=(i, i-2)$ in  $L$, and consider $L'$  from $L$ by swapping the values of $r_1$ and $s_1$. Over $L$ denote $A=[x, y]$. Then $B=[y+2, x-1]$ or $[y+1, x-2]$.
  Then from $C_{i-2}=[a+2, b]$, we know that over $L'$ we have $A'= [x+2, y]$, and $B'=[y, x-1]$ or $[y-1, x-2]$, which means $A'\cap B'\neq\emptyset$.
  \item Assume $C_{i-2}=[a+1, b-1]$ and $\ell\ge 2$. Since $t\ge 2$, we can assume $(r_1, r_2, s_1, s_2)=(m-1, m+1, m, m+2)$ in $L$, and choose a new ordering $L'$ from $L$ by swapping the values of $r_1$ and $s_1$, and the values of $r_2$ and $s_2$. Set $i=m+2$. No matter $C_{m+1}=[a+1, b]$ or $[a, b-1]$, the same analysis derives that $A'\cap B'\ne\emptyset$ over $L'$.
\end{itemize}

Summing up, we have proved all cases except $C_{i-2}=[a+1, b-1]$ and $\ell=1$. This completes the proof.
\end{proof}

\subsection{The weight is either big or small}\label{sbsec:c}

Recall that when $A\subset \mathbb F_p^n$ is of size at least $mp^{n-1}$, the weight $\omega=\omega(A)$ is at least $m$.
In this section, we show that if $A$ is  $(k, \ell)$-sum-free, $\omega$ is either very big, close to $p$, or very small, close to $m$. The main tools are Lemmas~\ref{coro_beta_sum}-\ref{lem_improving_improving_coro}. For convenience, denote
\[\theta:=p-(m-1)(k+\ell)=k+\ell+\lambda+2.\]

\begin{lemma}\label{lem_l_upper_default}
Let integers $m\ge 5$ and $k+\ell-5\ge \frac\lambda 2$. Suppose $A\subset \mathbb F_p^n$ is a $(k, \ell)$-sum-free  subset with $|A|\ge mp^{n-1}$. Then under any $(v, K)$-decomposition, we have either $\omega\le m+2$ or $\omega> p-\theta$.
\end{lemma}

\begin{proof}
Assume on the contrary, $m+3\le \omega\le p-\theta$. We will upper bound $\sum_{i=1}^{\omega}\beta_i$ and deduce a contradiction to the size of $A$.

First, we  partition the index set $[p-\theta]=[(m-1)(k+\ell)]$ into $m-1$ subsets of size $k+\ell$, such that the sum of indices in almost every subset is big enough to apply Lemmas~\ref{coro_beta_sum} or \ref{lem_improving_coro}. Define 
\[S_i=\{i\}\cup [p-\theta- i\cdot(k+\ell-1)+1, p-\theta-(i-1)(k+\ell-1)]\triangleq \{i\}\cup [a_i,b_i], \text{ for  $i\in [m-2]$,}\]
and \[S_{m-1}=[m-1,m+k+\ell-2].\]
It is easy to check that $|S_i|=k+\ell$, $S_i\cap S_j=\emptyset$ for any $i\ne j$, and $\sqcup_{i\in [m-1]}S_i=[p-\theta].$
Denote $\sigma_i$ the sum of all indices in $S_i$, $i\in [m-1]$.

Next, consider the first $m-2$ parts. As $k+\ell> 2$, we know that $\sigma_i$ decreases when $i$ grows. Furthermore,  $$\sigma_{m-2}=m-2 + \sum_{i=1}^{k+\ell-1}(m+k+\ell-2+i)= (k+\ell)m+ \frac{3}{2}(k+\ell)^2-\frac{7}{2}(k+\ell).$$
Solving $(k+\ell)m+ \frac{3}{2}(k+\ell)^2-\frac{7}{2}(k+\ell)\geq p+k+\ell-1$ for $k+\ell$ gives $k+\ell\ge  \sqrt{\frac{2}{3}}\sqrt{\lambda+\frac{35}{8}}+1.5$, which is true since $k+\ell -4\ge \frac\lambda 2$. So we have $ \sigma_{m-2}\ge p+k+\ell-1$ and thus $\sigma_i\ge p+k+\ell-1$ for any $i\in [m-2]$. However, we still can not apply Lemma~\ref{coro_beta_sum} directly since indices in $S_i$ may not belong to $[\omega]$.
There are three cases to be considered for each $S_i=\{i\}\cup [a_i,b_i]$. If $\omega\ge b_i$, then $S_i\subset [\omega]$ and we can apply Lemma~\ref{coro_beta_sum} directly.
 If $\omega<a_i$, then $S_i\cap[\omega]=\{i\}$, and hence $\sum_{j\in S_i}\beta_{j}= \beta_i\le p\le p+k+\ell-2$ because $k+\ell> 2$. Finally, if $ \omega\in [a_i,b_i-1]$, $S_i\cap[\omega]$ has size $h+1$ for some $h\in [k+\ell-2]$. Then adding $(k+\ell-h-1)$ more copies of $\omega$, we get a  new multi-set $S_i'$ consisting of $k+\ell$ elements,  whose elements  are all in $[\omega]$. Noting that for fixed $h$, the sum $\sigma'_{i}$ of indices in $S_i'$ also decreases when $i$ grows, we have
\[\begin{aligned}
\sigma'_{i}\geq &\sigma'_{m-2}\geq m-2+(k+\ell-1)a_{m-2}\\
 \ge& m-2+(k+\ell-1)(m+k+\ell-1)\\
 =&(k+\ell)m+(k+\ell-1)^2-2\\
 \ge& p+k+\ell-1.
\end{aligned}\]
%
%
Here the last inequality is valid since solving it gives $k+\ell\ge \sqrt{\lambda+\frac{17}{4}}+1.5$, which is true by assumption. 
Then by applying Lemma~\ref{coro_beta_sum} on $S_i'$, we have
$$\sum_{j\in S_i}\beta_{j}= \sum_{j\in S_i\cap[\omega]}\beta_{j}\leq \sum_{j\in S'_i}\beta_{j}\le p+k+\ell-2.$$

Finally, we consider the last part $S_{m-1}=[m-1,m+k+\ell-2]$. Since $k+\ell\ge 5$ and $w\ge m+3$, we turn to the multi-set $S'_{m-1}$ consisting of four elements in $[m-1,m+2]$ and $(k+\ell-4)$ copies of $m+3$, whose $k+\ell$ elements are all  in $[\omega]$. Further, the sum $\sigma'_{m-1}$ of indices in $S_{m-1}'$ is
\begin{equation}\label{eqbeta}
\begin{aligned}
\sigma'_{m-1}= & (
m-1)+m+(m+1)+(m+2)+(k+\ell-4)(m+3)\\=&(k+\ell)m+3(k+\ell)-10\ge p+k+\ell-2
  \end{aligned}
   \end{equation}from $k+\ell-5\ge \frac{\lambda}{2}$.
     Then by Lemma~\ref{lem_improving_coro}, we have $\sum_{j\in S_{m-1}}\beta_{j}\le\sum_{j\in S'_{m-1}}\beta_{j}\le p+k+\ell-2$ by the monotonicity of $\beta_i$.

Now we deduce the contradiction. Since $w\leq p-\theta$, $ \sum_{i=1}^{\omega}\beta_i=\sum_{i=1}^{p-\theta}\beta_i
  =\sum_{i=1}^{m-1}\sum_{j\in S_i}\beta_j
 \le (p+k+\ell-2)(m-1)$. So $|A|=\sum_{i=1}^{\omega}\beta_i p^{n-2}\leq  (p+k+\ell-2)(m-1)p^{n-2}< mp^{n-1}$, leading to a contradiction.
\end{proof}


Next, we show that when $\lambda \leq k+\ell-5$, 
$\omega=m+2$ can not happen.

\begin{lemma}\label{lem_ell_upper_bound}
Let  $m\ge 5$ and $\lambda\in [0, k+\ell-5]$. Suppose $A\subset \mathbb F_p^n$ is a $(k, \ell)$-sum-free subset with $|A|\ge mp^{n-1}$. Then under any $(v, K)$-decomposition,  either $\omega\leq m+1$ or $\omega>p-\theta$.
\end{lemma}
\begin{proof} By Lemma~\ref{lem_l_upper_default}, we only need to prove that $\omega=m+2$ is not possible.

Let $\omega=m+2$. We follow the proof of Lemma~\ref{lem_l_upper_default} until the last part $S'_{m-1}$, which can not contain ${m+3}$ this time. But instead, we can define $S'_{m-1}$ to be the multi-set consisting of three elements in $[m-1,m+1]$ and $(k+\ell-3)$ copies of $m+2$. Then when $\lambda \leq k+\ell-5$,
\begin{equation}\label{eqsig}
\begin{aligned}
\sigma'_{m-1}= & (
m-1)+m+(m+1)+(k+\ell-3)(m+2)\\=&(k+\ell)m+2(k+\ell-3)\ge p+k+\ell-3,
  \end{aligned}
   \end{equation}
where the last equality holds if and only if $\lambda = k+\ell-5$. By Lemma~\ref{lem_improving_improving_coro}, the proof goes exactly the same as in the proof of Lemma~\ref{lem_l_upper_default}, unless the following two conditions both hold:
%
\begin{itemize}
 \item[(i)] $\ell=1$ and $\lambda = k+\ell-5=k-4$;
  \item[(ii)] All $C_i$, $i\in [m-1,m+2]$ are intervals. Moreover, for any $i\in [m+1, m+2]$, if $C_i$ can be expressed as $[a, b]$ for some $a$ and $b$, then $C_{i-2}=[a+1, b-1]$.
%

 \end{itemize}
Next, we only need to show (i) and (ii) cannot hold simultaneously.

Suppose on the contrary, then $\ell=1$ and $k\ge 4$. Since $|A|\geq mp^{n-1}$, $mp\le \sum_{i=1}^{m+2}\beta_i\le (m-1)p+3\beta_m$.
That is, $3\beta_m\ge p$.
Thus $k\beta_m\ge 4\beta_m\ge \frac43p\ge p+k-1$, where the last inequality holds because $m\ge 3$. Then by Lemma~\ref{lem_stru_of_Cm}, $C_m$ is $(k, 1)$-sum-free, which means $kC_m\cap C_m=\emptyset$, and $\mathbb F_p\backslash(kC_m\cup C_m)=p-m-(km-k+1)=k+\lambda+1=2k-3$. Denote $C_m=[a, a+m-1]$ for some $a$, then $kC_m=[ka,k(a+m-1)]$. Assume that the interval $[k(a+m-1)+1, a-1]$, which is one of the two intervals connecting $C_m$ and $kC_m$, has a shorter length $i\in [0, k-2]$. Then $i\equiv a-k(a+m-1)-1 \pmod p$.
 %
%

Next, we will deduce a contradiction to the size of $A$.
As $i\le k-2$, we consider the following three equations over $\mathbb F_p$ with $k$ terms on the left and one term on the right, and all terms are from $[a-1, a+1]\cup[a+m-2, a+m]$.
\begin{equation}\label{eq3a}
\begin{aligned}
(i+1)(a+m)+(k-i-2)(a+m-1)+(a+m-2)&=a-1;\\
(i+1)(a+m)+(k-i-1)(a+m-1)&=a;\\
(i+2)(a+m)+(k-i-2)(a+m-1)&=a+1.
\end{aligned}
\end{equation}
By (ii), $C_{m+2}=[a-1,a+m]$. As $m\ge 5$, $[a-1, a+1]$ and $[a+m-2, a+m]$ are two disjoint subsets of $C_{m+2}$. 
From that $A$ is $(k, 1)$-sum-free, by applying Fact~\ref{fact1}, 
each equation in (\ref{eq3a}) gives an inequality:
\begin{equation}\label{eq3b}
\begin{aligned}
(i+1)|A_{a+m}|+(k-i-2)|A_{a+m-1}|+|A_{a+m-2}|+|A_{a-1}|&\le (p+k-1)p^{n-2}; \\ (i+1)|A_{a+m}|+(k-i-1)|A_{a+m-1}|+|A_{a}|&\le (p+k-1)p^{n-2};\\ (i+2)|A_{a+m}|+(k-i-2)|A_{i+m-1}|+|A_{a+1}|&\le (p+k-1)p^{n-2}.
\end{aligned}
\end{equation}
Combining all those three together, we have $|A_{a-1}|+|A_a|+|A_{a+1}|+|A_{a+m-2}|+|A_{a+m-1}|+|A_{a+m}|\le (3p+3k-3)p^{n-2}$. Hence we get
\[
\begin{aligned}
|A|=&\sum_{i=a-1}^{a+m}|A_i|\\
 =&\sum_{i=a+2}^{a+m-3}|A_i|+(|A_{a-1}|+|A_a|+|A_{a+1}|+|A_{a+m-2}|+|A_{a+m-1}|+|A_{a+m}|)\\
\le &(m-4)p^{n-1}+ (3p+3k-3)p^{n-2}<mp^{n-1},
\end{aligned}
\]
which leads to a contradiction when $m\ge 5$.

As a consequence, (i) and (ii) cannot hold simultaneously and the case $\omega=m+2$ cannot happen.
\end{proof}
Finally, we show that if the weight $\omega$ of a large $(k, \ell)$-sum-free set $A$ under some decomposition $(v,K)$ is big, then the distribution of $A$ is nearly balanced, that is, the size of each part $A_i$ won't differ too much.

\begin{lemma}\label{lem_four_coeff_small}
Let $\lambda\in [0, k+\ell-3]$, $k+\ell\ge 4$ and $m\ge 9+\lambda$.
Suppose  $A\subset \mathbb F_p^n$ is a $(k, \ell)$-sum-free subset of size at least $mp^{n-1}$ and with $\omega> p-\theta$ under a decomposition $(v, K)$.
Then there exists a nonnegative integer $u$ sastisfying $\sum_{i\in \mathbb F_p}||A_i|-u|\le (2+\theta)p^{n-1}$.
\end{lemma}
\begin{proof}
We choose $u=|B_{\frac{p+1}2}|$. It suffices to show that $\sum_{i\in [p]}|\beta_i-\beta_{\frac{p+1}2}|\le (2+\theta)p$.
 Notice that
\begin{equation}
\begin{aligned}\label{eq_4}
\sum_{i\in [p]}|\beta_i-\beta_{\frac{p+1}2}|&= \beta_1+\cdots+\beta_{\frac{p-1}2}-(\beta_{\frac{p+1}2+1}+\cdots+ \beta_p)\\
 &=3(\beta_1+\cdots+\beta_{\frac{p-1}2})+2\beta_{\frac{p+1}2} +(\beta_{\frac{p+1}2+1}+\cdots+ \beta_p)-2|A|\slash p^{n-2}.
\end{aligned}
\end{equation}
From that $m\ge 4$, we have $\frac{p+1}2+1<p-\theta$. Together with $|A|\ge mp^{n-1}$, we have
\begin{equation}\label{eq_beta}
\begin{aligned}
\sum_{i\in [p]}|\beta_i-\beta_{\frac{p+1}2}| &\le 3(\beta_1+\cdots+\beta_{\frac{p-1}2})+2\beta_{\frac{p+1}2} +(\beta_{\frac{p+1}2+1}+\cdots+ \beta_{p-\theta})+\theta p-2mp\\
&\triangleq \beta+\theta p-2mp.
\end{aligned}
\end{equation}
Next we try to upper bound $\beta$ by applying Lemma~\ref{coro_beta_sum}.
Let $\mathcal{S}$ be the multiset by collecting $3(k+\ell)$ copies of each $i\in [\frac{p-1}2]$, $2(k+\ell)$ copies of ${\frac{p+1}2}$, and $(k+\ell)$ copies of each $i\in [\frac{p+1}2+1, p-\theta]$. Then $\mathcal{S}$ contains  $(2p-\theta)(k+\ell)$ numbers.
\begin{claim}\label{cl:si}
    When $k+\ell\ge 4$ and $m\ge 9+\lambda$, there exists a partition of  $\mathcal{S}$ into parts $S_i, i\in [2p-\theta]$, such that each $S_i$ has $k+\ell$ numbers, and the sum $\sigma_i$ of all numbers in $S_i$  is at least $p+k+\ell-1$.
%
\end{claim}



Suppose Claim~\ref{cl:si} is true, by Lemma~\ref{coro_beta_sum},
\[\beta=\frac{1}{k+\ell}\sum_{j\in \mathcal{S}}\beta_j=\frac{1}{k+\ell}\sum_{i\in [2p-\theta]}\sum_{j\in S_i}\beta_j\leq \frac{(2p-\theta)(p+k+\ell-2)}{k+\ell}.\]
%
Together with Eq.~(\ref{eq_beta}),
\[
\begin{aligned}
\sum_{i\in [p]}|\beta_i-\beta_{\frac{p+1}2}|& \le \frac{(2p-\theta)(p+k+\ell-2)}{k+\ell}+\theta p-2mp\\
& =\frac{(2p-\theta)[(k+\ell)m+\lambda]}{k+\ell}+2p+\theta p-\theta-2mp\\
& =(2+\theta)p- (m+1)\theta + \frac{(2p-\theta)\lambda}{k+\ell}.
\end{aligned}
\]
As $\theta\ge 2\lambda$ and $k+\ell>2+\lambda$, $(m+1)\theta(k+\ell)\ge 2\lambda p\ge(2p-\theta)\lambda$.
Hence we can get $\sum_{i\in [p]}|\beta_i-\beta_{\frac{p+1}2}|\le (2+\theta)p$. So it is left to prove
Claim~\ref{cl:si}.


\vspace{0.2cm}
\textbf{Proof of Claim~\ref{cl:si}.} We construct each $S_i$ explicitly by considering two cases.

When $k+\ell\ge 5$,
\[\frac{p+1}4= \frac{m(k+\ell)+\lambda+3}4 \ge m+2 \geq \frac{p+k+\ell-1}{k+\ell}.\]
So it suffices to require the average number in each $S_i$ to be at least $\frac{p+1}4$.
Let $\mathcal{S}_1$ be the collection of $3(k+\ell)$ copies of  $i\in [\frac{p-1}2]$.
First, we partition $\mathcal{S}_1$ into $\lfloor\frac{3(k+\ell)(p-1)}4\rfloor$ pairs by pairing each $i$ with ${\frac{p+1}2-i}$, where one copy of ${\frac{p+1}4}$ may be unpaired.
Then the average value in each pair is exactly $\frac{p+1}4$, and all the unpaired numbers in $\mathcal{S}$ are at least $\frac{p+1}4$. By evenly distributing the pairs and unpaired numbers into  $S_i$, $i\in [2p-\theta]$, we can get a partition with the required condition provided that the number of unpaired elements is at least $2p-\theta$. So it suffices to show that $|\mathcal{S}\setminus \mathcal{S}_1|=(2+\frac{p-1}2-\theta)(k+\ell)\geq 2p-\theta$.
In fact,
\[
\begin{aligned}
& (2+\frac{p-1}2-\theta)(k+\ell)-\left(2p-\theta\right)\\
=&\frac{p-1}2(k+\ell-4)-(k+\ell-1)(k+\ell+\lambda)>0.
\end{aligned}
\]Here, the last inequality holds since when
 $m\ge 9+\lambda$ and $k+\ell\ge 5$, we always have
 \begin{equation}\label{eqver}
   \frac{k+\ell-4}{k+\ell-1}>\frac{2(k+\ell+\lambda)}{(k+\ell)m+\lambda+1}.
 \end{equation}
 For verification of Eq. (\ref{eqver}), if $k+\ell\ge 6$,
\[\frac{k+\ell-4}{k+\ell-1}\ge \frac{2}{5}> \frac{2(k+\ell+\lambda)}{9(k+\ell)+\lambda(k+\ell+1)+1}\geq\frac{2(k+\ell+\lambda)}{(k+\ell)m+\lambda+1};\] and if $k+\ell= 5$,  then $\lambda\leq k+\ell-3\leq 2$, and hence
\[\frac{k+\ell-4}{k+\ell-1}=\frac{1}{4}>\frac{2\lambda+10}{6\lambda+46}\geq\frac{2(k+\ell+\lambda)}{(k+\ell)m+\lambda+1}.\]

When $k+\ell=4$, $\frac{p+1}4$ may be smaller than $m+2$ but $k=3$ is fixed.
As $p=4m+2+\lambda$ is a prime and $\lambda\in [0, k+\ell-3]$, we have $\lambda=1$ and hence $p=4m+3$,  $\frac{p-1}2=2m+1$, $\theta=k+\ell+2+\lambda=7$, and $p+k+\ell-1=4m+6$.
We construct $S_i, i\in [2p-\theta]$ from $S_1$ recursively, by letting each $S_i$ contain two of the  smallest numbers and two of the largest numbers in $\mathcal{S}\backslash\uplus_{j=1}^{i-1}S_i$. For example, both $S_1$ and $S_2$ consist of $1, 1, {4m-4}, {4m-4}$, but $S_3$ consists of $1, 1, {4m-5}, {4m-5}$.
By the construction,  when $i\le 2(p-\theta-\frac{p+1}{2})+4=4m-8$, at least two numbers in $S_i$ are out of $\mathcal{S}_1$, and hence $\sigma_i\geq 2(2m+2)+2\cdot1=4m+6$.
When $4m-8<i\le 2p-\theta$,  $S_i$ is contained in $\mathcal{S}_1$, with $\sigma_i$ non-decreasing with $i$. Hence, $\sigma_i\geq \sigma_{4m-7}=2(2m+1)+2\lceil\frac{4m-7}{6}\rceil \geq 4m+6$ as $m\ge 4$.
\end{proof}

\section{A proof of Theorem~\ref{thm_with_fourier}}\label{sec_proof_of_main_theorem}

In this section, we prove Theorem~\ref{thm_with_fourier}. We split into two cases depending on whether the weight $\omega(A)$ is big or small. Recall that $k>\ell\ge 1$ are integers, $p=m(k+\ell)+\lambda+2$, and
$\lambda\in[0, k+\ell-3]$.

\subsection{Small weight}
By Lemma~\ref{lem_ell_upper_bound}, we assume that $\omega\leq m+1.$
As $C_m$ usually have a short $(k, \ell)$-sum-free interval cover in $\mathbb F_p$ by Lemma~\ref{lem_stru_of_Cm}, the structures of $C_m$ are limited.
We apply different analysis when $C_m$ behaves differently. The following fact analyzes the simplest cases. 

\begin{fact}\label{factt1}
  Let   $k+\ell\ge 4$ and $m\ge 1$.
Let $A\subset \mathbb F_p^n$ be a  $(k, \ell)$-sum-free set of size at least $mp^{n-1}$.
If there exists a decomposition $(v, K)$ such that
\begin{itemize}
  \item[(1)] $\omega=m+1$ and $C_{m+1}$ is $(k, \ell)$-sum-free, then $A$ is trivial; and
  \item[(2)] $\omega=m$ and $C_m$ is  an interval, then  $A$ is either trivial or of type $1$.
\end{itemize}
\end{fact}
\begin{proof} When $\omega=m+1$,
if $C_{m+1}\subseteq \mathbb F_p$ is $(k, \ell)$-sum-free, then by Lemma~\ref{lem_m+1_interval}, $C_{m+1}$ is already one of the $\lceil\frac{\lambda+1}{2}\rceil$ standard $(k, \ell)$-sum-free intervals, and hence $A$ is contained in an extremal cuboid.

When $\omega=m$, $A=C_m\times K$.
By Lemmas~\ref{lem_contain_cret} and \ref{lem_m+1_interval}, $A$ is  trivial if and only if $C_m$ is contained in a $(k, \ell)$-sum-free interval of length $m+1$. Suppose $C_m=[a,a+m-1]$ for some $a\in \mathbb F_p$ and  $A$ is nontrivial. Then  both $[a-1,a+m-1]$ and $[a,a+m]$ can not be $(k, \ell)$-sum-free.  Let $i$ and $j$ denote the lengths of the two intervals connecting  $kC_m$ and $\ell C_m$ on the circle formed by $\mathbb F_p$.
More specifically, let $i$ be the length of the interval connecting the tail of $kC_m$, which is $k(a+m-1)$, and the head of $\ell C_m$, which is $\ell a$; and let $j$ be the length of the interval connecting the tail of $\ell C_m$, and the head of $k C_m$. Then $i\equiv \ell a-k(a+m-1)-1 \pmod p$ ($i$ can be zero), and
$i+j=p-|kC_m|-|\ell C_m|=k+\ell +\lambda$.  By symmetry, we can assume that $0\leq i\leq j\leq k+\ell +\lambda$. If $i\geq \ell$ and $j\geq k$, then it can be verified that at least one of the above two intervals is $(k, \ell)$-sum-free, leading to a contradiction. So either $i<\ell$ or $j<k$ (that is, $\ell+\lambda<i<k$). From $i\equiv \ell a-k(a+m-1)-1 \pmod p$, the form of $A$ here is exactly of type 1.
\end{proof}

\begin{theorem}\label{thm_without_fourier}
Let $\lambda\in [0, k+\ell-5]$ and $m\ge \max\{7,k+\ell-1,5T(\frac1{k+\ell})\}$.
Let $A\subset \mathbb F_p^n$ be a nontrivial $(k, \ell)$-sum-free set of size at least $mp^{n-1}$.
If there exists a decomposition $(v, K)$ such that $\omega\le m+1$, then $A$ is normal up to isomorphism.
\end{theorem}
\begin{proof}

Since $|A|\geq mp^{n-1}$, $mp \le |A|\slash p^{n-1}=\sum_{i=1}^{m+1}\beta_i$, thus $\beta_m\ge p\slash 2$. From $k\ge 3$, $k\beta_m\ge p+k-1$. By  Lemma~\ref{lem_stru_of_Cm},  $C_m$ is $(k, \ell)$-sum-free. 

If $\omega=m$, then $|A|=mp^{n-1}$ and each $B_i, i\in [m]$ is exactly $K$. When $\lambda\in [0, k+\ell-5]$ and $m\ge \max\{7, 5T(\frac 1{k+\ell})\}$, by Lemma~\ref{lem_stru_of_Cm}, up to isomorphism, $C_m$ is either an interval of length $m$, or a subset of a $(k, \ell)$-sum-free interval of length $m+1$. In each case, $A$ is either trivial or of type $1$ by Lemma~\ref{lem_m+1_interval} or Fact~\ref{factt1} (2).



When $\omega=m+1$,
if $C_{m+1}\subseteq \mathbb F_p$ is also $(k, \ell)$-sum-free, then $A$ is trivial by Fact~\ref{factt1} (1). 
Otherwise, we  show that  $A$ can only be of type 2. Since $kC_{m+1}\cap\ell C_{m+1}\neq \emptyset$, the following equation has at least one solution in $C_{m+1}$:
\begin{equation} \label{eq_1}
b_{r_1}+\cdots+ b_{r_k}=b_{s_1}+\cdots+ b_{s_\ell}.
\end{equation}
Let $S$ denote the multi-set by collecting the $k+\ell$ terms in Eq. (\ref{eq_1}).
Since $C_m$ is already $(k, \ell)$-sum-free, one element from the multiset $S$ must be $b_{m+1}$, which by definition is exactly the element in $C_{m+1}\slash C_m$.

First, we claim that $0\notin S$. Suppose $0 =b_j\in S$ for some $j$, then for any $i\in[m+1]\setminus\{j\}$, $(k-\ell) b_j+\ell b_i= 0+ \ell b_i= \ell b_i$. Hence by Fact~\ref{fact1},  we have
\[\left((k-\ell)B_{j}+\ell B_i\right)\cap \ell B_i=\emptyset.\]
So $p^{n-1}\geq |(k-\ell)B_{j}+\ell B_i|+|\ell B_i|\geq |B_i|+|B_{j}|$, which implies $\beta_i+\beta_{j}\leq p$ for any $i\in [m+1]\setminus\{j\}$. Then $mp\le \sum_{i=1}^{m+1}\beta_{i}\le m(p-\beta_{j})+\beta_{j}$, which forces $\beta_{j}$ to be zero. This contradicts to the hypothesis $\omega=m+1$.


Second, we claim that  the multiset $S$ contains exactly two different elements, one is $b_{m+1}$, and the other is $b_s$ for some $s\in [m]$.
If the multiset $S$ contains only one distinct element which is $b_{m+1}$, Eq.~(\ref{eq_1}) shows that $kb_{m+1}=\ell b_{m+1}$. As $p$ is prime, $kb_{m+1}=\ell b_{m+1}$ implies that $b_{m+1}=0$, contradicting to the first claim.
If the multiset $S$ contains at least three different elements,  say, $b_s, b_t, b_{m+1}$ for some $s\ne t\in [m]$, by Fact~\ref{fact1}, 
Eq.~(\ref{eq_1}) gives
\[
p^{n-1}+ (k+\ell-2)p^{n-2}
  \ge \sum_{i}^{k}|A_{b_{r_i}}|+\sum_{j}^{\ell}|A_{b_{s_j}}|\geq |B_s|+|B_t|+|B_{m+1}|.
\]
Then $mp\le \sum_{i=1}^{m+1}\beta_i=\beta_s+\beta_t+\beta_{m+1}+\sum_{i\in [m]\backslash\{s, t\}}\beta_i\le p+k+\ell-2+(m-2)p$, which is impossible because $p\ge k+\ell-1$.


Third, we  claim that in Eq. (\ref{eq_1}), $\{b_{r_i}, i\in [k]\}\cap\{b_{s_j}, j\in [\ell]\}= \emptyset$, that is, the two different elements $b_s$ and $b_{m+1}$ of $S$ are separated by the equal sign. We prove this claim by contradiction.
If one of $b_s$ or $b_{m+1}$ is contained on both sides of Eq.~(\ref{eq_1}) and appears at least three times in $S$, by replacing one copy of this element on each side of the equation with a third element in $C_{m+1}$, we can derive another solution with at least three distinct elements in $S$, which contradicts to the second claim.
Otherwise, we have $\ell=1$ and Eq.~(\ref{eq_1}) of the form $(k-1)b_{m+1}+b_s=b_s$ (or $b_s$ and $b_{m+1}$ are swapped). However, this leads to $b_{m+1}=0$ or $b_s=0$,  contradicting to the first claim.

We conclude that  Eq.~(\ref{eq_1}) is either $kb_s=\ell b_{m+1}$ or $k b_{m+1}=\ell b_s$. The former one is impossible since  $kB_s\cap \ell B_{m+1}=\emptyset$ by Fact~\ref{fact1}, but  $kB_s=K$ due to $|B_s|\ge |B_{m}|> p^{n-1}\slash 2$ and $k> 2$. So we have $k b_{m+1}=\ell b_s$ and hence $k B_{m+1}\cap \ell B_s =\emptyset$ by Fact~\ref{fact1}.  When $b_{m+1}$ is fixed, $b_s$ is uniquely determined by Eq.~(\ref{eq_1}). So $kb_{m+1}=\ell b_{s}$ is the only solution to Eq.~(\ref{eq_1}), which means
$ k C_{m+1}\cap \ell C_{m+1} $ is of size one.
%
%
%
Then
\[
\begin{aligned}
p+1&\ge |\ell C_{m+1}|+|k C_{m+1}|\\
& \ge \frac{\ell-1}2|2C_{m+1}|+|C_{m+1}|-\frac{\ell-1}2
+ \frac{k-1}2|2C_{m+1}|+|C_{m+1}|-\frac{k-1}2
\end{aligned}
\] with the second inequality holding by Theorem~\ref{thm_Kneser}.
By plugging in $|C_{m+1}|=m+1$ and $p=(k+\ell)m+2+\lambda$, we have
\[|2C_{m+1}|\le 2m+ \frac{2\lambda+k+\ell}{k+\ell-2}.\]
As $\lambda\leq k+\ell-4$ and $|2C_{m+1}|$ is an integer, $|2C_{m+1}|\le 2m+2$.

If $|2C_{m+1}|=2m+2$, from Theorem~\ref{thm_3k-4_r=0}, after a suitable isomorphism, $C_{m+1}$ is  $1$-holed from an interval of length $m+2$, denoted  by $I=[a, a+m+1]$. If the element in $I\backslash C_{m+1}$ is not in $\{ a+1, a+m\}$, it is easy to check that $|2C_{m+1}|=|2I|=2m+3$, contradicting to $|2C_{m+1}|=2m+2$. As $|kC_{m+1}\cap \ell C_{m+1}|=1$ and $I\backslash C_{m+1}\subset\{ a+1, a+m\}$, we get $|kI\cap \ell I|\le 2$. So $p\ge |kI\cup \ell I|\ge |kI|+|\ell I|-2=(k+\ell)m+(k+\ell)$, which leads to the contradiction as $\lambda\leq k+\ell -3$.

If $|2C_{m+1}|=2m+1$, by Theorem~\ref{thm_Cauchy_daven} and Theorem~\ref{thm_Vosper}, $C_{m+1}$ must be an AP of length $m+1$, and without loss of generality we can set it as an interval. Denote $C_{m+1}=[a, a+m]$ for some $a\in\mathbb F_p$, and hence $kC_{m+1}=[ka, ka+km]$ and $\ell C_{m+1}=[\ell a, \ell a+\ell m]$. As $ |k C_{m+1}\cap \ell C_{m+1}|=1 $, we have $ka+km=\ell a$ by symmetry. So $a=-km(k-\ell)^{-1}$. 
 Consider the two special subsets $A_a$ and $A_{a+m}$ in $K$. As $k(a+m)=\ell a$ and $A$ is $(k, \ell)$-sum-free, $k A_{a+m}\cap \ell A_a=\emptyset$ by Fact~\ref{fact1}.
Moreover, since $|A|\ge mp^{n-1}$, we have $|A_{a+m}|+|A_a|\ge p^{n-1}$, with equality if and only if $|A|=mp^{n-1}$.  So $p^{n-1}\geq |kA_{a+m}|+|\ell A_a|\geq |A_{a+m}|+|A_a|\ge p^{n-1} $, which implies that $|kA_{a+m}|=|A_{a+m}|$, $|\ell A_a|=| A_a|$ and $|A_{a+m}|+|A_a|= p^{n-1}$. Hence \[\text{$kA_{a+m}=A_{a+m}$, $\ell A_a= A_a$,
$A_{a+m}\sqcup A_a=K$, $A_i=K$ for $i\in [a+1,a+m-1]$ and $|A|=mp^{n-1}$.}\]
 Denote $H=Sym(kA_{a+m})$ which is a subgroup of $K$ and hence a linear subspace.
From Theorem~\ref{thm_Kneser}, $|H+A_{a+m}|\geq |A_{a+m}|= |kA_{a+m}|\ge k|H+A_{a+m}|-(k-1)|H|$. So $|H|\ge |H+ A_{a+m}|$, and hence $|H|= |H+ A_{a+m}|$. This means that $A_{a+m}$ lies in one coset of $H$. As $A_{a+m}= kA_{a+m}\supseteq H$, $A_{a+m}=H+x$ for some  $x\in K$.
As $a=-m k (k-\ell)^{-1}$, $a+m\ne 0$. Without loss of generality we can assume that $x=0$, because otherwise we do the following isomorphism: for any $y\in A_i$, map $y$ to $y-ix(a+m)^{-1}$.
This map does not change $A_i$ for each $i\in [a+1,a+m-1]$ but map $A_{a+m}$ from $H+x$ to $H$, and  hence $A_a= K\backslash H$. This is exactly the form of type 2.
%
\end{proof}

\subsection{Big weight by Fourier analysis}\label{sec_fourier}
We begin with some notations in discrete Fourier analysis.

 Given a finite additive abelian group $G$, let $\widehat G$ be the group of all homomorphisms from $G$ to the unit circle in the complex plane. Each member  $\chi\in\widehat G$ is known as a {\it character} of $G$. Specially, the zero character of $G$ is denoted by $\chi_0$, which maps all $g\in G$ to $1$. It is known that $\widehat{G}\cong G$~\cite{kurzweil2004theory} and hence $|\widehat G|=|G|$.

Write $L(G)$ for the vector space of all complex-valued functions on $G$. Define 
$$\langle f, h\rangle=\frac 1{|G|}\sum_{g\in G}f(g)\overline{h(g)}, \text{ for any }f, h\in L(G).$$
Then all characters of $G$ form an orthonormal basis of $L(G)$ with the above dot product. Given a function $f\in L(G)$, define the {\it Fourier transform} of $f$ as a function $\hat f: \widehat G\to \mathbb C$ by $$\hat f(\chi)=\langle f, \chi\rangle=\frac 1{|G|}\sum_{g\in G}f(g)\overline{\chi(g)},~\chi\in\widehat G. $$
Hence, we have
\begin{equation}\label{eq_fh}
\langle f, h\rangle=\sum_{\chi\in\widehat{G}}\hat{f}(\chi)\overline{\hat{h}(\chi)}, 
\end{equation} and the corresponding {\it Plancherel formula} $\Vert f\Vert^2=\sum_{\chi\in\widehat G}|\hat f(\chi)|^2.$
We refer to $\hat f(\chi)$ as the {\it Fourier coefficient} of $f$ at the character $\chi$. The coefficient at the zero character is called the zero Fourier coefficient.

Given two functions $f, h\in L(G)$, the {\it convolution} of $f$ and $h$, denoted as $f * h$, 
is defined as follows,
\[
(f * h)(g)=\frac{1}{|G|}\sum_{b\in G}f(g-b)h(b)=\frac{1}{|G|}\sum_{a+b=g}f(a)h(b).
\]The following folklore identity about convolution can be found in  \cite[Section 4.1]{Tao2006},
\begin{equation}\label{eq_hat_convo}
    \widehat{f*h}=\hat{f}\cdot\hat{h}.
\end{equation}
Similarly, for any integer $i\ge 3$ and $f\in L(G)$, define the convolution of $i$ copies of $f$ by $*_if=(*_{i-1}f)*f$ with  $*_2f=f*f$. By using Eq. (\ref{eq_hat_convo}) recursively, we get for any integer $i\ge 2$,
\[
\widehat{*_{i}f}=(\hat{f})^i.
\]

Finally, for any subset $A$ of $G$, define the {\it indicator function} of $A$, denoted by ${\bm 1}_A\in L(G)$, as follows:
\[
{\bm 1}_A(g)=\left\{
\begin{aligned}
    1, & \text{~if $g\in A$;}\\
    0, & \text{~otherwise.}
\end{aligned}\right.\]
The first lemma below shows that if $A$ is a $(k, \ell)$-sum-free subset of a finite abelian additive group, then one of the nonzero Fourier coefficients of the indicator function must have a big norm.



\begin{lemma}\label{lem_exists_chi_small}
Let $A\ne\emptyset$ be a $(k, \ell)$-sum-free subset of a finite abelian additive group $G\ne 0$. Then there exists at least one character $\chi\ne \chi_0$ such that $|\widehat{\bm 1_A}(\chi)|\ge (\frac{\alpha^{k+\ell-1}}{1-\alpha})^{\frac 1{k+\ell-2}}$, where $\alpha=|A|\slash|G|$ is the density.
\end{lemma}
\begin{proof}
Consider the inner product of convolutions of the indicator of $A$:
\[\langle *_k {\bm 1}_A, *_\ell {\bm 1}_A\rangle=\frac 1{|G|}\sum_{g\in G}*_k {\bm 1}_A(g) \cdot \overline{*_\ell {\bm 1}_A(g)}=\frac 1{|G|}\sum_{g\in G}*_k {\bm 1}_A(g) \cdot {*_\ell {\bm 1}_A(g)},\]
where the second equality holds because ${\bm 1}_A$ is real-valued.
By the definition of convolution, for any $i\geq 1$,
\[*_i {\bm 1}_A(g)=\frac{1}{|G|^{i-1}}|\{(x_1, \ldots, x_i)\in A^i: x_1+\cdots +x_i=g\}|.\]
Since $A$ is $(k, \ell)$-sum-free, for any $g\in G$,
either $*_k {\bm 1}_A(g)=0$ or $*_\ell {\bm 1}_A(g)=0$. Hence $\langle *_k {\bm 1}_A, *_\ell {\bm 1}_A\rangle=0$.

On the other hand, by Eq. (\ref{eq_fh}) and the convolution identity,
\begin{equation}\label{eq_2}
0=\langle *_k {\bm 1}_A, *_\ell {\bm 1}_A\rangle= \sum_{\chi\in\widehat{G}} \widehat{*_k {\bm 1}_A}(\chi) \overline{\widehat{*_\ell {\bm 1}_A}(\chi)}
=\sum_{\chi\in\widehat{G}}(\widehat{\bm 1_A}(\chi))^{k}(\overline{\widehat{\bm 1_A}(\chi)})^{\ell}
=\sum_{\chi\in\widehat G}(\widehat{\bm 1_A}(\chi))^{(k-\ell)}|\widehat{\bm 1_A}(\chi)|^{2\ell}.
\end{equation}
From the definition of Fourier transform, $\widehat{\bm 1_A}(\chi_0)=\langle \bm 1_A, \chi_0\rangle=\frac 1{|G|}\sum_{g\in A}\chi_0(g)=\alpha$. Hence $(\widehat{\bm 1_A}(\chi_0))^{(k-\ell)}|\widehat{\bm 1_A}(\chi_0)|^{2\ell}=\alpha^{k+\ell}$ and by plugging this into Eq.~(\ref{eq_2}), we have
\begin{equation}\label{eq_3}
\sum_{\chi\ne\chi_0}(\widehat{\bm 1_A}(\chi))^{(k-\ell)}|\widehat{\bm 1_A}(\chi)|^{2\ell}=-\alpha^{k+\ell}.
\end{equation}

From Plancherel's formula, we have $\alpha=\langle{\bm 1_A}, {\bm 1_A}\rangle=
\sum_{\chi\in\widehat G}|\widehat{\bm 1_A}(\chi)|^2$, and hence
\[\sum_{\chi\ne \chi_0}|\widehat{\bm 1_A}(\chi)|^2=\alpha-\alpha^2.\]
Combining with Eq.~(\ref{eq_3}), we have
\[
\sum_{\chi\ne\chi_0}|\widehat{\bm 1_A}(\chi)|^2\left((1-\alpha)(\widehat{\bm 1_A}(\chi))^{(k-\ell)}|\widehat{\bm 1_A}(\chi)|^{2\ell-2}+ \alpha^{k+\ell-1}\right)=0,
\]
which means there must exist at least one $\chi\in\widehat G\backslash\{\chi_0\}$, such that the real part
\[
\mathfrak{Re}((\widehat{\bm 1_A}(\chi))^{(k-\ell)}|\widehat{\bm 1_A}(\chi)|^{2\ell-2})\le -\alpha^{k+\ell-1}\slash(1-\alpha).
\]
For this special $\chi$, one must have
$$|\widehat{\bm 1_A}(\chi)|\ge \left|\mathfrak{Re}((\widehat{\bm 1_A}(\chi))^{(k-\ell)}|\widehat{\bm 1_A}(\chi)|^{2\ell-2})\right|^{1\slash(k+\ell-2)}\ge\left(\frac{\alpha^{k+\ell-1}}{1-\alpha}\right)^{1\slash(k+\ell-2)}.$$
Hence the lemma is proved.
\end{proof}

Finally, we can show a large subset with big weight under any decomposition can not be $(k, \ell)$-sum-free.

\begin{lemma}\label{lem_summation_fourier}
Let $k+\ell\ge 4$, $\lambda\in [0, k+\ell-3]$, and $m\ge \max\{9+\lambda, (2+\theta)(1+\frac1{k+\ell})(k+\ell)^{1\slash(k+\ell-2)}\}$.
If  $A\subset \mathbb F_p^n$ satisfies  $|A|\geq mp^{n-1}$ and $\omega(A)> p-\theta$ under any  decomposition, then $A$ is not $(k, \ell)$-sum-free. 
\end{lemma}
\begin{proof}
On the contrary, if $A$ is $(k, \ell)$-sum-free, by Lemma~\ref{lem_exists_chi_small}, there exists some $\chi\ne \chi_0$ such that $|\widehat{\bm 1_A}(\chi)|\ge (\frac{\alpha^{k+\ell-1}}{1-\alpha})^{\frac 1{k+\ell-2}}$ with $\alpha=\frac{|A|}{p^n}\geq \frac{m}{p}$. Let $\zeta=e^{\frac{2\pi \sqrt{-1}}p}$. As $\chi\ne \chi_0$, there exists some $v\in \mathbb F_p^n$ such that $\chi(v)=\zeta^{-1}$. Let $K=\ker(\chi)$, which is a subspace of dimension $n-1$ not containing $v$.

Under the decomposition $(v,K)$, each $g\in A$ can be written as $g=iv+ g'$ for some $i\in {\mathbb F_p}$ and $g'\in A_i\subset K$ in a unique way.  Then $\overline{\chi(g)}=\zeta^i$, and
\[\widehat{\bm 1_A}(\chi)=\langle \bm 1_A, \chi\rangle=\frac 1{p^n}\sum_{g\in A}\overline{\chi(g)}=\frac 1{p^n}\sum_{i\in {\mathbb F_p}}|A_i|\zeta^i = \frac 1{p^n}\sum_{i\in {\mathbb F_p}}(|A_i|-u)\zeta^i\]
for any constant $u$ since  $\sum_{i\in {\mathbb F_p}}\zeta^i=0$. Since $\omega> p-\theta$, by setting $u$ to be the integer in Lemma~\ref{lem_four_coeff_small}, we have
\[
|\widehat{\bm 1_A}(\chi)|\le \frac 1{p^n}\sum_{i\in {\mathbb F_p}}||A_i|-u|\le  \frac 1{p^n}(2+\theta)p^{n-1}=\frac{2+\theta}p.
\]
As $m\ge (2+\theta)(1+\frac1{k+\ell})(k+\ell)^{1\slash(k+\ell-2)}$, $\frac{2+\theta}p< \frac{2+\theta}{m(k+\ell)}\le \frac 1{k+\ell+1}\left(\frac1{k+\ell}\right)^{\frac 1{k+\ell-2}}$.

Notice that the function $(\frac{\alpha^{k+\ell-1}}{1-\alpha})^{\frac 1{k+\ell-2}}$ is monotonically increasing with $\alpha$. As  $m\geq 9+\lambda$, $\alpha\ge\frac m p\geq \frac {m}{(k+\ell)m+2+\lambda}\ge \frac 1{k+\ell+1}$, and hence $(\frac{\alpha^{k+\ell-1}}{1-\alpha})^{\frac 1{k+\ell-2}}\ge (\frac{{\frac 1 {k+\ell+1}}^{k+\ell-1}}{1-1\slash(k+\ell+1)})^{\frac 1{k+\ell-2}} = \frac 1{k+\ell+1}\left(\frac1{k+\ell}\right)^{\frac 1{k+\ell-2}}$. So $ |\widehat{\bm 1_A}(\chi)|< (\frac{\alpha^{k+\ell-1}}{1-\alpha})^{\frac 1{k+\ell-2}}$, which is a contradiction.
\end{proof}

Lemma~\ref{lem_summation_fourier} shows that any $(k, \ell)$-sum-free set $A\subset \mathbb F_p^n$ of size at least $ mp^{n-1}$ must have a small weight  $\omega(A)\leq m+1$ under some decomposition. Then combining
 Theorem~\ref{thm_without_fourier}, we prove Theorem~\ref{thm_with_fourier} completely.

\section{When $\lambda\in\{0, 1\}$}\label{sec_unique_extrem}

When $\lambda\in\{0, 1\}$, the extremal structure of a $(k, \ell)$-sum-free set in $\mathbb F_p^n$ of size $(m+1)p^{n-1}$ is a unique cuboid up to isomorphism by Theorem~\ref{thm_k_l_max_structure}. In this section, we study this special case for all $k>\ell\geq 1$ and $(k,\ell)\neq (2,1)$.  When $\lambda=0$ and  $k+\ell\ge 5$, and $\lambda=1$ and $k+\ell\ge 6$, a nontrivial $(k, \ell)$-sum-free set of size at least $mp^{n-1}$  must be normal by Theorem~\ref{thm_with_fourier}. So focusing on $\lambda\in\{0, 1\}$, the only cases left are  when $(k+\ell,\lambda)=(5,1)$ or $(4,1)$. Here, the $(4,0)$ case is not applicable since $p$ is a prime.

The idea is similar to the proof of  Theorem~\ref{thm_with_fourier}, but we consider $(k, \ell, \lambda)\in\{(4, 1, 1), (3, 2, 1),(3, 1, 1)\}$ separately.
First, similar to Section~\ref{sbsec:c}, we show that the weight of $A$ is either too big or too small: if there exists one decomposition such that the weight is not too big (not reaching $p-\theta$), then $\omega(A)\le m+2$.
Second, similar to Section~\ref{sbsec:a}, from that $C_m$ is $(k, \ell)$-sum-free, we characterize all the possible patterns of $C_m$ up to isomorphism. Finally, from each pattern of $C_m$, we carefully determine all possible structures of $A$ exhaustively. In this final part,
 very detailed structure analysis will be applied where new methods are needed to efficiently check whether $A$ satisfying certain property exists or not. Some common criterions frequently used are listed as follows. For convenience, we call  an equation $r_1+\cdots +r_k= s_1+\cdots + s_\ell$ a \emph{$(k, \ell)$-sum in $S$} if all terms in this equation are in $S$.

\begin{lemma}\label{lem_contradiction_list_m+1}
Let $A\subset \mathbb F_p^n$ be a $(k, \ell)$-sum-free set of size at least $mp^{n-1}$, such that under  certain decomposition $\omega(A)= m+1$. Then each $(k, \ell)$-sum in $C_{m+1}$ involves at most two distinct values.
\end{lemma}
\begin{proof}
When $n=1$, the statement is trivial as $C_{m+1}$ is $(k, \ell)$-sum-free.
Assume that $n\ge 2$. Suppose  $r_1+\cdots +r_k= s_1+\cdots + s_\ell$ is a $(k, \ell)$-sum in $C_{m+1}$, then by  Fact~\ref{fact1},
     $\sum_{i=1}^k|A_{r_i}|+\sum_{j=1}^\ell|A_{s_j}|\le p^{n-1}+ (k+\ell-2)p^{n-2}$. If the $(k, \ell)$-sum has at least three distinct values being $t_1, t_2, t_3$, then from $\omega= m+1$ and $|A|\ge mp^{n-1}$, $\sum_{i=1}^3|A_{t_i}| \ge 2p^{n-1}$. But $\sum_{i=1}^3|A_{t_i}|\le \sum_{i=1}^k|A_{r_i}|+\sum_{j=1}^\ell|A_{s_j}|\le p^{n-1}+ (k+\ell-2)p^{n-2}$, which leads to a contradiction because $k+\ell-2< p$.
\end{proof}

\begin{lemma}\label{lem_contradiction_list}
Let $A\subset \mathbb F_p^n$ be a $(k, \ell)$-sum-free set of size at least $mp^{n-1}$, such that under  certain decomposition $\omega(A)= m+2$. When $m\ge 6$,  all of the followings hold.
\begin{itemize}
\item[(1)] Each $(k, \ell)$-sum in $C_{m+2}$  involves at most three distinct values.
\item[(2)] Every two $(k, \ell)$-sums in $C_{m+2}$ involve at most four distinct values in total.
\item[(3)] If two $(k, \ell)$-sums in $C_{m+2}$ involve four distinct values, then each value in $C_m$ appears at most once in these two sums.
\end{itemize}
%
%
%
\end{lemma}
\begin{proof} Similarly, we only need to consider the case $n\ge 2$. From $\omega= m+2$ and $|A|\ge mp^{n-1}$, we have $\sum_{i=1}^s|A_{t_i}|\ge (s-2)p^{n-1}$ for any $s$ distinct values $t_1,\ldots, t_s$ in $C_{m+2}$, where $s\geq 4$. Also, each $(k,\ell)$ sum $r_1+\cdots +r_k= s_1+\cdots + s_\ell$  in $C_{m+2}$ gives $\sum_{i=1}^k|A_{r_i}|+\sum_{j=1}^\ell|A_{s_j}|\le p^{n-1}+ (k+\ell-2)p^{n-2}$ by Fact~\ref{fact1}. All cases are proved by contradiction as in Lemma~\ref{lem_contradiction_list}.

For (1), it is easy since $\sum_{i=1}^4|A_{t_i}|\ge 2p^{n-1}$  if there are four different values $t_1, \ldots, t_4$ in the $(k,\ell)$-sum.
For (2), suppose  two $(k,\ell)$-sums $r_1+\cdots +r_k= s_1+\cdots + s_\ell$ and $r_1'+\cdots +r_k'= s_1'+\cdots + s_\ell'$ in $C_{m+2}$ involve five different values $t_1, \ldots, t_5$. Then
$3p^{n-1}\le \sum_{i=1}^5|A_{t_i}|\leq \sum_{i=1}^k(|A_{r_i}|+|A_{r_i'}|)+\sum_{j=1}^\ell(|A_{s_j}|+|A_{s_j'}|)\le 2p^{n-1}+ 2(k+\ell-2)p^{n-2}$, which leads to a contradiction as $2(k+\ell-2)<p$. For (3), if the two mentioned $(k,\ell)$-sums involve four different values $t_1, \ldots, t_4$, and $t_1\in C_m$ appears at least twice. Then  $ 2p^{n-1}\le \sum_{i=1}^4|A_{t_i}| \le \sum_{i=1}^k(|A_{r_i}|+|A_{r_i'}|)+\sum_{j=1}^\ell(|A_{s_j}|+|A_{s_j'}|)\le 2p^{n-1}+ 2(k+\ell-2)p^{n-2}$ as before. As $t_1$ occurs at least twice,
$|A_{t_1}|\le \sum_{i=1}^k(|A_{r_i}|+|A_{r_i'}|)+\sum_{j=1}^\ell(|A_{s_j}|+|A_{s_j'}|)-\sum_{i=1}^4|A_{t_i}|\le (2p^{n-1}+ 2(k+\ell-2)p^{n-2})-2p^{n-1}= 2(k+\ell-2)p^{n-2}$. As $t_1\in C_m$, $|B_m|\le |A_{t_1}|$ and hence $|A|=\sum_{i=1}^{m-1}|B_{i}|+|B_m|+|B_{m+1}|+|B_{m+2}|\le (m-1)p^{n-1}+3|B_m|\le (m-1)p^{n-1}+6(k+\ell-2)p^{n-2}$. This leads to a contradiction since $6(k+\ell-2)<p$ when $m\ge 6$.
\end{proof}

\subsection{$(k+\ell,\lambda)=(5,1)$}\label{sbsc:51}

Recall that when $k+\lambda=5$ and $\lambda=1$, we have $p=5m+3$ and $\theta=8$. In the spirit of Lemmas \ref{lem_l_upper_default} and \ref{lem_ell_upper_bound}, we have the following lemma to describe the upper bound of $\omega$ and the possible forms of $C_m$, whose proof is given in Appendix~\ref{ap1}.

\begin{lemma}\label{lem_51_whole}
Let $(k+\ell,\lambda)=(5,1)$ and $m\ge \max\{5T(1\slash5),21\}$. Suppose $A\subset \mathbb F_p^n$ is a $(k, \ell)$-sum-free  subset with $|A|\ge mp^{n-1}$. Then there exists a decomposition of $\mathbb F_p^n$ such that $\omega(A)\le m+2$. Further, under this decomposition, we have the following properties.
\begin{itemize}
  \item[(1)] $C_m$ is  $(k, \ell)$-sum-free. Moreover,
\begin{equation}\label{eq_8}
kC_m\cap \ell C_{\omega}=\emptyset
\end{equation}
except when $(k,\omega)= (3,m+2)$, we have
\begin{equation}\label{eq_9}
3C_{m-1}\cap 2 C_{\omega}=\emptyset.
\end{equation}
  \item[(2)] Up to isomorphism, $C_m$ must be one of the following three forms,
\begin{itemize}
\item[(i)] an interval of length $m$;
\item[(ii)] covered by  a $(k, \ell)$-sum-free interval of length $m+1$; or
\item[(iii)] 
$C_m=\{a\}\cup [a+2, a+m-1]\cup\{a+m+1\}$ with $a=m\slash 3\equiv 2m+1 \pmod p$.
\end{itemize}
\end{itemize}
\end{lemma}

The following lemma shows the solution when the form (iii) of $C_m$ in Lemma~\ref{lem_51_whole} (2) occurs and gives the first abnormal form of $A$.

\begin{lemma}\label{lem_5_ex_1}
Let  $(k+\ell,\lambda)=(5,1)$ and $m\ge \max\{5T(1\slash5),21\}$. Suppose $A\subset \mathbb F_p^n$ is a $(k, \ell)$-sum-free  subset with $|A|\ge mp^{n-1}$ and define $a=2m+1$. Then under a decomposition with $\omega\le m+2$, the set $C_m=\{a\}\cup [a+2, a+m-1]\cup\{a+m+1\}$   up to isomorphism if and only if $A$ is of type 3.

\end{lemma}
\begin{proof}The sufficiency is trivial. We prove the necessity.
When $C_m= \{a\}\cup [a+2, a+m-1]\cup\{a+m+1\}$ with $a=2m+1$, it is easy to check that $kC_m\cup \ell C_m=\mathbb F_p$. We split into two cases.

If $(k,\omega)\neq (3,m+2)$, by Lemma~\ref{lem_51_whole} (1), we have
$kC_m\cap\ell C_{\omega}=\emptyset$.
As we already have $kC_m\cup \ell C_m=\mathbb F_p$, $\ell C_{\omega}=\ell C_m$. If $\omega> m$, there exists at least one more element $b\in C_{\omega}\backslash C_m$. Hence $b+(\ell-1)C_m\subset \ell C_{\omega}= \ell C_m$. If $\ell=1$, this gives $b\in C_m$, leading to a direct contradiction. If $\ell=2$, as $C_m= \{a\}\cup [a+2, a+m-1]\cup\{a+m+1\}$, the only way for $b$ satisfying $b+C_m\subset 2C_m$ is  $b\in C_m$, which also leads to a contradiction. Hence $\omega=m$ and $A\cong C_m\times\mathbb F_{p}^{n-1}$.

If $(k,\omega)= (3,m+2)$, by Lemma~\ref{lem_51_whole} (1)  we have $3C_{m-1}\cap 2 C_{m+2}=\emptyset$.
As $m\ge 8$, by all possible choices of $C_m\backslash C_{m-1}$, we must have $[3a+6, 3a+3m-3]\subseteq 3C_{m-1}$, and hence $2C_{m+2}\subseteq [3a+3m-2, 3a+5]=[4m-2, 6m+8]$.
Actually from this we can conclude that $C_{m+2}\subset [2m-1, 3m+4]$ by a simple two-step argument. First, from that $C_m=[2m+1, 3m+2]\backslash\{2m+2, 3m+1\}$ and $C_{m+2}+C_m\subset2C_{m+2}\subset[4m-2, 6m+8]$, we have $C_{m+2}\subset[2m-3, 3m+6]$; then from that $2\cdot C_{m+2}\subset2C_{m+2}\subset [4m-2, 6m+8]$, $C_{m+2}\subset [2m-1, 3m+4]$. 
As a consequence, $C_{m+2}\backslash C_m\subset\{2m-1, 2m, 2m+2, 3m+1, 3m+3, 3m+4\}$. Next, we show that this is impossible.
If $C_{m+2}\backslash C_m=\{2m+2, 3m+1\}$, then $C_{m+2}=[2m+1, 3m+2]$ and we have the following two $(3, 2)$-sums in $C_{m+2}$:
\begin{equation*}
    3(3m+2)=(2m+1)+(2m+2) \text{ and }
    (3m+1)+2(3m+2)=2(2m+1),
\end{equation*}
which contradict to Lemma~\ref{lem_contradiction_list} (3).
If $C_{m+2}\backslash C_m\neq\{2m+2, 3m+1\}$, then at least one in $\{2m-1, 2m, 3m+3, 3m+4\}$ should be in $C_{m+2}$. As $C_m=-C_m$, by symmetry we only need to show that $2m\in C_{m+2}$ or $2m-1\in C_{m+2}$ are impossible. If $2m\in C_{m+2}$, consider the following $(3, 2)$-sum
in $C_{m+2}$:
\begin{equation*}
    3m+2(3m+2)=2m+(2m+1);
\end{equation*}
if $2m-1\in C_{m+2}$, consider the following $(3, 2)$-sum
in $C_{m+2}$:
\begin{equation*}
    2(2m-1)+(2m+1)=3m+(3m-1).
\end{equation*}
Both of them lead to  contradictions to Lemma~\ref{lem_contradiction_list} (1).
%
\end{proof}

Next, the two cases $(k, \ell, \lambda)=(4, 1, 1)$ and $(3,2,1)$ are handled separately in the following two lemmas.

\begin{lemma}\label{lem_(4,1,1)_finial}
Let $(k, \ell, \lambda)=(4, 1, 1)$  and $m\ge \max\{5T(1\slash5),21\}$. Suppose $A\subset\mathbb F_p^n$ is a {nontrivial} $(4, 1)$-sum-free set with $|A|\ge mp^{n-1}$. 
Then $A$ is  abnormal if and only if $A$ is of type 3 or 4.
\end{lemma}
\begin{proof} By Lemma~\ref{lem_51_whole}, $\omega(A)\le m+2$ under a certain decomposition $(v,K)$.
By Lemma~\ref{lem_5_ex_1}, $A$ is of type 3 if and only if $C_m$ is isomorphic to the form (iii) in Lemma~\ref{lem_51_whole} (2). So we only need to prove that, when 
the forms (i) or (ii) of Lemma~\ref{lem_51_whole} (2) apply,
$A$ is either normal or of type 4.

\vspace{0.2cm}

\textbf{Case (1):} $C_m$ is an interval. If $\omega=m$, then $A$ is either a subset of an extremal cuboid, or $A$  is of type 1 by Fact~\ref{factt1}. Next assume $\omega=m+1$ or $m+2$.

Write $C_m=[a, a+m-1]$, then $4C_m=[4a,4(a+m-1)]$. As $4C_m \cap C_m=\emptyset$, on the circle formed by $\mathbb F_p$, denote $i$ the length of the interval connecting the tail $4(a+m-1)$ of $4C_m$ and the head $a$ of $C_m$, and denote $j$ the length of the one connecting the tail $a+m-1$ of $C_m$ and the head $4a$ of $4C_m$ (the intervals might be empty).
As $|C_m|+|4C_m|=5m-3=p-6$, we have $i+j=6$. So there are only four cases of $(i,j)$ by symmetry, that is, $(6, 0)$ case, $(5, 1)$ case, $(4, 2)$ case and $(3, 3)$ case.
In each case the value of $a$ can be uniquely determined. See the following table.
\begin{center}
\begin{tabular}{ |c|c|c|c|c| }
 \hline
 \textbf{cases:} & (6, 0) & (5, 1) & (4, 2) & (3, 3) \\
 \hline
 \textbf{requirements:} & $a+m=4a$ & $a+m+1=4a$ & $a+m+2=4a$ & $a+m+3=4a$ \\
\textbf{ values of $a$:} & $m\slash 3$ & $(m+1)\slash 3$ & $(m+2)\slash 3$ & $1+m\slash 3$ \\
 \hline
\end{tabular}
\end{center}
As $4C_m\cap C_m=\emptyset$ and $4C_m \cap C_{\omega}=\emptyset$, $C_{\omega}\backslash C_m\subset \mathbb F_p\backslash \{C_m\cup 4C_m\}\triangleq W$, which is a subset of size $6$.
However, in most cases, contradictions can be made by Lemmas~\ref{lem_contradiction_list_m+1} and~\ref{lem_contradiction_list} as below.


In the $(6, 0)$ case, $W=[a-6, a-1]$. When $\omega=m+1$, denote $C_{\omega}\backslash C_m= \{a-i\}$ for some $i\in [6]$ and hence we have a $(4,1)$-sum $(a-i)+3a=a+m-i$ in $C_{m+1}$, which  contradicts to Lemma~\ref{lem_contradiction_list_m+1} as $m\geq 7$. When $\omega=m+2$, denote $C_{\omega}\backslash C_m= \{a-i, a-j\}$ with $i\ne j\in [6]$ and hence we have two $(4,1)$-sums $(a-i)+3a=a+m-i$ and $(a-j)+3a=a+m-j$ in $C_{m+2}$, which contradict to  Lemma~\ref{lem_contradiction_list} (2). In the following analysis, for similar contradiction deductions we only list the corresponding $(4,1)$-sums  in the tables.

In the $(5, 1)$ case, $3a=m+1$ and $W=[a-5, a-1]\cup\{a+m\}$. The only possible structures of $A$ are trivial. Details are listed in Table~\ref{tab51411}.

\begin{table}[h!]
\centering
\begin{tabular}{ |c|c|c|c| }
 \hline
 \multirow{2}*{$\omega$} & \multirow{2}*{$C_{\omega}\backslash C_m$} & \multirow{2}*{\textbf{$(4,1)$-sums in $C_\omega$}} & \multirow{2}*{\textbf{Conclusion}}  \\
 &&&\\
 \hline
 \multirow{3}*{$m+1$} & $a-1$ & $2(a-1)+2a=a+m-1$ & \multirow{2}*{Contrad. by Lem.~\ref{lem_contradiction_list_m+1}} \\
 \cline{2-3}
  & $a-i$, $i\in [2, 5]$ & $(a-i)+3a=a+m+1-i$ &   \\
 \cline{2-4}
  & $a+m$ & NONE & Trivial\\
 \hline
 \multirow{9}*{$m+2$} & \multirow{2}*{$\{a-1, a-2\}$} & $3(a-1)+a=a+m-2;$ & \multirow{6}*{Contrad. by Lem.~\ref{lem_contradiction_list} (2)}\\
 & & $(a-2)+3a =a+m-1$ & \\
 \cline{2-3}
 & $\{a-i, a-j\}$, $i\ne j$, & The corresponding two equations & \\
 &    $\{1, 2\}\ne\{i, j\}\subset [5]$ & in the $\omega= m+1$ case   & \\
 \cline{2-3}
 & \multirow{2}*{$\{a+m, a-i\}$, $i\in [3, 5]$} & $(a-i)+3a=a+m+1-i;$ & \\
 & & $(6-i)(a+m)+ (i-2)(a+m-1)= a-i$ &\\
 \cline{2-4}
 & $\{a+m, a-2\}$ & $2(a-2)+ a+ (a+1)=a+m-2$ &Contrad. by Lem.~\ref{lem_contradiction_list} (1)\\
 \cline{2-4}
 & \multirow{2}*{$\{a+m, a-1\}$} & $(a-1)+3a= a+m;$ &\multirow{2}*{Contrad. by Lem.~\ref{lem_contradiction_list} (3)}\\
 & & $2(a-1)+2a=a+m-1$ &\\
 \hline
\end{tabular}
\caption{Scenarios in the $(5, 1)$ case for $(k, \ell, \lambda)=(4, 1, 1)$}
\label{tab51411}
\end{table}

In the $(4, 2)$ case, $3a=m+2$ and $W=[a-4, a-1]\cup\{a+m, a+m+1\}$. The only possible structure of $A$ is trivial again. Details are listed in Table~\ref{tab42411}. Henceforth, we omit the detailed specification of which conditions in Lemma~\ref{lem_contradiction_list_m+1} or Lemma~\ref{lem_contradiction_list} are used to derive contradictions since they can be verified easily.
%

\begin{table}[h!]
\centering
\begin{tabular}{ |c|c|c|c| }
 \hline
 \multirow{2}*{$\omega$} & \multirow{2}*{$C_{\omega}\backslash C_m$} & \multirow{2}*{\textbf{$(4,1)$-sums in $C_\omega$}} & \multirow{2}*{\textbf{Conclusion}}  \\
 &&&\\
 \hline
 \multirow{3}*{$m+1$} & $a-i$, $i\in [4]$ & $3(a-1)+ a=a+m+2-3i$ & \multirow{2}*{Contradiction} \\
 \cline{2-3}
  & $a+m+1$ & $3(a+m+1)+(a+m-1)=a+1$ &   \\
 \cline{2-4}
  & $a+m$ & NONE & Trivial\\
 \hline
 \multirow{9}*{$m+2$} & $\{a-i, a-j\}$, $i\ne j$, & $3(a-i)+a=a+m+2-3i;$ & \multirow{9}*{Contradiction}\\
 & $\{i, j\}\subset [4]$ & $3(a-j)+a=a+m+2-3j$ & \\
 \cline{2-3}
 & \multirow{2}{*}{$\{a+m+1, a-i\}$, $i\in [4]$} & $3(a+m+1)+ (a+m-1)= a+1$; & \\
 &     & $3(a-i)+a=a+m+2-3i$   & \\
 \cline{2-3}
 & \multirow{2}*{$\{a+m, a-1\}$} & $3(a-1)+a=a+m-1;$ & \\
 & & $2(a-1)+ 2a= a+m$ & \\
 \cline{2-3}
 & \multirow{2}*{$\{a+m, a-i\}$, $i\in [2, 4]$} &  $2(a-i)+2a=a+m+2-2i$; &\\
 & & $(a-i)+3a=a+m+2-i$ &\\
 \cline{2-3}
 & {$\{a+m, a+m+1\}$} & $(a+m-1)+(a+m)+2(a+m+1)=a$ &\\
 \hline
\end{tabular}
\caption{Scenarios in the $(4, 2)$ case for $(k, \ell, \lambda)=(4, 1, 1)$}
\label{tab42411}
\end{table}

In the $(3, 3)$ case, $a=1+m\slash 3\equiv 2m+2$ and $W=[a-3, a-1]\cup[a+m, a+m+2]$. Further, $C_m$ and $W$ are both symmetric, i.e., $C_m=-C_m$ and $W=-W$.
 By symmetry, only half of the cases need to be considered. Details are  listed in Table~\ref{tab33411}. We explain the first case, which gives type 2, and last case, which gives type 4.
\begin{table}[h!]
\centering
\begin{tabular}{ |c|c|c|c| }
 \hline
 \multirow{2}*{$\omega$} & \multirow{2}*{$C_{\omega}\backslash C_m$} & \multirow{2}*{\textbf{$(4,1)$-sums in $C_\omega$}} & \multirow{2}*{\textbf{Conclusion}}  \\
 &&&\\
 \hline
 \multirow{2}*{$m+1$} & $a-1$ & $4(a-1)=a+m-1$ & Type 2 \\
 \cline{2-4}
  & $a-i$, $i\in\{2, 3\}$ & $3(a-i)+a=a+m+3-3i$ & \multirow{5}*{Contradiction}  \\
 \cline{1-3}
 \multirow{5}*{$m+2$} & $\{a-i, a-j\}$, $i< j$, & \multirow{2}*{$(a-i)+2(a-j)+a=a+m+3-i-2j$} & \\
 & $\{i, j\}\subset [3]$ & & \\
 \cline{2-3}
 & $\{a-i, a+m+j\}$   & $2(a-i)+2a=a+m+3-2i;$ & \\
 &  $i\in \{2, 3\}$, $j\in \{0, 1, 2\}$   & $3(a-i)+a=a+m+3-3i$  & \\
 \cline{2-4}
 & $\{a-1, a+m\}$ & Special case. See analysis & Type 4\\
 \hline
\end{tabular}
\caption{Scenarios in the $(3, 3)$ case for $(k, \ell, \lambda)=(4, 1, 1)$}
\label{tab33411}
\end{table}

(i) In the first case, we have $\omega=m+1$ and a (4,1)-sum
 $4(a-1)=a+m-1$ in $C_{\omega}$. By Fact~\ref{fact1}, $4A_{a-1}\cap A_{a+m-1}=\emptyset$ and hence $|4A_{a-1}|+|A_{a+m-1}|\le p^{n-1}$. Moreover from $|A|\ge mp^{n-1}$ we know that $|A_{a-1}|+|A_{a+m-1}|\ge p^{n-1}$. Hence $|4A_{a-1}|=|A_{a-1}|$ and $A_{a-1}$ is a coset of a proper subspace of $K$.
Under linear transforms,  we can assume $A_{a-1}=V$ for some subspace $V\subsetneq K$.
Further $(4A_{a-1})\sqcup A_{a+m-1}=K$, one can only have $A_{a+m-1}=K\backslash V$, and hence $A=\left(\{a-1\}\times V\right)\sqcup\left([a, a+m-2]\times K\right)\sqcup\left(\{a+m-1\}\times (K\backslash V)\right).$ By the value of $a$, $A$ is of type 2.

(ii) In the last case, we have $\omega=m+2$ and two (4,1)-sums in $C_\omega$, $3(a-1)+a=a+m$ and $3(a+m)+(a+m-1)=(a-1)$. Then
$|A_{a+m}|+|A_a|\le |A_{a+m}|+|A_a+3A_{a-1}|\le p^{n-1}$ and $|A_{a-1}|+|A_{a+m-1}|\le |A_{a-1}|+|A_{a+m-1}+3A_{a+m}|\le p^{n-1}$ by Fact~\ref{fact1}. Together with that $|A_{a}|+|A_{a+m-1}|+|A_{a-1}|+|A_{a+m}|\ge 2p^{n-1}$, which is from $|A|\ge mp^{n-1}$, we know that $|A|=mp^{n-1}$, $|A_{a+m}|+|A_a|= |A_{a+m}|+|A_a+3A_{a-1}|= p^{n-1}$ and $|A_{a-1}|+|A_{a+m-1}|= |A_{a-1}|+|A_{a+m-1}+3A_{a+m}|= p^{n-1}$. Moreover, $A_{a+m}\sqcup(A_{a}+3A_{a-1})= A_{a-1}\sqcup (A_{a+m-1}+3A_{a+m})=K$. Denote $H=Sym(A_a)$, which is a proper subspace of $K$. Fix a subspace decomposition $K=H\times L'$. As all four sets are nonempty, $\dim(L')\ge 1$. As $|A_a|=|A_a+ 3A_{a-1}|$, there exists an element $c_1\in L'$ such that $3A_{a-1}\subseteq H+ c_1$. As $A_a$ is a union of cosets of $H$ and $3A_{a-1}\subseteq H+ c_1$, $A_{a}+3A_{a-1}$ is also a union of cosets of $H$. From $A_{a+m}\sqcup(A_{a}+3A_{a-1})=K$, $A_{a+m}$ is a nonempty union of cosets of $H$. Hence we have
$$|A_{a-1}|\le |3A_{a-1}|\le |H|\le |A_{a+m}|.$$
Do the same analysis to $A_{a+m-1}$, we also have $$|A_{a+m}|\le |3A_{a+m}|\le |Sym(A_{a+m-1})|\le |A_{a-1}|.$$ Compare those two inequalities and we find that all the equalities must hold. Specially, we have $|A_{a+m}|=|H|=|3A_{a-1}|=|A_{a-1}|$. As $A_{a+m}$ is a union of cosets of $H$ and $3A_{a-1}$ is contained in a coset of $H$, they are cosets of $H$ themselves. By the same analysis to $A_{a+m-1}$, $3A_{a+m}$ and $A_{a-1}$ are cosets of $Sym(A_{a+m-1})$. As $3A_{a+m}$ is also a coset of $H$, we can only have $H=Sym(A_{a+m-1})$ and hence all sets are unions of cosets of $H$. Denote $A_{a+m}=H+c_2$ and $A_{a-1}= H+ c_3$. From former analysis, as $a+m\ne 0$, by choosing a proper isomorphism we can set $c_2=0$. From that $A_{a+m}\sqcup(A_{a}+3A_{a-1})=K$, $A_a\slash H= L'\setminus \{-3c_3\}$. From that  $A_{a-1}\sqcup (A_{a+m-1}+3A_{a+m})=K$, $A_{a+m-1}\slash H= L'\setminus \{ c_3\}$.
Finally, by $4(a-1)=a+m-1$, we have $(4 A_{a-1})\cap A_{a+m-1}=\emptyset$, which means $c_3=0$. Concluding, $A_{a+m}=A_{a-1}= H$ and $A_a=A_{a+m-1}=K\setminus H$. Hence $A$ is of type 4.

\vspace{0.2cm}

\textbf{Case (2):} $C_m$ is covered by a $(4, 1)$-sum-free interval of length $m+1$.

Let $I=[a, a+m]$ be the interval of length $m+1$ covering $C_m$. Since $I$ is $(4, 1)$-sum-free, there is only one element left in $\mathbb{F}_p\setminus (I\cup 4I)$. So  up to isomorphism, we can set $a+m+1=4a$, which leads to $a=(m+1)\slash 3$.  For different possibilities of $C_m$, we have
\[
4C_m=
\left\{
\begin{aligned}
 \{4a\}\cup[4a+2, 4a+4m], & ~C_m=I\backslash\{a+1\};\\
[4a, 4a+4m-2]\cup\{4a+4m\}, & ~C_m=I\backslash\{a+m-1\};\\
[4a, 4a+4m], & \text{ otherwise.}
\end{aligned}
\right.
\]
Hence the corresponding $C_{\omega}$ behaves as follows from (\ref{eq_8}) in Lemma~\ref{lem_51_whole}.
\[
C_{\omega}\subseteq
\left\{
\begin{aligned}
 [a-1, a+m]\cup\{a+m+2\}, & ~C_m=I\backslash\{a+1\};\\
\{a-3\} \cup[a-1, a+m], & ~C_m=I\backslash\{a+m-1\};\\
 [a-1, a+m], & \text{ otherwise.}
\end{aligned}
\right.
\]

When $\omega=m+1$, if $C_{m+1}\slash C_m\subset I$, that is, $C_{m+1}=I$, then $A$ is trivial. So there are three choices left, all of which lead to a contradiction by Lemma~\ref{lem_contradiction_list_m+1}: if $C_{m+1}\slash C_m=\{a-3\}$, consider the $(4,1)$-sum $(a-3)+3a=a+m-2$; if $C_{m+1}\slash C_m=\{a+m+2\}$, consider $3(a+m+2)+(a+m)=a+4$; if $C_{m+1}\slash C_m=\{a-1\}$, consider $(a-1)+3a=a+m$.

When $\omega=m+2$, if $a-1\notin C_{m+2}$, then $C_{m+2}=\{a-3\}\cup [a, a+m]$ or $\{a+m+2\}\cup [a, a+m]$, contradicting to Lemma~\ref{lem_contradiction_list} (1) by considering  $(a-3)+2a+(a+2)=a+m$ for the former, and $2(a+m+2)+(a+m)+(a+m-1)=a+1$ for the latter.
If $a-1\in C_{m+2}$, then $C_{m+2}=[a-1, a+m]$, $\{a-3\}\cup[a-1, a+m]\backslash\{a+m-1\}$, or $\{a+m+2\}\cup[a-1, a+m]\backslash\{a+1\}$, all of which contradict to Lemma~\ref{lem_contradiction_list} (3) by the same two $(4,1)$-sums : $3(a-1)+(a+2)=a+m$ and $3(a-1)+a=a+m-2$.
\end{proof}

The proof idea of the $(3,2,1)$ case is much similar so we mainly focus on the difference.
\begin{lemma}\label{lem_(3,2,1)_finial}
Let $(k, \ell, \lambda)=(3, 2, 1)$  and $m\ge \max\{5T(1\slash5),21\}$. Suppose $A\subset\mathbb F_p^n$ is a  nontrivial $(3,2)$-sum-free set with $|A|\ge mp^{n-1}$. Then $A$ is  abnormal if and only if $A$ is of type 3 or 4.

\end{lemma}
\begin{proof} The idea is similar to the proof of Lemma~\ref{lem_(4,1,1)_finial}, but will be more complicated when $\omega=m+2$ if the weaker property (\ref{eq_9})  in Lemma~\ref{lem_51_whole} (1) happens. Recall that $C_i:=\{b_1,b_2,\ldots,b_i\}$ for $i\in [p]$.

\vspace{0.2cm}

\textbf{Case (1):} $C_m$ is an interval.
Set $C_m=[a, a+m-1]$, which is $(3,2)$-sum-free. 
By the same definition of $(i, j)$ cases, again only the following four cases need to be considered, but with different values of $a$.
\begin{center}
\begin{tabular}{ |c|c|c|c|c| }
 \hline
 \textbf{case:} & (6, 0) & (5, 1) & (4, 2) & (3, 3) \\
 \hline
 \textbf{requirement:} & $2a+2m-1=3a$ & $2a+2m=3a$ & $2a+2m+1=3a$ & $2a+2m+2=3a$ \\
 $a$\textbf{ value:} & $2m-1$ & $2m$ & $2m+1$ & $2m+2$ \\
 \hline
\end{tabular}
\end{center}
Recall that $p=5m+3$. When going through the four cases, we use $m$ as the main parameter for simplicity, since all instances of $a$ can be uniquely represented in terms of $m$.

In the $(6, 0)$ case, $a=2m-1$ and $C_m=[2m-1,3m-2]$. When $\omega= m+1$, $3C_{m}\cap 2C_{m+1}=\emptyset$ by Lemma~\ref{lem_51_whole} (1). So $2C_{m+1}\subset [4m-8, m-7]$. From $b_{m+1}+ C_m\subset 2C_{m+1}$, we have $b_{m+1}\in [2m-7,2m-2]$. Further, by $2b_{m+1}\in 2C_{m+1}$, $b_{m+1}\in [2m-4,2m-2]$.
 For each $b_{m+1}=2m-i$ with $i\in [2,4]$, the $(3,2)$-sum $3(2m-i)=(3m-2)+(3m+2-3i)$ leads to a contradiction by Lemma~\ref{lem_contradiction_list_m+1}.
 When $\omega=m+2$, Lemma~\ref{lem_51_whole} only gives a weak property $3C_{m-1}\cap 2C_{\omega}=\emptyset$ in (\ref{eq_9}).
 For different choices of $b_m$,  we have
\[
3C_{m-1}=
\left\{
\begin{aligned}
 [m-3, 4m-9], & ~b_m=2m-1;\\
\{m-6\}\cup[m-4, 4m-9], & ~b_m={2m};\\
[m-6, 4m-11]\cup\{4m-9\}, & ~b_m={3m-3};\\
[m-6, 4m-12], & ~b_m={3m-2};\\
[m-6, 4m-9], & \text{ otherwise.}
\end{aligned}
\right.
\]
Consequently, we have
\[
2C_{\omega}\subseteq
\left\{
\begin{aligned}
 [4m-8, m-4], & ~b_m=2m-1;\\
[4m-8, m-7]\cup\{m-5\}, & ~b_m=2m;\\
\{4m-10\}\cup[4m-8, m-7], & ~b_m=3m-3;\\
[4m-11, m-7], & ~b_m=3m-2;\\
[4m-8, m-7], & \text{ otherwise.}
\end{aligned}
\right.
\]
From the same process of computing $b_{m+1}$ in the $\omega=m+1$ case, we have
\[
\{b_{m+1}, b_{m+2}\}\subseteq
\left\{
\begin{aligned}
 [2m-4, 2m-2]\cup\{3m-1\}, & ~b_m=2m-1;\\
 [2m-4, 2m-2], & ~b_m=2m;\\
 [2m-5, 2m-2], & ~b_m=3m-3;\\
 [2m-5, 2m-2], & ~b_m=3m-2;\\
 [2m-4, 2m-2], & \text{ otherwise.}
\end{aligned}
\right.
\]
To sum up, either $\{b_{m+1}, b_{m+2}\}\subseteq[2m-5, 2m-2]$, or $\{b_{m+1}, b_{m+2}\}=\{3m-1, 2m-i\}$ for some $i\in [2,4]$. For the former case, we denote $b_{m+1}= 2m-i$ and $b_{m+2}= 2m-j$ for some $i\ne j\in [2,5]$. Then we have the following two $(3,2)$-sums in $C_{m+2}$ with at least five distinct terms:
\[
\begin{aligned}
(2m-i)+2(2m-1) = (3m-2)+(3m-i);\\
(2m-j)+2(2m-1) = (3m-2)+(3m-j).
\end{aligned}
\] This contradicts to Lemma~\ref{lem_contradiction_list} (2).
For the latter case, we have a $(3,2)$-sum $(2m-i)+(2m-1)+ (2m)=(3m-1)+(3m-i)$ with $i\in [2,4]$, contradicting to Lemma~\ref{lem_contradiction_list} (1).

In the $(5, 1)$ case, $a=2m$ and $C_m=[2m,3m-1]$. By the same analysis, when $\omega= m+1$, we have $b_{m+1} \in [2m-2, 2m-1]$; and when $\omega= m+2$,

\[
2C_{\omega}\subseteq
\left\{
\begin{aligned}
 [4m-5, m-1], & ~b_m=2m;\\
[4m-5, m-4]\cup\{m-2\}, & ~b_m=2m+1;\\
\{4m-7\}\cup[4m-5, m-4], & ~b_m=3m-2;\\
[4m-8, m-4], & ~b_m=3m-1;\\
[4m-5, m-4], & \text{ otherwise.}
\end{aligned}
\right.
\]
Correspondingly, we have
\[
\{b_{m+1}, b_{m+2}\}\subseteq
\left\{
\begin{aligned}
 \{2m-2, 2m-1, 3m, 3m+1\}, & ~b_m=2m;\\
 [2m-4, 2m-1], & ~b_m=3m-1;\\
 \{2m-2, 2m-1\}, & \text{ otherwise.}
\end{aligned}
\right.
\]
The analyses in different scenarios are listed in Table~\ref{tab51321}.
\begin{table}[h!]
\centering
\begin{tabular}{ |c|c|c|c| }
 \hline
 \multirow{2}*{$\omega$} & \multirow{2}*{$C_{\omega}\backslash C_m$} & \multirow{2}*{\textbf{$(3,2)$-sums in $C_\omega$}} & \multirow{2}*{\textbf{Conclusion}}  \\
 &&&\\
 \hline
 $m+1$ & $2m-i$, $i\in [2]$ & $3(2m-i)=(3m-i)+(3m-2i)$ & Contradiction \\
 \hline
 \multirow{4}*{$m+2$} & {$\{2m-i, 2m-j\}$, $i\ne j\in [4]$} & $(2m-i)+(2m-j)+2m=(3m-i)+(3m-j)$ & \multirow{4}*{Contradiction}\\
 \cline{2-3}
 & $\{3m, 3m+1\}$ & $2(2m)+(2m+1)=(3m)+(3m+1)$ & \\
 \cline{2-3}
 & $\{2m-i, 3m+j\}$, & \multirow{2}*{$(2m-i)+2m+(2m+i+2j)=2(3m+j)$} & \\
 & $i\in\{1, 2\}$; $j\in\{0, 1\}$ &  &\\
 \hline
\end{tabular}
\caption{Scenarios in the $(5, 1)$ case for $(k, \ell, \lambda)=(3, 2, 1)$}
\label{tab51321}
\end{table}

In the $(4, 2)$ case, $a=2m+1$ and $C_m=[2m+1,3m]$.
By the same analysis, when $\omega= m+1$, we have $b_{m+1}\in \{2m-1, 2m, 3m+1\}$; and when $\omega= m+2$,
\[
2C_{\omega}\subseteq
\left\{
\begin{aligned}
 [4m-2, m+2], & ~b_m=2m+1;\\
[4m-2, m-1]\cup\{m+1\}, & ~b_m=2m+2;\\
\{4m-4\}\cup[4m-2, m-1], & ~b_m=3m-1;\\
[4m-5, m-1], & ~b_m=3m;\\
[4m-2, m-1], & \text{ otherwise.}
\end{aligned}
\right.
\]
Correspondingly,
\[
\{b_{m+1}, b_{m+2}\}\subseteq
\left\{
\begin{aligned}
 \{2m-1, 2m, 3m+1, 3m+2\}, & ~b_m=2m+1;\\
 \{2m-2, 2m-1, 2m, 3m+1\}, & ~b_m=3m;\\
 \{2m-1, 2m, 3m+1\}, & \text{ otherwise.}
\end{aligned}
\right.
\]
The analyses in different scenarios are listed in Table~\ref{tab42321}. The only case that requires further comments is when $\omega= m+1$ and $C_{m+1}\backslash C_m=\{2m\}$. From $3(2m)=2(3m)$ we know that
$2p^{n-1}\le 2(|A_{2m}|+ |A_{3m}|)< 3|A_{2m}|+ 2|A_{3m}|\le p^{n-1}+3p^{n-2}$, which leads to an obvious contradiction.

\begin{table}[h!]
\centering
\begin{tabular}{ |c|c|c|c| }
 \hline
 \multirow{2}*{$\omega$} & \multirow{2}*{$C_{\omega}\backslash C_m$} & \multirow{2}*{\textbf{$(3,2)$-sums in $C_\omega$}} & \multirow{2}*{\textbf{Conclusion}}  \\
 &&&\\
 \hline
 \multirow{3}*{$m+1$} & $2m-1$ & $3(2m-1)=(3m-1)+(3m-2)$ & \multirow{2}*{Contradiction}\\
 \cline{2-3}
  & $2m$ & $3(2m)=2(3m)$ & \\
 \cline{2-4}
  & $a+m$ & NONE & Trivial\\
 \hline
 \multirow{7}*{$m+2$} & $\{2m-i, 2m-j\}$, $i<j$, & \multirow{2}*{$(2m-i)+(2m-j)+(2m+1)=(3m-i)+ (3m+1-j)$} & \multirow{7}*{Contradiction}\\
  & $i, j\in [0, 2]$ &  & \\
 \cline{2-3}
  & \multirow{2}*{$\{3m+2, 2m-i\}$, $i\in [0, 1]$} & $3(3m+2)=(2m+1)+(2m+2)$; & \\
  & & $3(3m+2)=(2m-i)+(2m+3+i)$ & \\
 \cline{2-3}
 & $\{3m+1, 2m-i\}$, $i\in [0, 2]$ & $2(2m-i)+(2m+2i+1) = (3m)+(3m+1)$ & \\
 \cline{2-3}
 &\multirow{2}*{$\{3m+1, 3m+2\}$} & $3(2m+1)=(3m+1)+(3m+2)$; & \\
 & & $2(2m+1)+(2m+2)=2(3m+2)$ &\\
 \hline
\end{tabular}
\caption{Scenarios in the $(4, 2)$ case for $(k, \ell, \lambda)=(3, 2, 1)$}
\label{tab42321}
\end{table}

In the $(3, 3)$ case, $a=2m+2$ and $C_m=[2m+2,3m+1]$. When $\omega= m+1$, we have $b_{m+1} \in \{2m+1, 3m+2\}$.
When $\omega= m+2$,
\[
2C_{\omega}\subseteq
\left\{
\begin{aligned}
 [4m+1, m+5], & ~b_m=2m+2;\\
[4m+1, m+2]\cup\{m+4\}, & ~b_m=2m+3;\\
\{4m-1\}\cup[4m+1, m+2], & ~b_m=3m;\\
[4m-2, m+2], & ~b_m=3m+1;\\
[4m+1, m+2], & \text{ otherwise.}
\end{aligned}
\right.
\]
Correspondingly,
\[
\{b_{m+1}, b_{m+2}\}\subseteq
\left\{
\begin{aligned}
 \{2m+1, 3m+2, 3m+3, 3m+4\}, & ~b_m=2m+2;\\
 \{2m-1, 2m, 2m+1, 3m+2\}, & ~b_m=3m+1;\\
 \{2m+1, 3m+2\}, & \text{ otherwise.}
\end{aligned}
\right.
\]
Because of the symmetry, we only need to check half of the scenarios. See Table~\ref{tab33321}. The last case is special but follows by the same analysis as for the last line in Table~\ref{tab33411} in the proof of Lemma~\ref{lem_(4,1,1)_finial}.
\begin{table}[h!]
\centering
\begin{tabular}{ |c|c|c|c| }
 \hline
 \multirow{2}*{$\omega$} & \multirow{2}*{$C_{\omega}\backslash C_m$} & \multirow{2}*{\textbf{$(3,2)$-sums in $C_\omega$}} & \multirow{2}*{\textbf{Conclusion}}  \\
 &&&\\
 \hline
 $m+1$ & $2m+1$ & NONE & Trivial \\
 \hline
 \multirow{5}*{$m+2$} & {$\{2m+i, 2m+j\}$, $i\ne j$,} & \multirow{2}*{$(2m+i)+2(2m+j)=(3m+j)+(3m+i+j)$} & \multirow{4}*{Contradiction}\\
 & $i, j\in [-1, 1]$ & & \\
 \cline{2-3}
 & $\{2m-j, 3m+2\},$ & $(2m+2)+2(2m-j)=2(3m+1-j);$ & \\
 & $j\in \{0, 1\}$    & $3(2m-j)=(3m+1-j)+(3m-1-2j)$   & \\
 \cline{2-4}
 & $\{2m+1, 3m+2\}$, & Special case & Type 4 \\
 \hline
\end{tabular}
\caption{Scenarios in the $(3, 3)$ case for $(k, \ell, \lambda)=(3, 2, 1)$}
\label{tab33321}
\end{table}

\textbf{Case (2):} $C_m$ is covered by a $(3, 2)$-sum-free interval of length $m+1$, say $I=[a, a+m]$. Then $I$ is unique up to isomorphism and without loss of generality we can set $2a+2m+1=3a$, which means $a=2m+1$. Hence $I=[2m+1,3m+1]$. Then 
\[
3C_m=
\left\{
\begin{aligned}
 \{3a\}\cup[3a+2, 3a+3m], & ~C_m=I\backslash \{a+1\};\\
[3a, 3a+3m-2]\cup\{3a+3m\}, & ~ C_m=I\backslash\{a+m-1\};\\
[3a, 3a+3m], & \text{ otherwise.}
\end{aligned}
\right.
\]
If $\omega=m+1$, we know from $3C_m\cap 2C_{m+1}=\emptyset$,
\[
2C_{m+1}\subseteq
\left\{
\begin{aligned}
 [2a-1, 2a+2m]\cup\{3a+1\}, & ~C_m=I\backslash \{a+1\};\\
\{2a-3\}\cup[2a-1, 2a+2m], & ~C_m=I\backslash \{a+m-1\};\\
[2a-1, 2a+2m], & \text{ otherwise.}
\end{aligned}
\right.
\]
All the three situations above lead to $C_{m+1}\subseteq [a, a+m]$, that is, $C_{m+1}=I$. Then $A$ is trivial since $I$ is $(3, 2)$-sum-free.

When $\omega=m+2$,  from Lemma~\ref{lem_51_whole} (1) we have $3C_{m-1}\cap 2C_{m+2}=\emptyset$.
 By the similar proving process as in \textbf{Case (1)}, we know that no matter what choice of $b_m$ we choose, the interval $3[a+2, a+m-2]=[m+6, 4m-6]$ must be contained by $3C_{m-1}$, and hence $2C_{m+2}\subset [4m-5, m+5]$, leading to $C_{m+2}\subset[2m-2, 3m+4]$. A convenient argument as follows shows that this directly leads to a contradiction.

As $C_m$ is not an interval but covered by $I=[2m+1,3m+1]$, both $2m+1$ and $3m+1$ must be contained in $C_m$.
Observe that $C_{m+2}$ cannot be $[2m+1, 3m+2]$, because if otherwise, two $(3, 2)$-sums $2(2m+1)+(2m+2)=2(3m+2)$ and $2(3m+2)+(3m+1)=2(3m+1)$ violate Lemma~\ref{lem_contradiction_list} (3).
Next, we prove that each element in $[2m-2, 2m]\cup\{3m+3, 3m+4\}$ cannot be in $C_{m+2}$, which leads to $C_{m+2}=[2m+1, 3m+2]$, and hence to the contradiction.
First, we check $2m-i$, $i\in \{1, 2\}$. Consider two $(3, 2)$-sums $2(2m-i)+(2m+1)=(3m-2i)+(3m+1)$ and $(m-i)+2(2m+1)=(3m+1-i)+(3m+1)$. At least one of $3m-2i$ and $3m+1-i$ must be in $C_m$, leading to a $(3, 2)$-sum  in $C_{m+2}$ violating  Lemma~\ref{lem_contradiction_list} (1).
Then we check $2m$. If $C_m=I\backslash\{3m\}$,  two $(3, 2)$-sums $2(2m)+(2m+2)=2(3m+1)$ and $(2m)+2(2m+1)=2(3m+1)$ violate Lemma~\ref{lem_contradiction_list} (3); otherwise the $(3, 2)$-sum $2(2m)+(2m+1)=(3m)+(3m+1)$ violates Lemma~\ref{lem_contradiction_list} (1).
Finally, we check $3m+i$, $i\in \{3, 4\}$. Consider two $(3, 2)$-sums $2(2m+1)+(2m+i-1)=(3m+1)+(3m+i)$ and $(3m+1)+2(3m+i)=(2m+1)+(2m+2i-3)$. At least one of $2m+i-1$ and $2m+2i-3$ must be in $C_m$, leading to a $(3, 2)$-sum  in $C_{m+2}$ violating Lemma~\ref{lem_contradiction_list} (1).
\end{proof}
From  Lemma~\ref{lem_(4,1,1)_finial} and Lemma~\ref{lem_(3,2,1)_finial}, Theorem~\ref{thm_(5, 1)_unnormal} is proved completely.

\subsection{$(k+\ell,\lambda)=(4,1)$}
The proof arguments  of Theorem~\ref{thm_3_1} is similar to that of Theorem~\ref{thm_(5, 1)_unnormal}, thus we move it into Appendix~\ref{appthm3}.

\section{Conclusion}\label{sec_conc}

We studied the Hilton-Milner type problem of $(k, \ell)$-sum-free subsets in $\mathbb F_p^n$ for general $k>\ell\geq 1$. Under certain parameter conditions, we completely classified all such sets with the largest size but  not contained in any extremal cuboid $A_{k, \ell, p, j}$, i.e., optimal nontrivial structures. More specifically, for large $p$, we solved the problem when $p\equiv 2+\lambda \pmod{k+\ell}$ for any $\lambda\in [0, k+\ell-5]$, and for $\lambda\in \{0,1\}$ with any $k>\ell\geq 1$. The former one was classified into only two types of structures, and special cases in the latter one gives more: when $(k+\ell,\lambda)=(5,1)$  we have four types, and when $(k+\ell,\lambda)=(4,1)$ we have three types of structures.

All our results show that when $p=(k+\ell)m+2+\lambda$, if $A$ cannot be covered by any extremal structure $A_{k, \ell, p, j}$ which is of size $(m+1)p^{n-1}$, i.e., $A$ is a nontrivial $(k, \ell)$-sum-free subset, then $|A|\le mp^{n-1}$. In other words, there is a gap of $p^{n-1}$ between the sizes of the ``best structure'' and the ``best nontrivial structure''. However, this phenomenon does not always happen for all $(k, \ell, \lambda)$ and large enough prime $p$. For example, when $(k, \ell)=(2, 1)$ and $p\equiv 1\mod 3$ (i.e., $\lambda=k+\ell-1=2$ which does not fall in our assumption), 
the largest sum-free set in $\mathbb F_p^n$ has size $(m+1)p^{n-1}$, but there are nontrivial examples of size $(m+1)p^{n-1}-1$ \cite{Rhemtulla1971}. {Note that when $\lambda=k+\ell-2$,  $p=(k+\ell)m+2+\lambda=(m+1)(k+\ell)$, which is not a prime. So the first question we raise is the following.

\begin{problem}\label{p1} For any $k>\ell\geq 1$ and $\lambda = k+\ell-1$, if $p\equiv \lambda+2 \equiv 1\pmod{k+\ell}$ is a large prime, determine the largest size $M_{k,\ell,\lambda}$ of a $(k, \ell)$-sum-free subset in $\mathbb F_p^n$ for any $n\geq 1$ and characterize all  extremal structures.
\end{problem}

Theorem~\ref{thm_k_l_max_structure} solved all the remainant cases of Problem~\ref{p1} when $\lambda \in [0, k+\ell-3]$. Motivated by the mentioned example of size $(m+1)p^{n-1}-1$ in \cite{Rhemtulla1971}, we raise the second question as follows.

\begin{problem}\label{p2}
Find out more triples $(k, \ell, \lambda)$ with $\lambda=k+\ell-1$ such that if $p\equiv 1\pmod{k+\ell}$ is a large prime,
there exists a $(k, \ell)$-sum-free subset in $\mathbb F_p^n$ of size $|A|\ge M_{k,\ell,\lambda}-c$ for some constant $c$, which is not contained in any $(k, \ell)$-sum-free subset of size $M_{k,\ell,\lambda}$.
\end{problem}

Then $(k, \ell, \lambda)=(2,1,2)$ is an example satisfying Problem~\ref{p2}. Furthermore, we hope to solve the following problem that is a complement of Theorem~\ref{thm_with_fourier}.

\begin{problem}\label{p3}
For any triple $(k, \ell, \lambda)$ with $k+\ell\geq 4$ and $\lambda\in\{k+\ell-4, k+\ell-3\}$, if $p\equiv \lambda+2 \pmod{k+\ell}$ is a large prime, 
classify all nontrivial $(k, \ell)$-sum-free subsets in $\mathbb F_p^n$ of the largest size.
\end{problem}
}

Theorems~\ref{thm_(5, 1)_unnormal} and~\ref{thm_3_1} solve Problem~\ref{p3} completely when $\lambda\in \{0,1\}$. Note that the cases $(k+\ell,\lambda)=(4,1)$ and $(5,1)$ are the two smallest cases for $\lambda=k+\ell-3$ and $k+\ell-4$ in Problem~\ref{p3}, respectively. So it is plausible to extend the methods for Theorems~\ref{thm_(5, 1)_unnormal} and~\ref{thm_3_1} to solve  Problem~\ref{p3} completely. We leave this for further study. Moreover, combining Theorem~\ref{thm_with_fourier}, we see that the Hilton-Milnor type problem has been solved for all $\lambda \in[0,k+\ell-3]$ when $k+\ell=4$ (that is, $(3,1)$-sum-free),  for $\lambda \in[0,k+\ell-4]$ when  $k+\ell=5$, and for $\lambda \in[0,k+\ell-5]$ when  $k+\ell\geq 6$ and $p$ is large.

It is also interesting to find out more results of the Hilton-Milnor type problem of $(k, \ell)$-sum-free subsets over other finite abelian ambient groups.

\vskip 10pt
\bibliographystyle{IEEEtran}
\bibliography{KLF}

\appendix

 A slightly further argument in the proving process of Lemma~\ref{lem_stru_of_Cm} leads to the following more detailed result, which is useful when dealing with $(k+\ell,\lambda)=(5,1)$ or $(4,1)$.

\begin{corollary}\label{coro_k+ell-4}
Let  $m\geq 7$. Suppose $C \subset\mathbb F_p$  is a $(k, \ell)$-sum-free set of size $m$, and  $I$ is the shortest interval covering $C$. If $|I|=m+1$,  $I$ is also $(k, \ell)$-sum-free. If $|I|=m+2$ or $m+3$,  the following hold.
\begin{itemize}
  \item[(1)] If $k+\ell= 4+\lambda$, there exists a certain $a\in\mathbb F_p$ such that up to isomorphism, $C=\{a, a+m+1\}\cup[a+2, a+m-1]$.
  \item[(2)] If $k+\ell= 3+\lambda$, there exists a certain $a\in\mathbb F_p$ such that up to isomorphism, $C=\{a, a+m+1\}\cup[a+2, a+m-1]$, $C= \{a, a+1, a+m+1\}\cup[a+3, a+m-1]$, or $C= \{a, a+1, a+m, a+m+1\}\cup[a+3, a+m-2]$.
\end{itemize}
\end{corollary}

\begin{proof} If $|I|=m+1$, $I$ is also $(k, \ell)$-sum-free by Claim~\ref{clm3}.
 From the proof of Lemma~\ref{lem_stru_of_Cm}, when $|I|=m+2$ or $m+3$, the condition $k+\ell=\lambda +3$ or $\lambda+4$ only violates inequalities (\ref{eqp3}) and  (\ref{eqp4}), all lying in the case $h=3$. If there exists a 2-hole, say, the hole between $I_2$ and $I_3$, then $|I|=m+3$ and hence
\[ p\ge |kI|+|\ell I|-4\geq  (m+2)(k+\ell)-2> m(k+\ell)+2+\lambda=p,\] leading to a contradiction. 
So $C$ is $1$-holed from $I$ and hence $|I|=m+2$.

If $k+\ell= 4+\lambda$, the only inequality that does not hold is (\ref{eqp3}). This gives the possibility of $C$ with $|I_1|=|I_3|=1$. So $C$ can be $\{a, a+m+1\}\cup[a+2, a+m-1]$.

If $k+\ell= 3+\lambda$, inequality (\ref{eqp4}) in the proof of Lemma~\ref{lem_stru_of_Cm} does not hold either. So two more cases are possible as stated in the corollary.
\end{proof}

\subsection{Proof of Lemma \ref{lem_51_whole}}\label{ap1}

Recall that when $(k+\ell,\lambda)=(5,1)$, we have $p=5m+3$ and $\theta=8$.
Lemma \ref{lem_51_whole} is obtained by combining Lemmas \ref{lem_m+3}-\ref{lem_5_strange_Cm}.
We begin with a weaker upper bound on $\omega(A)$ as follows.

\begin{lemma}\label{lem_m+3}
Let $(k+\ell,\lambda)=(5,1)$ and $m\ge 4$. Suppose $A\subset \mathbb F_p^n$ is a $(k, \ell)$-sum-free  subset with $|A|\ge mp^{n-1}$. Then under any decomposition $(v, K)$, we have either $\omega\le m+3$ or $\omega> p-7$.
\end{lemma}
\begin{proof} Assume by contrary that $ m+4\le \omega\le p-7=5m-4$.
By Lemma~\ref{coro_beta_sum}, we have $\beta_m+ \beta_{m+1}+ \beta_{m+2}+ \beta_{m+3}+ \beta_{m+4}\le p+3$, $\beta_{m-1}+ \beta_{m+1}+ \beta_{m+2}+ \beta_{m+3}+ \beta_{m+4}\le p+3$, $\beta_{m-1}+ \beta_{m}+ \beta_{m+2}+ \beta_{m+3}+ \beta_{m+4}\le p+3$, and $\beta_{m-1}+ \beta_{m}+ \beta_{m+1}+ \beta_{m+3}+ \beta_{m+4}\le p+3$. Moreover, from Lemma~\ref{lem_improving_coro}, we have $\beta_{m-1}+ \beta_{m}+ \beta_{m+1}+ \beta_{m+2}+ \beta_{m+4}\le p+3$.

By adding up those five inequalities above, we get
\[4\left(\sum_{i=m-1}^{m+4}\beta_i\right)< 4\left(\sum_{i=m-1}^{m+3}\beta_i\right)+ 5\beta_{m+4}\le 5(p+3),\]
and hence $\sum_{i=m-1}^{m+4}\beta_i< \frac 5 4(p+3)$.
Then following the proving process of Lemma~\ref{lem_l_upper_default} by setting each $S_i=\{i\}\cup [ p-4i-6, p-4i-3]$ for any $i\in [m-2]$, we get
\[
\begin{aligned}
|A|\slash p^{n-2}&\le \sum_{i\in [m-2]}\sum_{j\in S_i}\beta_j + \sum_{j\in [m-1, m+4]}\beta_j\\
 &\le (p+3)(m-2)+ \frac 5 4(p+3)\\
 &=(p+3)(m-\frac 3 4)\\
 &<mp.
\end{aligned}
\]
The last step comes from $p>5m$. This leads to a direct contradiction to $|A|\ge mp^{n-1}$.
\end{proof}


\begin{lemma}\label{lem_5_ell+2}
Let $(k+\ell,\lambda)=(5,1)$ and $m\ge 21$. Suppose $A\subset \mathbb F_p^n$ is a $(k, \ell)$-sum-free  subset with $|A|\ge mp^{n-1}$. Then there must exist a decomposition $(v, K)$ of $\mathbb F_p^n$ such that $\omega(A)\le m+2$.
%
\end{lemma}

\begin{proof}
By Lemma~\ref{lem_summation_fourier}, as $\theta=8$ and $(2+\theta)(1+\frac1{k+\ell})(k+\ell)^{1\slash(k+\ell-2)}<21$, there exists a  decomposition $(v, K)$ such that $\omega=\omega(A)\le  p-8$.  By Lemma~\ref{lem_m+3}, $\omega\le m+3$.

If $\omega=m+3$, consider $\beta_{m-1}+\cdots +\beta_{m+3}$.
As $\sum_{i=m-1}^{m+3}i=5m+5=p+k+\ell-3$, from Lemma~\ref{lem_improving_improving_coro} by letting $t=3$, we must have $\sum_{i=m-1}^{m+3}\beta_i\le p+k+\ell-2=p+3$ unless the exceptional case listed in Lemma~\ref{lem_improving_improving_coro} happens.
As $\omega=m+3$, this leads to $|A|\slash p^{n-2}\le \sum_{i=1}^{m+3}\beta_i\le (m-2)p+(p+3)<mp$, contradicting to $|A|\ge mp^{n-1}$.

Hence the exceptional case in Lemma~\ref{lem_improving_improving_coro} must happen when $\omega=m+3$.
This means, after writing $C_{m-1}=[a, a+m-2]$ for some $a\in\mathbb F_p$ and, without loss of generality, write $C_m=[a, a+m-1]$, we have  $(k, \ell)=(4, 1)$, $C_{m+1}=[a-1, a+m-1]$, $C_{m+2}=[a-1, a+m]$, and $C_{m+3}=[a-2, a+m]$.
We claim that
\begin{equation}\label{eq:ccc}
  \text{$C_{m-1}\cap (\sum_{i=0}^3C_{m+i})=\emptyset$ and $C_{m}\cap (C_{m-1}+\sum_{i=1}^3C_{m+i})=\emptyset$}
\end{equation}
 both hold. Otherwise, by Lemma~\ref{coro_beta_sum},
 we have $\sum_{i=m-1}^{m+3}\beta_i\le p+k+\ell-2=p+3$, leading to a contradiction to $|A|\ge mp^{n-1}$ as above.

It can be verified that $a=(m+4)\slash 3$ is the only solution to Eq. (\ref{eq:ccc}). So we have the following four
 $(4, 1)$-sums in $C_{m+3}$ for all $i\in [4]$,
\[3(a-2)+(a-3+i)=a+m-5+i.\]
For each $i\in [4]$, this gives
$|A_{a-3+i}|+|A_{a+m-5+i}|\le 3|A_{a-2}|+|A_{a-3+i}|+|A_{a+m-5+i}|\le p^{n-1}+3 p^{n-2}$.
As $m\ge 21$,
\[|A|\le \sum_{i\in [4]}(|A_{a-3+i}|+|A_{a+m-5+i}|)+\sum_{j\in [a+2, a+m-5]}|A_j|+|A_{a+m}|\le (m-1)p^{n-1}+12p^{n-2}< mp^{n-1},\]
which contradicts to $|A|\ge mp^{n-1}$.

To sum up, all cases when $\omega=m+3$ lead to contradictions, and hence $\omega\le m+2$.
\end{proof}


\begin{lemma}\label{lem_5_freeness_cm}
Let $(k+\ell,\lambda)=(5,1)$ and $m\ge 4$. Suppose $A\subset \mathbb F_p^n$ is a $(k, \ell)$-sum-free subset with $|A|\ge mp^{n-1}$. Assume that under a certain decomposition $(v, K)$ we have $\omega(A)\le m+2$, then $C_m$ is  $(k, \ell)$-sum-free. Moreover,
\begin{itemize}
    \item[(1)] if $(k,\omega)= (3,m+2)$, we have $3C_{m-1}\cap 2 C_{\omega}=\emptyset$;
    \item[(2)] otherwise, $kC_m\cap \ell C_{\omega}=\emptyset$.
\end{itemize}
\end{lemma}
\begin{proof}

We prove the case when $(k,\omega)\neq (3,m+2)$ first. As $\omega\le m+2$, we have $\beta_m\ge p\slash 3$.
When $k=4$, we have $k\beta_m\ge \frac 4 3 p\ge p+k-1$. From Lemma~\ref{lem_stru_of_Cm}, $kC_m\cap \ell C_{\omega}=\emptyset$.
Otherwise, $k=3$ and $\omega\le m+1$. This means $\beta_m\ge p\slash 2$, and hence $k\beta_m\ge \frac  3 2 p\ge p+k-1$. From Lemma~\ref{lem_stru_of_Cm}, $kC_m\cap \ell C_{\omega}=\emptyset$. As $C_m$ is a subset of $C_\omega$, $C_m$ is $(k, \ell)$-sum-free.


Next,  we consider when $(k,\omega)= (3,m+2)$.
If, on the contrary, $C_m$ is not $(k, \ell)$-sum-free, then there exists a choice of $r_1, r_2, r_3, s_1, s_2\in [m]$ such that  $(B_{r_1}+B_{r_2}+B_{r_3})\cap (B_{s_1}+B_{s_2})=\emptyset$ by Fact~\ref{fact1}.
However, from Theorem~\ref{thm_Kneser}, $|B_{r_1}+B_{r_2}+B_{r_3}|+| B_{s_1}+B_{s_2}|\ge 5|B_m|-3p^{n-2}\ge \frac53p^{n-1}-3p^{n-2}>p^{n-1}$, which leads to a contradiction. Hence $C_m$ is $(k, \ell)$-sum-free.

Furthermore, from $\omega=m+2$ we have $\beta_{m-1}\ge 2p\slash 4= p\slash2$.
From Theorem~\ref{thm_Kneser}, for any $r_1, r_2, r_3\in [m-1]$, as long as $B_{r_1}+B_{r_2}+B_{r_3}\ne K$,  $|B_{r_1}+B_{r_2}+B_{r_3}|\ge 3|B_{m-1}|-2p^{n-2}\ge \frac32p^{n-1}-2p^{n-1}> p^{n-1}$, leading to a contradiction. Hence from Fact~\ref{fact1}, $3C_{m-1}\cap 2 C_{\omega}=\emptyset$.
%
\end{proof}

\begin{lemma}\label{lem_5_strange_Cm}
Let $(k+\ell,\lambda)=(5,1)$ and $m\ge \max\{5T(1\slash5),21\}$. Suppose $A\subset \mathbb F_p^n$ is a $(k, \ell)$-sum-free subset with $|A|\ge mp^{n-1}$. If under a certain decomposition $\omega< m+2$, then up to isomorphism, $C_m$ must be one of the three following forms,
\begin{itemize}
\item[(1)] an interval of length $m$;
\item[(2)] covered by  a $(k, \ell)$-sum-free interval of length $m+1$; or
\item[(3)]
$C_m=\{a\}\cup [a+2, a+m-1]\cup\{a+m+1\}$ with $a= 2m+1$.
\end{itemize}
\end{lemma}
\begin{proof} By Lemma~\ref{lem_5_freeness_cm}, $C_m$ is $(k, \ell)$-sum-free.
 When $k=4$, we have $5m+3=p\ge |4C_m|+|C_m|\ge 2|2C_m|-1+m$. When $k=3$, we  have $5m+3= p\ge |3C_m|+|2C_m|\ge |2C_m|+|C_m|-1+|2C_m|$. For both cases, we get $|2C_m|-2|C_m|\le 2$. So by setting $c=1\slash5$ in Definition~\ref{deftau}, we have $|C_m|<cp$ and $|2C_m|\le 2|C_m|+2\le (2+\tau(1\slash 5))|C_m|-3$ when $m\ge 5T(1\slash5)$, so $C_m$ is covered by an interval of length $|2C_m|-|C_m|+1\le m+3$ up to isomorphism.


Applying Corollary~\ref{coro_k+ell-4} with $C_m$, we get the required three forms.  For the third form, as $|k C_m|+|\ell C_m|= 5(m+1)+2-4=5m+3=p$, we must have $ka+1=\ell(a+m+1)$. When $(k, \ell)=(4, 1)$, $4a+1=a+m+1$, and hence $a=m\slash 3$. When $(k, \ell)=(3, 2)$, $3a+1=2(a+m+1)$, we get $a=2m+1$. Observe that $m\slash 3=2m+1$ since $m=m+p=6m+3$.
\end{proof}

\subsection{Proof of Theorem~\ref{thm_3_1}}\label{appthm3}

In this case, $k=3$, $\ell=\lambda=1$, $p=4m+3$ and $\theta=k+\ell+\lambda+2=7$. We first have the following lemma as an analog of Lemma~\ref{lem_51_whole}.

\begin{lemma}\label{lem_41_total}
Let $p=4m+3$ with $m\ge \max\{5T(1\slash4),23\}$. Suppose $A\subset \mathbb F_p^n$ is a $(3, 1)$-sum-free  subset with $|A|\ge mp^{n-1}$. Then there exists a decomposition  of $\mathbb F_p^n$ such that $\omega(A)\le m+2$. Further, under this decomposition, we have the following properties.
\begin{itemize}
  \item[(1)] $C_m$ is  $(3, 1)$-sum-free. Moreover, if $\omega\le m+1$,
\begin{equation}\label{eq_8_4}
3C_m\cap  C_{\omega}=\emptyset;
\end{equation}
if $\omega= m+2$,
\begin{equation}\label{eq_9_4}
3C_{m-1}\cap  C_{\omega}=\emptyset.
\end{equation}
  \item[(2)] Up to isomorphism, $C_m$ must be one of the following two forms,
\begin{itemize}
\item[(i)] an interval of length $m$; or
\item[(ii)] covered by  a $(3, 1)$-sum-free interval of length $m+1$.
\end{itemize}
\end{itemize}
\end{lemma}

Lemma \ref{lem_41_total} is obtained by combining Lemmas \ref{lem_4_ell+2}-\ref{lem_4_strange_Cm}.

\begin{lemma}\label{lem_4_ell+2}

Let $p=4m+3$ and $m\ge 23$. Suppose $A\subset \mathbb F_p^n$ is a $(3, 1)$-sum-free subset with $|A|\ge mp^{n-1}$. Then there must exist a decomposition of $\mathbb F_p^n$ such that $\omega(A)\le m+2$.
\end{lemma}

\begin{proof}
By Lemma~\ref{lem_summation_fourier} as $\theta=7$ and $(2+\theta)(1+\frac1{k+\ell})(k+\ell)^{1\slash(k+\ell-2)}<23$, we know that there exists a a decomposition of $\mathbb F_p^n$ such that $\omega(A)< p-7$. Fix such a decomposition.

As $p+k+\ell-2=p+2=4m+5< (m-1)+(m+2)+2(m+3)$, as long as $\omega\ge m+3$, by Lemma \ref{lem_improving_coro} we have $\beta_{m-1}+\beta_m+2\beta_{m+3}\le p+2$. Similarly, as $2(m+1)+2(m+2)=4m+6=p+k+\ell-1$, by Lemma~\ref{coro_beta_sum} we have $2\beta_{m+1}+2\beta_{m+2}\le p+2$.
Combining these two inequalities we get $\sum_{i=-1}^3\beta_{m+i}\le \frac 3 2(p+2)$.

For each $i\in [m-2]$, denote $S_i=\{i\}\cup[p-6-3i, p-4-3i]$. By following the same proving process of Lemma~\ref{lem_l_upper_default}, $\sum_{j\in S_i}\beta_j\leq p+2$ for any $i\in [m+2]$ regardless of the value of $\omega$. As $[m-2]\cup[m+4, p-7]\subset\cup_{i\in[m-2]}S_i$,

\[
\begin{aligned}
|A|\slash p^{n-2}&\le \sum_{i=1}^{m-2}\beta_i+ \sum_{i=-1}^3\beta_{m+i} +\sum_{i=m+4}^{p-7}\beta_i\\
&\le (p+2)(m-2)+\frac 3 2(p+2)\\
&=(p+2)(m-1\slash 2)< mp,
\end{aligned}
\]
which contradicts to $|A|\ge mp^{n-1}$. Hence $\omega\le m+2$.
\end{proof}

\begin{lemma}\label{lem_4_strange_Cm}
Let $p=4m+3$ and $m\ge 23$. Suppose $A\subset \mathbb F_p^n$ is a $(3, 1)$-sum-free  subset with $|A|\ge mp^{n-1}$.
 If under a certain decomposition $\omega(A)\le m+2$, then $C_m$ is $(3, 1)$-sum-free. If $\omega\le m+1$, we have $3C_{m}\cap C_{m+1}=\emptyset$. If $\omega=m+2$, instead we have $3C_{m-1}\cap C_{m+2}=\emptyset$. 

Moreover, up to isomorphism, $C_m$ must be one of the two following forms when {$m\ge 5T(\frac 14)$.}
\begin{itemize}
\item[(1)] an interval of length $m$; or
\item[(2)] covered by  a $(3, 1)$-sum-free interval of length $m+1$.
\end{itemize}
\end{lemma}
\begin{proof}
If otherwise $C_m$ is not $(3, 1)$-sum-free, then there exists a choice of $r_1, r_2, r_3, s\in C_m$ such that $r_1+ r_2+ r_3=s$. From that $A$ is $(3, 1)$-sum-free, by Fact~\ref{fact1}, $(A_{r_1}+A_{r_2}+A_{r_3})\cap A_s=\emptyset$. From that $|A|\ge mp^{n-1}$ and $\omega\le m+2$, by Theorem~\ref{thm_Kneser}, we know  $|A_{r_1}+A_{r_2}+A_{r_3}|+|A_s|\ge 4|B_m|-2p^{n-2}\ge \frac43 p^{n-1}-2p^{n-2}> p^{n-1}$, contradicting to $(A_{r_1}+A_{r_2}+A_{r_3})\cap A_s=\emptyset$.
Hence, $C_m$ is $(3, 1)$-sum-free.

If $\omega=m+1$, by $|A|\ge mp^{n-1}$ we have $\beta_m\ge \frac p 2$ and then $3\beta_m\ge p+k+1$. Hence by Lemma~\ref{lem_stru_of_Cm}, 
$3C_m\cap C_{\omega}$ must be empty. When $\omega=m+2$, by $|A|\ge mp^{n-1}$ we have $\beta_{m-1}\ge \frac p 2$. The same analysis as proving Lemma~\ref{lem_5_freeness_cm} (2) leads to $3C_{m-1}\cap C_{m+2}=\emptyset$.

To show the structure of $C_m$, we first upper bound the size of $2C_m$. From that $C_m$ is $(3, 1)$-sum-free, $p\ge |3C_m|+|C_m|\ge |2C_m|+2|C_m|-1$ and hence $|2C_m|\le 2m+4$. It is really bad for directly using this upper bound. However, we define a new set $X=(2C_m)\cup(-2C_m)$, which is the centralized version of $2C_m$. As $(C_m+2C_m)\cap C_m=\emptyset$ and $(C_m-2C_m)\cap C_m=\emptyset$, actually we have $(C_m+X)\cap C_m=\emptyset$ and hence $|C_m+X|\le p-m=3m+3$.

If $|X|\ge 2m+4$, then by Theorem~\ref{thm_Cauchy_daven},
$3m+3\ge |C_m+X|\ge |X|+m-1\ge 3m+3$. This forces all equalities hold, i.e., $|X|=2m+4$ and $|C_m+X|=|C_m|+|X|-1$.
Then by Theorem~\ref{thm_Vosper}, after an isomorphism $C_m$ and $X$ are both intervals. This goes to the form (1) of $C_m$.

If $|X|\le 2m+3$, as $(C_m+2C_m)\cap C_m=\emptyset$, $0\notin 2C_m$ and hence $0\notin X$. Moreover, as $X$ is centralized and $0\notin X$, $|X|$ is even. This means $|2C_m|\le|X|\le 2m+2$. As $m\ge 5T(\frac 14)$, under some isomorphism $C_m$ must be covered by an interval of length at most $m+3$.
To ensure that the form of $C_m$ must be one of (1) and (2), by applying Corollary~\ref{coro_k+ell-4} it suffices to kick out the three special forms of $C_m$, i.e., up to isomorphism there exists some $a$ such that $C_m=\{a, a+1, a+m+1\}\cup[a+3, a+m-1]$, $C_m=\{a, a+1, a+m, a+m+1\}\cup[a+3, a+m-2]$, or $C_m=\{a\}\cup[a+2, a+m-1]\cup\{a+m+1\}$.
Remember that $C_m$ is $(3, 1)$-sum-free. Easy calculation shows that no matter which form $C_m$ chooses among those three, no $a\in \mathbb F_p$ satisfies $C_m\cap3C_m=\emptyset$.
This means the form of $C_m$ lies in (1) and (2).

\end{proof}

\begin{proof}[Proof of Theorem~\ref{thm_3_1}]
From Lemma~\ref{lem_41_total} (2), under isomorphisms $C_m$ is either an interval or contained in the unique $(3, 1)$-sum-free interval of length $m+1$ in $\mathbb F_{p}$. We prove these two cases separately.

\vspace{0.2cm}

\textbf{Case (1):} $C_m$ is an interval.
Write $C_m=[a, a+m-1]$, then $3C_m=[3a, 3a+3m-3]$. As $3C_m\cap C_m=\emptyset$, over the circle formed by $\mathbb F_p$ denote $i$ the length of interval connecting the tail $3a+3m-3$ of $3C_m$ and the head $a$ of $C_m$, and $j$ the length of interval connecting the tail $a+m-1$ of $C_m$ and the head $3a$ of $3C_m$.
By symmetry we only have three cases of $(i, j)$ as follows.
\begin{center}
\begin{tabular}{ |c|c|c|c| }
 \hline
 \textbf{case:} & (5, 0) & (4, 1) & (3, 2) \\
 \hline
 \textbf{requirement:} & $a+m=3a$ & $a+m+1=3a$ & $a+m+2=3a$  \\
 $a$\textbf{ value:} & $m\slash 2$ & $(m+1)\slash2$ & $(m+2)\slash2$ \\
 \hline
\end{tabular}
\end{center}
As this time it is easier to express $m$ by $a$, when going through the three cases, we use $a$ as the main parameter for simplicity.

In the $(5, 0)$ case, $m=2a$, $p=8a+3$, $C_m=[a, 3a-1]$, and $\mathbb F_p\backslash\{C_{m}\cup 3C_{m}\}=[a-5, a-1]$. When $\omega= m+1$, we have $C_{m+1}\backslash C_{m}\subset [a-5, a-1]$. Denote $C_{m+1}\backslash C_m=\{a-i\}$ for some $i\in [5]$. Then the $(3, 1)$-sum $(a-i)+2a=a+m-i$ violates Lemma~\ref{lem_contradiction_list_m+1}.
When $\omega=m+2$, from that $3C_{m-1}\cap C_{\omega}=\emptyset$ by Lemma~\ref{lem_41_total} (1),
we have the following statement. Recall that $p=4m+3=8a+3$, and $C_m=[a, 3a-1]$.
\[
3C_{m-1}=
\left\{
\begin{aligned}
 [3a+3, a-6], & ~b_m=a;\\
\{3a\}\cup[3a+2, a-6], & ~b_m=a+1;\\
[3a, a-8]\cup\{a-6\}, & ~b_m=3m-2;\\
[3a, a-9], & ~b_m=3a-1;\\
[3a, a-6], & \text{ otherwise.}
\end{aligned}
\right.
\]
Correspondingly, we have
\[
C_{m+2}\backslash C_m\subset
\left\{
\begin{aligned}
 [a-5, a-1]\cup[3a, 3a+2], & ~b_m=a;\\
 [a-5, a-1]\cup\{3a+1\}, & ~b_m=a+1;\\
 [a-8, a-1], & \text{ otherwise.}
\end{aligned}
\right.
\]
All above possible choices violate Lemma~\ref{lem_contradiction_list} under certain sets of equations. See Table~\ref{tab50311}.
\begin{table}[h!]
\centering
\begin{tabular}{ |c|c|c|c| }
 \hline
 \multirow{2}*{$\omega$} & \multirow{2}*{$C_{\omega}\backslash C_m$} & \multirow{2}*{\textbf{$(3,1)$-sums in $C_\omega$}} & \multirow{2}*{\textbf{Conclusion}}  \\
 &&&\\
 \hline
 \multirow{6}*{$m+2$} &  \multirow{2}*{$\{a-i, a-j\}$, $i\ne j\in [8]$} & $(a-i)+2a=3a-i;$ & \multirow{6}*{Contradiction}\\
 & & $(a-j)+2a=3a-j$ & \\
 \cline{2-3}

 & $\{3a+i, 3a+j\}$ & $2a+(a+i)=3a+i;$ & \\
 & $0\le i<j\le 2$  & $2a+(a+j)=3a+j$  & \\
 \cline{2-3}
 & $\{a-i, 3a+j\}$, & $(a-i)+2a=3a-i;$ & \\
 & $i\in [5]$; $j\in[0, 2]$ & $2a+(a+j)=3a+j$ &\\
 \hline
\end{tabular}
\caption{Scenarios in the $(5, 0)$ case for $(k, \ell, \lambda)=(3, 1, 1)$ and $\omega=m+2$}
\label{tab50311}
\end{table}

In the $(4, 1)$ case, $m=2a-1$, $p=8a-1$, $C_m=[a, 3a-2]$, and $\mathbb F_p\backslash\{C_{m}\cup 3C_{m}\}=[a-4, a-1]\cup\{3a-1\}$. When $\omega=m+1$, only trivial solutions are possible, see Table~\ref{tab41311}. When $\omega=m+2$,
\[
C_{m+2}\backslash C_m\subset
\left\{
\begin{aligned}
 [a-4, a-1]\cup[3a-1, 3a+2], & ~b_m=a;\\
 [a-4, a-1]\cup\{3a-1, 3a+1\}, & ~b_m={a+1};\\
 [a-7, a-1]\cup\{3a-1\}, & \text{ otherwise.}
\end{aligned}
\right.
\]
Most of the situations directly violate Lemma~\ref{lem_contradiction_list}, see Table~\ref{tab41311}, except when $C_{m+2}\backslash C_m=\{3a-1, 3a\}$.
In this situation, as $\omega=m+2$, we know that $2p^{n-1}\le |A_a|+|A_{a+1}|+|A_{3a-1}|+|A_{3a}|$. But from $a+a+a=3a$, we know $3|A_a|\le 3|A_a|+|A_{3a}|\le p^{n-1}+2p^{n-2}$. From $3(3a)=a+1$, we know $|A_{a+1}|+|A_{3a}|\le p^{n-1}+2p^{n-2}$. Hence $2p^{n-1}\le |A_a|+|A_{a+1}|+|A_{3a-1}|+|A_{3a}|\le (|A_{a+1}|+|A_{3a}|)+2|A_a|\le \frac 5 3(p^{n-1}+2p^{n-2})$, which means $p\le 10$, leading to a contradiction.
\begin{table}[h!]
\centering
\begin{tabular}{ |c|c|c|c| }
 \hline
 \multirow{2}*{$\omega$} & \multirow{2}*{$C_{\omega}\backslash C_m$} & \multirow{2}*{\textbf{$(3,1)$-sums in $C_\omega$}} & \multirow{2}*{\textbf{Conclusion}}  \\
 &&&\\
 \hline
 \multirow{2}*{$m+1$} & $a-i$, $i\in [4]$ & $2(a-i)+a=(3a-2i)$ & Contradiction \\
 \cline{2-4}
 & $a+m$ & NONE & Trivial\\
 \hline
 \multirow{9}*{$m+2$} & $\{a-i, a-j\}$, $1\le i< j\le 7$ & $a+(a-i)+(a-j)=(3a-i-j)$ & \multirow{9}*{Contradiction}\\
 \cline{2-3}
 &$\{3a+i, 3a+j\}$, & $2a+(a+j)=3a+j;$ & \\
 & $-1\le i<j\le 2$; $j\ne 0$& $(3a+i)+ 2(3a+j)=(a+i+2j-1)$ & \\
 \cline{2-3}
 & \multirow{2}*{$\{3a-1, 3a\}$} & $a+a+a=3a;$ & \\
 & & $3(3a)=a+1$ &\\
 \cline{2-3}
 & $\{a-i, 3a+j\},$ & \multirow{2}*{$(a-i)+a+(a+i+j)=3a+j$} & \\
 & $i\in [4]$, $j\in [0, 2]$ & &\\
 \cline{2-3}
 & \multirow{2}*{$\{a-i, 3a-1\}$, $i\in [7]$} & $(a-i)+2a= 3a-i;$ &\\
 & & $2(a-i)+a=3a-2i$ &\\
 \hline
\end{tabular}
\caption{Scenarios in the $(4, 1)$ case for $(k, \ell, \lambda)=(3, 1, 1)$}
\label{tab41311}
\end{table}

In the $(3, 2)$ case,  we have $m=2a-2$, $p=8a-5$, $C_m=[a, 3a-3]$, and $\mathbb F_p\backslash\{C_{m}\cup 3C_{m}\}=[a-3, a-1]\cup\{a+m, a+m+1\}$.
When $\omega=m+1$, we have $C_{m+1}\backslash C_m=[a-3, a-1]\cup\{3a-2, 3a-1\}$. All nontrivial situations are normal. See Table~\ref{tab32311} for details.
When $\omega= m+2$,
\[
C_{m+2}\backslash C_m\subset
\left\{
\begin{aligned}
 [a-3, a-1]\cup[ 3a-2, 3a+2], & ~b_m=a;\\
 [a-3, a-1]\cup\{ 3a-2, 3a-1, 3a+1\}, & ~b_m=a+1;\\
 [a-6, a-1]\cup\{ 3a-2, 3a-1\}, & \text{ otherwise.}
\end{aligned}
\right.
\]
The analyses in different scenarios are listed in Table~\ref{tab32311}. Only the last line needs further explanation as follows.

\begin{table}[h!]
\centering
\begin{tabular}{ |c|c|c|c| }
 \hline
 \multirow{2}*{$\omega$} & \multirow{2}*{$C_{\omega}\backslash C_m$} & \multirow{2}*{\textbf{$(3,1)$-sums in $C_\omega$}} & \multirow{2}*{\textbf{Conclusion}}  \\
 &&&\\
 \hline
 \multirow{4}*{$m+1$} & $a-1$ & $3(a-1)=3a-3$ & Type 2\\
 \cline{2-4}
  & $3a-2$ & NONE & Trivial \\
 \cline{2-4}
  & $a-i$, $i\in\{2, 3\}$ & $2(a-i)+a=3a-2i$ & \multirow{2}*{Contradiction}\\
 \cline{2-3}
  & $3a-1$ & $2(3a-1)+(3a-3)=a$ & \\
 \hline
 \multirow{7}*{$m+2$} & $\{3a+i, 3a+j\}$, & $(3a+i)+2(3a+j)=(a+i+2j+5)$; & \multirow{6}*{Contradiction}\\
  & $-2\le i<j\le 2$ & $2(3a+j)+(3a-3)=a+2j+2$ & \\
 \cline{2-3}
 & \multirow{2}*{Contain $a-i$, $i\in [2, 6]$} & $2(a-i)+a=(3a-2i)$; & \\
 & & $2(a-i)+(a+1)=(3a+1-2i)$ & \\
 \cline{2-3}
 & \multirow{2}*{$\{a-1, 3a+j\}$, $j\in [-1, 2]$} & $2(3a+j)+(3a-3)=(a+2j+2)$; & \\
 & & $(a-1)+a+(a+j+1)=(3a+j)$ & \\
 \cline{2-4}
 &$\{a-1, 3a-2\}$ & The special case. See analysis & Type 5\\
 \hline
\end{tabular}
\caption{Scenarios in the $(3, 2)$ case for $(k, \ell, \lambda)=(3, 1, 1)$}
\label{tab32311}
\end{table}



When $\omega=m+2$ and $C_{m+2}\backslash C_m=\{a-1, 3a-2\}$. Under this condition, there are three $(3, 1)$-sums in $C_{m+2}$ as listed,
\begin{equation}\label{eq_10}
3(a-1)=(3a-3),
\end{equation}
\begin{equation}\label{eq_11}
2(a-1)+a=(3a-2),
 \end{equation}and
 \begin{equation}\label{eq_12}
 3(3a-2)=a-1.
\end{equation}
From that $\omega=m+2$ and $|A|\le mp^{n-1}$, \begin{equation}\label{eq_13}
|A_{a-1}|+|A_a|+|A_{3a-3}|+|A_{3a-2}|\ge 2p^{n-1}.
\end{equation}
From (\ref{eq_10}) we know $p^{n-1}\ge |A_{3a-3}|+|3A_{a-1}|\ge |A_{3a-3}|+|A_{a-1}|$;  from (\ref{eq_11}) we know $a=(3a-2)+2(-(a-1))$ and hence $p^{n-1}\ge |A_a|+|A_{3a-2}-2A_{a-1}|\ge |A_{3a-2}|+|A_a|$. Together with (\ref{eq_13}), all the equalities must hold. Specially, $|A_{3a-3}|+|3A_{a-1}|= |A_{3a-3}|+|A_{a-1}|$, which means $|3A_{a-1}|= |A_{a-1}|$.

Set $H=Sym(3A_{a-1})$, which is an additive subgroup of $K$. As $K$ is a linear space over $\mathbb F_p$ of dimension $n-1$, $H$ is also a linear space of dimension at most $n-2$ (otherwise if $H=K$ we have $3A_{a-1}=K$ and $|A_{3a-3}|=0$, contradicting to $\omega=m+2$). Fix a linear space decomposition $K=H\times L'$. From Vosper's theorem, $|A_{a-1}|=|3A_{a-1}|\ge 3|A_{a-1}+H|-2|H|\ge 3|A_{a-1}|-2|H|$, hence $|A_{a-1}|\le |H|\le |3A_{a-1}|=|A_{a-1}|$, which means the equalities hold and hence $|A_{a-1}|=|H|$. As $|A_{a-1}+H|=|H|$, $A_{a-1}$ is contained in one coset of $H$. As $|A_{a-1}|=|H|$, $A_{a-1}$ is exactly a coset of $H$. As $a-1\ne 0$, we can choose a special $L$ such that $A_{a-1}=H$. Hence from (\ref{eq_10}) we know that $A_{3a-3}=K\backslash H$. As from (\ref{eq_11}) $|A_{3a-2}-2A_{a-1}|=|A_{3a-2}|$, by bring in $A_{a-1}=H$, we know $|A_{3a-2}+H|=|A_{3a-2}|$, which means $A_{3a-2}$ is also a union of cosets of $H$. Then from (\ref{eq_11}), $A_a= K\backslash(A_{3a-2}+H)$. So $A_a$ is also a coset union of $H$.

Hence it suffices for us to only consider the structure by projecting on $K\slash H=L'$. We already know $A_{a-1}\slash H=\{0\}$ and $A_{3a-3}\slash H=L'\backslash \{0\}$. Denote $P= A_{3a-2}\slash H$. Hence from (\ref{eq_11}) $A_{a}\slash H = L'\backslash P$. Finally, the four subsets need to satisfy (\ref{eq_12}), which means $3P\cap\{0\}=\emptyset$. These requirements exactly shape type 5.

\textbf{Case (2):}
$C_m$ is covered by a $(3, 1)$-sum-free interval of length $m+1$, say $I=[a, a+m]$. Then $I$ is unique up to isomorphism and without loss of generality we can set $a+m+1=3a$, which means $2a=m+1$. Hence $I=[a,3a-1]$. Regarding to different choices of $I\backslash C_m$,
\[
3C_m=
\left\{
\begin{aligned}
 \{3a\}\cup[3a+2, a-2], & ~I\backslash C_m=\{a+1\};\\
[3a, a-4]\cup\{a-2\}, & ~I\backslash C_m=\{3a-2\};\\
[3a, a-2], & \text{ otherwise.}
\end{aligned}
\right.
\]

If $\omega=m+1$, from Lemma~\ref{lem_41_total} (1) we have $3C_m\cap C_{m+1}=\emptyset$, and hence
\[
C_{m+1}\subseteq
\left\{
\begin{aligned}
 [a-1, 3a-1]\cup\{3a+1\}, & ~I\backslash C_m=\{a+1\};\\
\{a-3\}\cup[a-1, 3a-1], & ~I\backslash C_m=\{3a-2\};\\
[a-1, 3a-1], & \text{ otherwise.}
\end{aligned}
\right.
\]
By the structure of $C_m$, $a$ and $3a-1$ must be in $C_m$. To avoid trivial cases, that is, $C_{m+1}=I$, $C_{m+1}\backslash C_m\in\{a-3, a-1, 3a+1\}$. If $3a+1\in C_{m+1}$, the $(3, 1)$-sum $(3a+1)+2(3a-1)=a$ violates Lemma~\ref{lem_contradiction_list_m+1}; if $a-1\in C_{m+1}$, the $(3, 1)$-sum $(a-1)+2a=(3a-1)$ also violates Lemma~\ref{lem_contradiction_list_m+1}.
If $a-3\in C_{m+1}$, then $I\backslash C_m=\{3a-2\}$ and hence $3a-3\in C_m$, which violates Lemma~\ref{lem_contradiction_list_m+1} by $(a-3)+2a=(3a-3)$.


When $\omega=m+2$, we still have $3C_{m-1}\cap 2C_{m+2}=\emptyset$ from Lemma~\ref{lem_41_total} (1). So we must have $[3a+6, a-8]\subset 3C_{m-1}$ and hence $C_{m+2}\subset [a-7, 3a+5]$. If $I=[a, 3a-1]\subset C_{m+2}$, we can get the contradictions from the corresponding equalities listed in Table~\ref{tab10}.
Only when $C_{m+2}\backslash I=3a$ we do not deduce the contradiction directly from Lemma~\ref{lem_contradiction_list}, but instead apply the following argument.

Denote $x\in I$ satisfies $\{x, 3a\}=C_{m+2}\backslash C_m$. From that $|A|\ge mp^{n-1}$ we must have $|A_a|+|A_x|+|A_{3a-1}|+|A_{3a}|\ge 2p^{n-1}$ and $|A_a|\ge |A_x|$. The first equation $a+a+a=3a$ shows that $3|A_a|< 3|A_a|+|A_{3a}|\le p^{n-1}+3p^{n-2}$; equation $2(3a)+(3a-1)=a$ shows that $|A_a|+|A_{3a-1}|+|A_{3a}|\le p^{n-1}+3p^{n-2}$. However, these three inequalities force that $2p^{n-1}\le |A_a|+|A_x|+|A_{3a-1}|+|A_{3a}|\le (|A_a|+|A_{3a-1}|+|A_{3a}|)+|A_a|< \frac43(p^{n-1}+3p^{n-2})$, which is impossible as $p>6$.
\begin{table}[h!]
\centering
\begin{tabular}{ |c|c|c| }
 \hline
  \multirow{2}*{$C_{m+2}\backslash I$} & \multirow{2}*{\textbf{$(3,1)$-sums in $C_\omega$}} & \multirow{2}*{\textbf{Conclusion}}  \\
 &&\\
 \hline
 \multirow{2}*{$a-i$, $i\in [7]$} & $(a-i)+2a=(3a-i)$; & \multirow{6}*{Contradiction}\\
 & $2(a-i)+a= (3a-2i)$ &\\
 \cline{1-2}
 \multirow{2}*{$3a+j$, $j\in [5]$} & $(3a+j)+2(3a)=(a+j+1)$; & \\
  & $2(3a+j)+(3a)=(a+2j+1)$ & \\
 \cline{1-2}
  \multirow{2}*{$3a$} & $a+a+a=3a$; & \\
  & $2(3a)+(3a-1)=a$ & \\
 \hline
\end{tabular}
\caption{Scenarios when $C_m$ is not an interval, $\omega=m+2$ and $I\subset C_{m+2}$}
\label{tab10}
\end{table}

When $\omega=m+2$ and $I\not\subset C_{m+2}$,  $C_{m+2}\backslash C_{m}\subset [a-7, a-1]\cup[3a, 3a+5]$. We deduce the contradiction in three situations separately.

When $C_{m+2}\backslash C_{m}=\{a-i, a-j\}$ for some $1\le i<j\le 7$, if $(3a-i-j)\in C_m$, then the equation $(a-i)+(a-j)+a=(3a-i-j)$ violates Lemma~\ref{lem_contradiction_list_m+1}. Otherwise, $C_{m}= I\backslash\{3a-i-j\}$. Hence the two equations $(a-i)+2a=(3a-i)$ and $(a-j)+2a= (3a-j)$ violate Lemma~\ref{lem_contradiction_list} (2).

When $C_{m+2}\backslash C_{m}=\{3a+i, 3a+j\}$ for $0\le i<j\le 6$, if $a+i+j\in C_m$, then the equation $(3a-1)+(3a+i)+(3a+j)=(a+i+j)$ violates Lemma~\ref{lem_contradiction_list_m+1}. Otherwise, $C_{m}= I\backslash\{a+i+j\}$. Hence the two equations $2(3a+i)+(3a-1)=(a+2i)$ and $2(3a+j)+(3a-1)=(a+2j)$ violate Lemma~\ref{lem_contradiction_list} (2).

When $C_{m+2}\backslash C_{m}=\{a-i, 3a+j\}$ for $i\in [7]$ and $j\in [0, 5]$, if $a+i+j\in C_m$, then the equation $(a-i)+(a)+(a+i+j)=(3a+j)$ violates Lemma~\ref{lem_contradiction_list_m+1}. Otherwise, $C_{m}= I\backslash\{a+i+j\}$. Hence the two equations $(a-i)+2a=(3a-i)$ and $2a+(a+j)=(3a+j)$ violate Lemma~\ref{lem_contradiction_list} (2) or (3).
\end{proof}

\end{document}